\documentclass[11pt]{amsart}

\usepackage[overload]{textcase}
\usepackage{comment}
\usepackage{hyperref}
\usepackage[capitalise,noabbrev]{cleveref}
\usepackage{amsmath}
\usepackage{amssymb}
\usepackage{enumitem}
\usepackage{thmtools}

\usepackage{subcaption}
\usepackage{tikz}
\usepackage{tkz-graph}
\tikzstyle{VertexStyle} = [shape = circle, draw, fill]
\tikzset{pre/.style={-}}    %hide all loop arrow heads
\tikzstyle{every node}=[circle, inner sep=0pt, minimum width=4pt]
\usetikzlibrary{backgrounds}
\newtheorem{theorem}{Theorem}[section]
\newtheorem{lemma}[theorem]{Lemma}

\newtheorem{sublemma}{}[theorem]
\newtheorem{subsublemma}{}[sublemma]

\theoremstyle{definition}
\newtheorem{definition}{Definition}[section]

\newcommand{\ba}{\backslash}

\DeclareMathOperator{\cl}{cl}
\DeclareMathOperator{\fcl}{fcl}
\DeclareMathOperator{\co}{co}
\DeclareMathOperator{\si}{si}
\newcommand{\cocl}{\cl^*}

\newcommand{\seq}[1]{[#1]}
\newcommand{\lc}{\sqcap}
\newcommand{\unfortunate}{$N$-grounded}

\newcommand{\pspider}{elongated-quad $3$-separator}
\newcommand{\spider}{double-quad $3$-separator}

\newcommand{\twisted}{skew-whiff $3$-separator}
\newcommand{\spikelike}{spike-like $3$-separator}
\newcommand{\Tvamoslike}{Twisted cube-like $3$-separator}
\newcommand{\tvamoslike}{twisted cube-like $3$-separator}
\newcommand{\vamoslike}{V\'amos-like $3$-separator}

\newcommand{\psep}{particular $3$-separator}
\newcommand{\Psep}{Particular $3$-separator}

\newenvironment{slproof}[1][Subproof]{\begin{proof}[#1]}{\end{proof}}
\newenvironment{sslproof}[1][Subproof]{\begin{proof}[#1]}{\end{proof}}

\crefformat{sublemma}{#2\rm{#1}#3}
\crefrangeformat{sublemma}{#3\rm{#1}#4--#5\rm{#2}#6}
\crefmultiformat{sublemma}{#2\rm{#1}#3}{ and~#2\rm{#1}#3}{, #2\rm{#1}#3}{ and~#2\rm{#1}#3}
\crefrangeformat{subsublemma}{#3\rm{#1}#4--#5\rm{#2}#6}
\crefmultiformat{subsublemma}{#2\rm{#1}#3}{ and~#2\rm{#1}#3}{, #2\rm{#1}#3}{ and~#2\rm{#1}#3}
\crefrangeformat{enumi}{#3\rm{#1}#4--#5\rm{#2}#6}
\crefname{enumi}{}{}
\crefname{sublemma}{}{}
\Crefname{sublemma}{Claim}{Claims}

\setenumerate{label=\rm(\roman*),midpenalty=2}

\begin{document}

\title[$N$-detachable pairs II: life in $X$]{$N$-detachable pairs in $3$-connected matroids II: life~in~$X$}

\thanks{The authors were supported by the New Zealand Marsden Fund.}

\author{Nick Brettell \and Geoff Whittle \and Alan Williams}
\address{School of Mathematics and Statistics, Victoria University of Wellington, New Zealand}
\email{nick.brettell@vuw.ac.nz}
\email{geoff.whittle@vuw.ac.nz}
\email{ayedwilliams@gmail.com}

\keywords{$3$-connected, splitter theorem, matroid structure}

\subjclass{05B35}
\date{\today}

\maketitle

\begin{abstract}
Let $M$ be a $3$-connected matroid, and let $N$ be a $3$-connected minor of $M$.
A pair $\{x_1,x_2\} \subseteq E(M)$ is \emph{$N$-detachable} if one of the matroids $M/x_1/x_2$ or $M \backslash x_1 \backslash x_2$ is both $3$-connected and has an $N$-minor.
This is the second in a series of three papers where we describe the structures that arise when it is not possible to find an $N$-detachable pair in $M$.
In the first paper in the series, we showed that, under mild assumptions, either $M$ has an $N$-detachable pair, $M$ has one of three particular $3$-separators that can appear in a matroid with no $N$-detachable pairs, or there is a $3$-separating set $X$ with certain strong structural properties.
In this paper, we analyse matroids with such a structured set $X$, and prove that they have either an $N$-detachable pair, or one of five particular $3$-separators that can appear in a matroid with no $N$-detachable pairs.
\end{abstract}

\pagenumbering{arabic}

\section{Introduction}

Let $M$ be a $3$-connected matroid, and let $N$ be a $3$-connected minor of $M$.
We say that a pair $\{x_1,x_2\} \subseteq E(M)$ is \emph{$N$-detachable} if 
one of the matroids $M/x_1/x_2$ or $M \ba x_1 \ba x_2$ is both $3$-connected and has an isomorphic copy of $N$ as a minor.
This is the second in a series of three papers where we describe the structures that arise when it is not possible to find an $N$-detachable pair in $M$.

Our setup is as follows.
Let %$M$ be a $3$-connected matroid with a $3$-connected matroid $N$ as a minor, where
$|E(N)| \ge 4$.
  %and $M$ has no $N$-detachable pair.
We say that a triangle or triad~$T$ of $M$ is \emph{\unfortunate} if, for all distinct $a,b \in T$, none of $M/a/b$, $M/a\ba b$, $M\ba a/b$, and $M\ba a\ba b$ have an $N$-minor.
In this paper, we assume that every triangle or triad of $M$ is \unfortunate\ (due to \cite[Theorem~3.2]{paper1}).
By Seymour's Splitter Theorem~\cite{pds} and duality, we may assume that there exists some $d \in E(M)$ such that $M \ba d$ is $3$-connected and has an $N$-minor.
Let $d' \in E(M \ba d)$ such that $M \ba d \ba d'$ has an $N$-minor.
If $M \ba d \ba d'$ is $3$-connected, then $\{d,d'\}$ is an $N$-detachable pair.
So suppose $M \ba d \ba d'$ opens up a non-trivial $2$-separation $(Y,Z)$. 
Since $N$ is $3$-connected, any $N$-minor lies primarily on one side of the $2$-separation, so we may assume, up to swapping $Y$ and $Z$, that $|Y \cap E(N)| \le 1$.
For now, we also assume that $|Y| \ge 4$.

In the first paper of the series~\cite[Theorem~7.4]{paper1}, we showed that 
  there is a $3$-separating subset~$X$ of $Y$ with $|X| \ge 4$ such that either for every $x \in X$:
  \begin{enumerate}[label=\rm(\alph*)]
    \item $M \ba d \ba x$ is $3$-connected up to series classes,
    \item $M \ba d / x$ is $3$-connected, and
    \item $M \ba d \ba x$ and $M \ba d / x$ have $N$-minors,
  \end{enumerate}
  or $X \cup \{c,d\}$ is %an \pspider, a \twisted, or a \tvamoslike\ of $M$.
  one of three \psep s that can appear in a matroid with no $N$-detachable pairs for some $c \in \cocl(X \cup d)$.  (We defer the definition of such \psep s to \cref{problem3s}.)

In this paper,
we analyse this structured set~$X$ further,
in the case where $X \cup \{c,d\}$ is not a \psep.
In \cref{seclifetriad}, we consider when the set~$X$ contains a triad; in this case we show that $M$ has an $N$-detachable pair.
In \cref{seclifenontriad}, we consider when $X$ does not contain a triad; in this case, either $M$ has an $N$-detachable pair, or $X \cup d$ %together with elements from $Z$ form
is contained in a \psep\ that can appear in a matroid with no $N$-detachable pairs.
%a \spikelike, an \pspider, a \tvamoslike, or a \spider.
Combining these results, we obtain our main result, \cref{usefulonep2}, in the final section.

Subject to \cref{usefulonep2} and the results in \cite{paper1}, it remains to consider the case when for every $d' \in E(M \ba d)$ such that $M \ba d \ba d'$ has an $N$-minor, the pair $\{d,d'\}$ is contained in a $4$-element cocircuit; 
and to show that when $M$ has a \psep~$P$ and no $N$-detachable pairs, there is at most one element of $M$ that is not in $E(N) \cup P$.
%only elements whose removal retains the $N$-minor are contained in $P$.
We analyse these cases in the third paper in the series.

We denote $\{1,2,\dotsc,n\}$ by $\seq{n}$.

\section{A taxonomy of \psep s}
\label{problem3s}

Let $M$ be a $3$-connected matroid with ground set $E$.
We say that a $4$-element set $Q \subseteq E$ is a \emph{quad} if it is both a circuit and a cocircuit of $M$.

We now define five $3$-separating sets with specific structure, illustrated in \cref{fig-pseps}.
We refer to any one of these as a \psep.

\begin{definition}
\label{def-spike-like2}
Let $P \subseteq E$ be an exactly $3$-separating set of $M$.
If there exists a partition $\{L_1,\dotsc,L_t\}$ of $P$ with $t\geq 3$ such that 
\begin{enumerate}[label=\rm(\alph*)]
  \item $|L_i|=2$ for each $i\in\seq{t}$, and
  \item $L_i\cup L_j$ is a quad for all distinct $i,j\in\seq{t}$, %and
\end{enumerate}
then $P$ is a \emph{\spikelike} of $M$.
\end{definition}

\begin{definition}
\label{def-twisted2}
  Let $P \subseteq E$ be a $6$-element exactly 3-separating set of $M$. If there exists a labelling $\{s_1,s_2,t_1,t_2,u_1,u_2\}$ of $P$ such that
  \begin{enumerate}[label=\rm(\alph*)]
    \item $\{s_1,s_2,t_2,u_1\}$, $\{s_1,t_1,t_2,u_2\}$, and $\{s_2,t_1,u_1,u_2\}$ are the circuits of $M$ contained in $P$; and
    \item $\{s_1,s_2,t_1,t_2\}$, $\{s_1,s_2,u_1,u_2\}$, and $\{t_1,t_2,u_1,u_2\}$ are the cocircuits of $M$ contained in $P$; %and
  \end{enumerate}
  then $P$ is a \emph{\twisted} of $M$.
\end{definition}

\begin{definition}
\label{def-pspider2}
  Let $P \subseteq E$ be a $6$-element exactly $3$-separating set such that $P = Q \cup \{p_1,p_2\}$, and $Q$ is a quad.  If there exists a labelling $\{q_1,q_2,q_3,q_4\}$ of $Q$ such that
  \begin{enumerate}[label=\rm(\alph*)]
    \item $\{p_1,p_2,q_1,q_2\}$, $\{p_1,p_2,q_3,q_4\}$, and $Q$ are the circuits of $M$ contained in $P$, and
    \item $\{p_1,p_2,q_1,q_3\}$, $\{p_1,p_2,q_2,q_4\}$, and $Q$ are the cocircuits of $M$ contained in $P$,
  \end{enumerate}
  then $P$ is an \emph{\pspider} of $M$. 
\end{definition}

\begin{definition}
  \label{def-spider-like2}
  Let $P \subseteq E$ be an exactly $3$-separating set such that $P = Q_1 \cup Q_2$ where $Q_1$ and $Q_2$ are disjoint quads of $M$.
  If there exist labellings $\{p_1,p_2,p_3,p_4\}$ of $Q_1$ and $\{q_1,q_2,q_3,q_4\}$ of $Q_2$ such that
  \begin{enumerate}[label=\rm(\alph*)]
    \item $\{p_1,p_2,q_1,q_2\}, \{p_1,p_2,q_3,q_4\}, \{p_3,p_4,q_1,q_2\}$, $\{p_3,p_4,q_3,q_4\}$, $Q_1$, and $Q_2$ are the circuits of $M$ contained in $P$, and
    \item $\{p_1,p_3,q_1,q_3\}, \{p_1,p_3,q_2,q_4\}, \{p_2,p_4,q_1,q_3\}$, $\{p_2,p_4,q_2,q_4\}$, $Q_1$, and $Q_2$ are the cocircuits of $M$ contained in $P$, 
  \end{enumerate}
  then $P$ is a \emph{\spider} with {\em associated partition} $\{Q_1,Q_2\}$.
\end{definition}

%\begin{definition}
  %If $Q \subseteq E$ is a quad, where $Q$ is exactly $3$-separating, and $Q$ is not contained in a \spikelike, an \pspider, or a \spider, then we say that $Q$ is a \emph{quad $3$-separator} of $M$.
%\end{definition}

These four \psep s are self-dual in the following sense:
if $P$ is a %quad $3$-separator,
\spikelike, \pspider, \spider, or \twisted\ of $M$, then $P$ is also a %quad $3$-separator,
\spikelike, \pspider, \spider, or \twisted\ of $M^*$, respectively.
The same is not true of the next %type of
\psep.

  \begin{definition}
  Let $P \subseteq E$ be an exactly $3$-separating set with $P=\{p_1,p_2,q_1,q_2,s_1,s_2\}$, and let $Y=E-P$.
  Suppose that
  \begin{enumerate}[label=\rm(\alph*)]
    \item $\{p_1,p_2,s_1,s_2\}$, $\{q_1,q_2,s_1,s_2\}$, and $\{p_1,p_2,q_1,q_2\}$ are the circuits of $M$ contained in $P$; and
    \item $\{p_1,q_1,s_1,s_2\}$, $\{p_2,q_2,s_1,s_2\}$, $\{p_1,p_2,q_1,q_2,s_1\}$, and $\{p_1,p_2,q_1,q_2,s_2\}$ are the cocircuits of $M$ contained in $P$.
  \end{enumerate}
  Then $P$ is a \emph{\tvamoslike} of $M$. %with {\em associated partition} $\{P,\{q_1,q_2\}\}$
  \end{definition}

\begin{figure}
  \begin{subfigure}{0.49\textwidth}
    \centering
    \begin{tikzpicture}[rotate=90,scale=0.8,line width=1pt]
      \tikzset{VertexStyle/.append style = {minimum height=5,minimum width=5}}
      \clip (-2.5,-6) rectangle (3.0,2);
      \node at (-1,-1.4) {\large$E-P$};
      \draw (0,0) .. controls (-3,2) and (-3.5,-2) .. (0,-4);
      \draw (0,0) -- (2.5,0.5);
      \draw (0,0) -- (2.25,-0.75);
      \draw (0,0) -- (2,-2);
      \draw (0,0) -- (1,-3);

      \SetVertexNoLabel
      \Vertex[x=1.25,y=0.25,LabelOut=true,L=$q_3$,Lpos=180]{c1}
      \Vertex[x=2.25,y=-0.75,LabelOut=true,L=$q_2$,Lpos=90]{c2}
      \Vertex[x=2.5,y=0.5,LabelOut=true,L=$q_1$,Lpos=180]{c3}
      \Vertex[x=1.5,y=-0.5,LabelOut=true,L=$q_4$,Lpos=135]{c4}

      \Vertex[x=1,y=-1,LabelOut=true,L=$q_1$,Lpos=180]{c5}
      \Vertex[x=2,y=-2,LabelOut=true,L=$q_1$,Lpos=180]{c6}

      \Vertex[x=1,y=-3,LabelOut=true,L=$q_4$,Lpos=135]{c7}
      \Vertex[x=0.67,y=-2,LabelOut=true,L=$q_4$,Lpos=135]{c8}

      \draw (0,0) -- (0,-4);

      \SetVertexNoLabel
      \tikzset{VertexStyle/.append style = {shape=rectangle,fill=white}}
      %\Vertex[x=0,y=0]{a1}
    \end{tikzpicture}
    \caption{An example of a \spikelike.} %\ with $t=4$.}
  \end{subfigure}
  \begin{subfigure}{0.49\textwidth}
    \centering
    \begin{tikzpicture}[rotate=90,scale=0.64,line width=1pt]
      \tikzset{VertexStyle/.append style = {minimum height=5,minimum width=5}}
      \clip (-2.5,-6) rectangle (4.4,2);
      \node at (-1,-1.4) {\large$E-P$};
      \draw (0,0) .. controls (-3,2) and (-3.5,-2) .. (0,-4);
      \draw (0,0) -- (4.0,0.9);
      \draw (0,-2) -- (2.5,-2.2); 
      \draw (0,-4) -- (3.8,-4.9); 

      \Vertex[x=3.0,y=0.67,LabelOut=true,Lpos=180,L=$s_1$]{a2}
      \Vertex[x=2.0,y=0.45,LabelOut=true,Lpos=180,L=$s_2$]{a3}

      \Vertex[x=2.5,y=-2.2,LabelOut=true,Lpos=90,L=$t_2$]{b1}
      \Vertex[x=0.64,y=-2.056,LabelOut=true,Lpos=-45,L=$t_1$]{b2}

      \Vertex[x=3.8,y=-4.9,LabelOut=true,L=$u_1$]{c1}
      \Vertex[x=2.8,y=-4.67,LabelOut=true,L=$u_2$]{c2}

      \draw[dashed] (3.8,-4.9) .. controls (2.0,-2) .. (4.0,0.9);
      \draw[dashed] (2.8,-4.67) .. controls (1.0,-2) .. (3.0,0.67);
      \draw[dashed] (1.8,-4.45) .. controls (0.25,-2) .. (2.0,0.45);

      \draw (0,0) -- (0,-4);

      \SetVertexNoLabel
      \tikzset{VertexStyle/.append style = {shape=rectangle,fill=white}}
      \Vertex[x=4.0,y=0.9]{a1}
      \Vertex[x=1.5,y=-2.12]{b3}
      \Vertex[x=1.8,y=-4.45]{c3}

      %\Vertex[x=0,y=0]{d1}
      %\Vertex[x=0,y=-2]{d2}
      %\Vertex[x=0,y=-4]{d3}
    \end{tikzpicture}
    \caption{A \twisted.}
  \end{subfigure}
  %\begin{subfigure}{0.49\textwidth}
    %\centering
    %\begin{tikzpicture}[rotate=90,xscale=1.1,yscale=0.55,line width=1pt]
      %\tikzset{VertexStyle/.append style = {minimum height=5,minimum width=5}}
      %\clip (-1.5,2) rectangle (2.7,-6);
      %\node at (-0.6,-1.4) {\large$E-P$};
      %\draw (0,0) .. controls (-1.6,2) and (-2,-2) .. (0,-4);

      %\draw (0.8,-1) -- (0,0) -- (1.2,1);
      %\draw (2,-4) -- (1.2,-3);
      %\draw (1.2,1) -- (1.2,-3);
      %\draw (0,-4) -- (1.2,-3);

      %\draw[white,line width=5pt] (0.8,-1) -- (2,0);
      %\draw[white,line width=5pt] (0.8,-1) -- (0.8,-5);
      %\draw (2,0) -- (0.8,-1) -- (0.8,-5);
      %\draw (1.2,-3) -- (0.8,-5);
      %\draw (1.2,1) -- (0.8,-1);
      %\draw (1.2,1) -- (2,0) -- (2,-4);
      %\draw (0,-4) -- (0.8,-5) -- (2,-4);

      %\Vertex[x=1.2,y=1,LabelOut=true,L=$q_1$,Lpos=180]{c1}
      %\Vertex[x=2,y=0,LabelOut=true,L=$s_1$,Lpos=90]{c2}
      %\Vertex[x=2,y=-4,LabelOut=true,L=$s_2$,Lpos=90]{c3}
      %\Vertex[x=1.2,y=-3,LabelOut=true,L=$q_2$,Lpos=135]{c4}
      %\Vertex[x=0.8,y=-1,LabelOut=true,L=$p_1$,Lpos=-45]{c5}
      %\Vertex[x=0.8,y=-5,LabelOut=true,L=$p_2$,Lpos=0]{c6}

      %\draw (0,0) -- (0,-4);

      %%\SetVertexNoLabel
      %%\tikzset{VertexStyle/.append style = {shape=rectangle,fill=white}}
      %%\Vertex[x=0,y=0]{a1}
      %%\Vertex[x=0,y=-4]{a2}
    %\end{tikzpicture}
    %\caption{A \vamoslike\ in $M$.}
  %\end{subfigure}

  \begin{subfigure}{0.49\textwidth}
    \centering
    \begin{tikzpicture}[rotate=90,scale=0.8,line width=1pt]
      \tikzset{VertexStyle/.append style = {minimum height=5,minimum width=5}}
      \clip (-2.5,2) rectangle (3.0,-6);
      \node at (-1,-1.4) {\large$E-P$};
      \draw (0,0) .. controls (-3,2) and (-3.5,-2) .. (0,-4);
      \draw (0,0) -- (2,-2) -- (0,-4);
      \draw (0,0) -- (2.5,0.5) -- (2,-2);
      \draw (0,0) -- (2.25,-0.75);
      \draw (2,-2) -- (1.25,0.25);

      \Vertex[x=1.25,y=0.25,LabelOut=true,L=$q_3$,Lpos=180]{c1}
      \Vertex[x=2.25,y=-0.75,LabelOut=true,L=$q_2$,Lpos=90]{c2}
      \Vertex[x=2.5,y=0.5,LabelOut=true,L=$q_1$,Lpos=180]{c3}
      \Vertex[x=1.5,y=-0.5,LabelOut=true,L=$q_4$,Lpos=135]{c4}
      \Vertex[x=1.33,y=-2.67,LabelOut=true,L=$p_1$,Lpos=45]{c5}
      \Vertex[x=0.67,y=-3.33,LabelOut=true,L=$p_2$,Lpos=45]{c6}

      \draw (0,0) -- (0,-4);

      \SetVertexNoLabel
      \tikzset{VertexStyle/.append style = {shape=rectangle,fill=white}}
      %\Vertex[x=0,y=0]{a1}
      %\Vertex[x=0,y=-4]{a2}
    \end{tikzpicture}
    \caption{An \pspider.}
  \end{subfigure}
  \begin{subfigure}{0.49\textwidth}
    \centering
    \begin{tikzpicture}[rotate=90,scale=0.8,line width=1pt]
      \tikzset{VertexStyle/.append style = {minimum height=5,minimum width=5}}
      \clip (-2.5,-6) rectangle (3.0,2);
      \node at (-1,-1.4) {\large$E-P$};
      \draw (0,0) .. controls (-3,2) and (-3.5,-2) .. (0,-4);

      \draw (0,0) -- (2,-2) -- (0,-4);

      \draw (0,0) -- (2.5,0.5) -- (2,-2);
      \draw (0,0) -- (2.25,-0.75);
      \draw (2,-2) -- (1.25,0.25);

      \draw (0,-4) -- (2.5,-4.5) -- (2,-2);
      \draw (0,-4) -- (2.25,-3.25);
      \draw (2,-2) -- (1.25,-4.25);

      \Vertex[x=1.25,y=0.25,LabelOut=true,L=$q_3$,Lpos=180]{c1}
      \Vertex[x=2.25,y=-0.75,LabelOut=true,L=$q_2$,Lpos=90]{c2}
      \Vertex[x=2.5,y=0.5,LabelOut=true,L=$q_1$,Lpos=180]{c3}
      \Vertex[x=1.5,y=-0.5,LabelOut=true,L=$q_4$,Lpos=135]{c4}

      \Vertex[x=1.25,y=-4.25,LabelOut=true,L=$p_3$]{c1}
      \Vertex[x=2.25,y=-3.25,LabelOut=true,L=$p_2$,Lpos=90]{c2}
      \Vertex[x=2.5,y=-4.5,LabelOut=true,L=$p_1$]{c3}
      \Vertex[x=1.5,y=-3.5,LabelOut=true,L=$p_4$,Lpos=45]{c4}

      \draw (0,0) -- (0,-4);

      \SetVertexNoLabel
      \tikzset{VertexStyle/.append style = {shape=rectangle,fill=white}}
      %\Vertex[x=0,y=0]{a1}
      %\Vertex[x=0,y=-4]{a2}
    \end{tikzpicture}
    \caption{A \spider.}
  \end{subfigure}

  \begin{subfigure}{0.49\textwidth}
    \centering
    \begin{tikzpicture}[rotate=90,xscale=1.1,yscale=0.55,line width=1pt]
      \tikzset{VertexStyle/.append style = {minimum height=5,minimum width=5}}
      \clip (-1.5,2) rectangle (2.7,-6);
      \node at (-0.6,-1.4) {\large$E-P$};
      \draw (0,0) .. controls (-1.6,2) and (-2,-2) .. (0,-4);

      \draw (0.8,-1) -- (0,0) -- (1.2,1) -- (0.8,-1);
      \draw (0,-4) -- (1.2,-3);
      \draw (1.2,1) -- (2.2,-1) -- (1.2,-3);
      \draw (2.2,-1) -- (2.0,-3);
      \draw (1.2,1) -- (1.2,-3);
      \draw (0,0) -- (0,-4);

      \draw[white,line width=5pt] (0.8,-1) -- (2.0,-3) -- (0.8,-5);
      \draw[white,line width=5pt] (0.8,-1) -- (0.8,-5);
      \draw (0.8,-1) -- (2.0,-3) -- (0.8,-5);
      \draw (0.8,-1) -- (0.8,-5);
      \draw (1.2,-3) -- (0.8,-5);
      \draw (0.8,-5) -- (0,-4);

      \Vertex[x=1.2,y=1,LabelOut=true,L=$q_2$,Lpos=180]{c1}
      \Vertex[x=1.2,y=-3,LabelOut=true,L=$p_2$,Lpos=225]{c4}
      \Vertex[x=0.8,y=-1,LabelOut=true,L=$q_1$,Lpos=-45]{c5}
      \Vertex[x=0.8,y=-5,LabelOut=true,L=$p_1$,Lpos=0]{c6}
      \Vertex[x=2.0,y=-3,LabelOut=true,L=$s_1$,Lpos=45]{c5}
      \Vertex[x=2.2,y=-1,LabelOut=true,L=$s_2$,Lpos=45]{c6}
    \end{tikzpicture}
    \caption{A \tvamoslike\ of $M$.}
  \end{subfigure}
  \begin{subfigure}{0.50\textwidth}
    \centering
    \begin{tikzpicture}[rotate=90,scale=0.59,line width=1pt]
      \tikzset{VertexStyle/.append style = {minimum height=5,minimum width=5}}
      \clip (-2.78,-6) rectangle (5.0,2);
      \node at (-1,-1.4) {\large$E-P$};
      \draw (0,0) .. controls (-3,2) and (-3.5,-2) .. (0,-4);
      \draw (0,0) -- (4.0,0.9);
      \draw (0,-2) -- (2.5,-2.2); 
      \draw (0,-4) -- (3.8,-4.9); 

      \Vertex[x=2.0,y=0.45,LabelOut=true,Lpos=180,L=$p_2$]{a3}

      \Vertex[x=2.5,y=-2.2,LabelOut=true,Lpos=90,L=$q_1$]{b1}
      \Vertex[x=0.61,y=-2.05,LabelOut=true,Lpos=35,L=$q_2$]{b2}

      \Vertex[x=3.00,y=-4.7,LabelOut=true,L=$s_1$]{c1}
      \Vertex[x=2.15,y=-4.5,LabelOut=true,L=$s_2$]{c2}

      \Vertex[x=4.0,y=0.9,LabelOut=true,Lpos=180,L=$p_1$]{a1}

      \draw[dashed] (3.8,-4.9) .. controls (2.0,-2) .. (4.0,0.9);
      \draw[dashed] (1.3,-4.3) .. controls (0.25,-2) .. (2.0,0.45);

      \draw (0,0) -- (0,-4);

      \SetVertexNoLabel
      \tikzset{VertexStyle/.append style = {shape=rectangle,fill=white}}
      \Vertex[x=3.8,y=-4.9]{c1}
      \Vertex[x=1.3,y=-4.3]{c3}

      %\Vertex[x=0,y=0]{d1}
      %\Vertex[x=0,y=-2]{d2}
      %\Vertex[x=0,y=-4]{d3}
    \end{tikzpicture}
    \caption{A \tvamoslike\ of $M^*$.}
  \end{subfigure}
  \caption{\Psep s that can appear in a matroid with no $N$-detachable pairs.}
  \label{fig-pseps}
\end{figure}
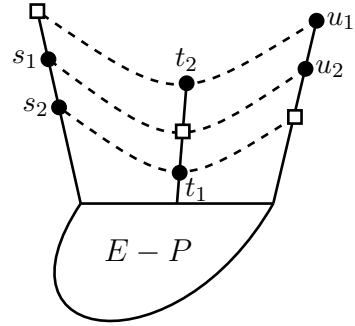
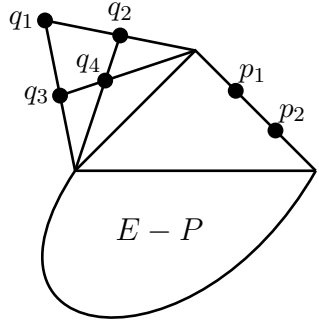
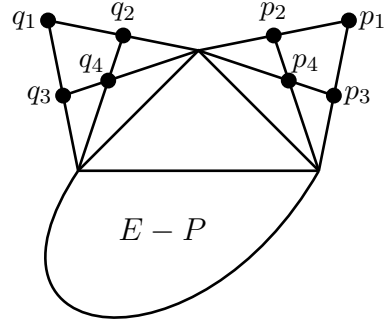
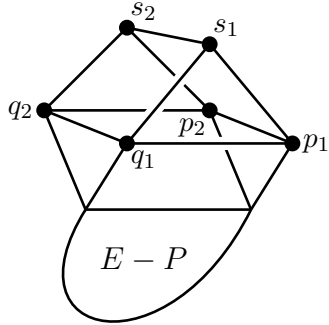
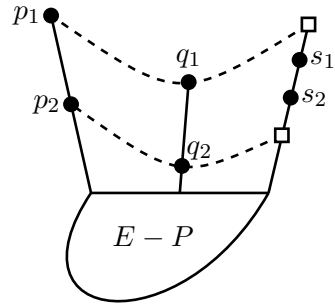

  Each of these five \psep s can appear in a $3$-connected matroid~$M$ with a $3$-connected minor~$N$ such that $E(M)-E(N) \subseteq P$ and $M$ has no $N$-detachable pairs.
  (For a \spikelike, this is shown in \cite[Section~2]{paper1}.
  For an \pspider, a \twisted, or a \tvamoslike, see the discussion in \cite[Section~5]{paper1}; the \spider\ is similar.)
  For all except the \tvamoslike, the intrinsic problem is connectivity; that is, for such a $3$-separator~$P$ in a matroid $M$, there is no pair of elements contained in $P$ for which $M$ remains $3$-connected after deleting or contracting the pair.
  On the other hand, a \tvamoslike~$P$ can appear in a matroid with no $N$-detachable pairs where $P$ contains a pair whose deletion preserves $3$-connectivity (the deletion of the pair destroys the $N$-minor).

\section{Preliminaries}
\label{presec2}

The notation and terminology in the paper follow Oxley~\cite{oxbook}.
For a set~$X$ and element~$e$, we write $X \cup e$ instead of $X \cup \{e\}$, and $X-e$ instead of $X-\{e\}$.
We say that $X$ meets $Y$ if $X \cap Y \neq \emptyset$.

%We write $x \in \clstar(Y)$ to denote that either $x \in \cl(Y)$ or $x \in \cocl(Y)$.
The phrase ``by orthogonality'' refers to the fact that a circuit and a cocircuit cannot intersect in exactly one element.
The following is a straightforward consequence of orthogonality, which is used freely without reference.

\begin{lemma}
\label{swapSepSides2}
Let $e$ be an element of a matroid~$M$, and let $X$ and $Y$ be disjoint sets whose union is $E(M) - e$.  Then $e \in \cl(X)$ if and only if $e \notin \cl^{*}(Y)$.
\end{lemma}

%\subsection*{Connectivity}
Let $M$ be a matroid with ground set $E$.  The \emph{connectivity function} of $M$, denoted by $\lambda_M$, is defined as follows, for a subset $X$ of $E$:
\begin{align*}
  \lambda_M(X) = r(X) + r(E - X) - r(M).
\end{align*}
The following is easily shown to be equivalent:
\begin{align*}
  \lambda_M(X) = r(X) + r^*(X) - |X|.
\end{align*}
A subset $X$ or a partition $(X, E-X)$ of $E$ is \emph{$k$-separating} if $\lambda_M(X) \leq k-1$.
A $k$-separating partition $(X,E-X)$ is a \emph{$k$-separation} if $|X| \ge k$ and $|E-X|\ge k$.
A $k$-separating set $X$, a $k$-separating partition $(X,E-X)$, or a $k$-separation $(X,E-X)$ is \emph{exact} if $\lambda_M(X) = k-1$.
The matroid~$M$ is \emph{$n$-connected} if, for all $k < n$, it has no $k$-separations.
When a matroid is $2$-connected, we simply say it is \emph{connected}.

%The following two lemmas are used frequently in the paper.  The first is well-known (see, for example, \cite[Proposition 2.1.12]{oxbook}) and is a consequence of orthogonality; the second is a consequence of the first.

For subsets $X$ and $Y$ in a matroid $M$, the \emph{local connectivity between $X$ and $Y$}, denoted $\lc(X, Y)$, is defined as follows: \[\lc(X, Y ) = r(X)+r(Y)- r(X \cup Y).\] %Evidently, $\lc(Y,X) = \lc(X, Y)$. Note that if $(X,Y)$ is a partition of $E(M)$, then $\lc(X, Y) = \lambda_M(X)$.

%The following lemma is a consequence of the easily verified fact that the connectivity function is submodular.
We write ``by uncrossing'' to refer to an application of the next lemma.
\begin{lemma}
\label{onetrick2}
Let $M$ be a $3$-connected matroid, and let $X$ and $Y$ be $3$-separating subsets of $E(M)$.
\begin{enumerate}%[label=\rm(\roman*)]
\item If $|X \cap Y| \ge 2$, then $X \cup Y$ is $3$-separating.
\item If $|E(M) - (X \cup Y)| \ge 2$, then $X \cap Y$ is $3$-separating.
\end{enumerate}
\end{lemma}

The following connectivity lemmas are well known and used freely.

\begin{lemma}
\label{extendSep2}
Let $(X,Y)$ be an exactly $3$-separating partition of a $3$-connected matroid, 
and suppose that $e \in Y$.  Then $X \cup e$ is $3$-separating if and only if %$e \in \clstar(X)$.
$e \in \cl(X)$ or $e \in \cocl(X)$.
\end{lemma}

\begin{lemma}
    \label{exactSeps2}
Let $(X, Y)$ be an exactly $3$-separating partition of a $3$-connected matroid. Suppose $|Y| \ge 3$ and $e \in Y$. Then %$x \in \clstar(X-x)$.
$e \in \cl(Y-e)$ or $e \in \cocl(Y-e)$.
\end{lemma}

\begin{lemma}
    \label{gutses2}
Let $(X, Y)$ be an exactly $3$-separating partition of a $3$-connected matroid. Suppose $|Y| \ge 3$ and $e \in Y$. Then
$(X \cup e, Y -e)$ is exactly $3$-separating if and only if $e$ is in %exactly
    one of $\cl(X) \cap \cl(Y-e)$ and $\cocl(X) \cap \cocl(Y-e)$.
\end{lemma}

%If $(X, Y)$ and $(X-x, x \cup Y)$ are exactly $3$-separating partitions in a $3$-connected matroid, then we say $x$ is a \emph{guts element} if $x \in \cl(X-x) \cap \cl(Y)$, and $x$ is a \emph{coguts element} if $x \in \cocl(X-x) \cap \cocl(Y)$.

We also freely use the next three lemmas.
The first is a straightforward consequence of \cref{swapSepSides2,exactSeps2};
the second follows immediately from \cref{swapSepSides2,extendSep2,exactSeps2,gutses2};
and the third is elementary (see \cite[Proposition~8.2.7]{oxbook}).

\begin{lemma}
  \label{gutsstayguts2}
  Let $(X,Y)$ be an exactly $3$-separating partition of a $3$-connected matroid, with $|Y| \ge 3$.
  Then $\cl(X) \cap \cocl(X) \cap Y = \emptyset$.
\end{lemma}

\begin{lemma}
  \label{aggregatelemma}
  Let $(X,Y)$ be an exactly $3$-separating partition of a $3$-connected matroid, with $|Y| \ge 3$.
  If $e \in \cl(X) \cap Y$, then $e \in \cl(Y-e)$ and $(X \cup e, Y-e)$ is exactly $3$-separating.
\end{lemma}

\begin{lemma}
  Let $M$ be a matroid and let $d \in E(M)$.  Suppose that $M\ba d$ is $3$-connected but $M$ is not.  Then either $d$ is in a parallel pair of $M$, or $d$ is a loop or coloop of $M$.
\end{lemma}

The next two lemmas are well known.  We refer to the latter as Bixby's Lemma.

\begin{lemma}
\label{rank2Remove2}
Let $M$ be a $3$-connected matroid and let $S$ be a rank-$2$ subset with at least four elements.  If $s \in S$, then $M \ba s$ is $3$-connected.
\end{lemma}

%We say that $X \subseteq E(M)$ is a \emph{cosegment} if $X$ has corank two.

%The following is known as Bixby's Lemma~\cite{bixby}.
%We refer to the following as Bixby's Lemma.

\begin{lemma}[Bixby's Lemma~\cite{bixby}]
\label{bixbyL2}
Let $e$ be an element of a $3$-connected matroid $M$.
Then either $M/e$ is $3$-connected up to parallel pairs, or $M\ba e$ is $3$-connected up to series pairs.
\end{lemma}

A $k$-separation $(X, E-X)$ of a matroid $M$ with ground set $E$ is \emph{vertical} if $r(X) \ge k$ and $r(E-X) \ge k$.
We also say a partition $(X, \{z\}, Y)$ of $E$ is a \emph{vertical $3$-separation} when $(X \cup \{z\}, Y)$ and $(X, Y \cup \{z\})$ are both vertical $3$-separations and $z \in \cl(X) \cap \cl(Y)$.
Note that, given a vertical $3$-separation $(X,Y)$ and some $z \in Y$, if $z \in \cl(X)$, then $(X,\{z\},Y-z)$ is a vertical $3$-separation, by %\cref{extendSep2,exactSeps2}.
\cref{aggregatelemma}.

A vertical $3$-separation in $M^*$ is known as a cyclic $3$-separation in $M$.
More specifically, a $3$-separation $(X, E-X)$ of $M$ is \emph{cyclic} if $r^*(X) \ge 3$ and $r^*(E-X) \ge 3$; or, equivalently, if $X$ and $E-X$ contain circuits.
We also say that a partition $(X, \{z\}, Y)$ of $E$ is a \emph{cyclic $3$-separation} if $(X, \{z\}, Y)$ is a vertical $3$-separation in $M^*$.

We say that a partition $(X_1,X_2,\dotsc,X_m)$ of $E$ is a \emph{path of $3$-separations} if $(X_1 \cup \dotsm \cup X_i, X_{i+1} \cup \dotsm \cup X_m)$ is a $3$-separation for each $i \in \seq{m-1}$.
Observe that a vertical, or cyclic, $3$-separation $(X, \{z\},Y)$ is an instance of a path of $3$-separations.

The next two lemmas are also used freely.
A proof of the first is in \cite{stabilizers}; %. We use this lemma freely. % without reference.
the second is a straightforward corollary of Bixby's Lemma, \cref{openVertSep2}, and orthogonality. 

\begin{lemma}
  \label{openVertSep2}
  Let $M$ be a $3$-connected matroid and let $z \in E(M)$.  
  The following are equivalent:
  \begin{enumerate}%[label=\rm(\roman*)]
    \item $M$ has a vertical $3$-separation $(X, \{z\}, Y)$.\label{vsi2}
    \item $\si(M/z)$ is not $3$-connected.\label{vsii2}
  \end{enumerate}
\end{lemma}

\begin{lemma}
    \label{orthogVertSep}
    Let $(X,\{z\},Y)$ be a vertical $3$-separation of a $3$-connected matroid~$M$.
    Then either
    \begin{enumerate}%[label=\rm(\roman*)]
        \item $M \ba z$ is $3$-connected, or
        \item $z$ is in a triad that meets $X$ and $Y$.
    \end{enumerate}
\end{lemma}
%\begin{proof}
    %By \cref{openVertSep2}, $\si(M/z)$ is not $3$-connected, so $\co(M \ba z)$ is $3$-connected by Bixby's Lemma.
    %It follows that $M\ba z$ is also $3$-connected if $z$ is not in a triad.  Moreover, such a triad must intersect $X$ and $Y$, by orthogonality.
%\end{proof}

The following is known as Tutte's Triangle Lemma.

\begin{lemma}[Tutte's Triangle Lemma~\cite{tutte1966}]
    \label{ttL2}
    Let $\{a,b,c\}$ be a triangle in a $3$-connected matroid $M$. If neither $M \ba a$ nor $M \ba b$ is $3$-connected, then $M$ has a triad which contains $a$ and exactly one element from $\{b,c\}$.
\end{lemma}

When we refer to an application of Tutte's Triangle Lemma in this paper, the following equivalent formulation is usually more pertinent.
A set $X \subseteq E(M)$ is a \emph{$4$-element fan} if $X$ is the union of a triangle and a triad with $|X| = 4$.
%The following is an equivalent formulation of Tutte's Triangle Lemma. %~\cite{tutte1966}.

\begin{lemma}%[Tutte's Triangle Lemma~\cite{tutte1966}]
  Let $T^*$ be a triad in a $3$-connected matroid $M$.
  If $T^*$ is not contained in a $4$-element fan, then,
  for any pair of distinct elements $a,b \in T^*$, either $M / a$ or $M / b$ is $3$-connected.
\end{lemma}

Proofs of the next two lemmas are in \cite{stabilizers} and \cite{bs2014}, respectively. 

\begin{lemma}%[Whittle 99?]
  \label{r3cocircsi2}
  Let $C^*$ be a rank-$3$ cocircuit of a $3$-connected matroid $M$.
If $x \in C^*$ has the property that $\cl_M(C^*)-x$ contains a triangle of $M/x$, then $\si(M/x)$ is $3$-connected.
\end{lemma}

\begin{lemma}%[Whittle 99?]
    \label{r3cocirc2}
    Let $M$ be a $3$-connected matroid with $r(M) \ge 4$.
    Suppose that $C^*$ is a rank-$3$ cocircuit of $M$.
    If there exists some $x \in C^*$ such that $x \in \cl(C^*-x)$, then $\co(M \ba x)$ is $3$-connected.
\end{lemma}

A set $X$ in a matroid $M$ is {\em fully closed} if it is closed and coclosed; that is, $\cl(X)=X=\cl^*(X)$.
The {\em full closure} of a set $X$, denoted $\fcl(X)$, is the intersection of all fully closed sets that contain $X$.
It is easily seen that the full closure is a well-defined closure operator, and that one way of obtaining the full closure of a set $X$ is to take the closure of $X$, then the coclosure of the result, and repeat until neither the closure nor coclosure introduces new elements.

We use the next lemma frequently.  The straightforward proof is omitted.%

\begin{lemma}
\label{aquickaside1}
Let $(X,Y)$ be a $2$-separation in a connected matroid $M$ where $M$ contains no series or parallel pairs.
Then $(\fcl(X),Y-\fcl(X))$ is also a $2$-separation of $M$.
\end{lemma}

We say that a $2$-separation $(U,V)$ is \emph{trivial} if $U$ or $V$ is a series or parallel class.

We say that $M$ \emph{has an $N$-minor} if $M$ has an isomorphic copy of $N$ as a minor.
For a matroid $M$ with a minor $N$ and $e \in E(M)$, we say $e$ is \emph{$N$-contractible} if $M/e$ has an $N$-minor, we say $e$ is \emph{$N$-deletable} if $M \ba e$ has an $N$-minor, and we say $e$ is \emph{doubly $N$-labelled} if $e$ is both $N$-contractible and $N$-deletable.

The dual of the following is proved in \cite{bs2014,ben}.  

\begin{lemma}
    \label{doublylabelII}
Let $N$ be a $3$-connected minor of a $3$-connected matroid $M$. Let $(X, \{z\}, Y)$ be a cyclic $3$-separation of $M$ such that $M\ba z$ has an $N$-minor with $|X \cap E(N)| \le 1$. Let $X' = X-\cocl(Y)$
and $Y' = \cocl(Y) - z$.
Then
\begin{enumerate}%[label=\rm(\roman*)]
  \item each element of $X'$ is $N$-deletable; and\label{dlIIi}
  \item at most one element of $\cocl(X)-z$ is not $N$-contractible, and if such an element~$x$ exists, then $x \in X' \cap \cl(Y')$ and $z \in \cocl(X' - x)$.\label{dlIIii}
\end{enumerate}
\end{lemma}

%By \cref{presingle}, if $|\cl(Y')-Y'| \neq 1$ or $|\cocl(Y')-Y'| \neq 1$, then every element of $\cl(X)$ is $N$-deletable.
%A careful analysis of case (ii) yields the following:
%\begin{lemma}
    %\label{doublylabel2}
%Let $N$ be a $3$-connected minor of a $3$-connected matroid $M$. Let $(X, \{z\}, Y)$ be a vertical $3$-separation of $M$ such that $M/z$ has an $N$-minor, where $|X \cap E(N)| \le 1$.
%%If there exists some $x' \in X$ for which $x \in \cl(X-x') \cap \cl(Y)$, then every element of $\cl(X)$ is $N$-deletable.
%If there exists some $z' \in X$ for which $z' \in \cl(X-z') \cap \cl(Y)$, then every element of $\cl(X)-z$ is $N$-deletable.
%\end{lemma}

Let $M$ be a matroid with $d \in E(M)$.
Suppose $X \subseteq E(M\ba d)$ is exactly $k$-separating in $M \ba d$.  We say that $d$ \emph{blocks} $X$ if $X$ is not $k$-separating in $M$.
If $d$ blocks $X$, then it follows that $d \notin \cl(E(M\ba d)-X)$, so $d \in \cocl(X)$ by \cref{swapSepSides2}.
We say that $d$ \emph{fully blocks} $X$ if neither $X$ nor $X \cup d$ is $k$-separating in $M$.
It is easily shown that $d$ fully blocks $X$ if and only if $d \notin \cl(X) \cup \cl(E(M \ba d)-X)$.
Usually, when we use this terminology, we are considering elements that block a $3$-separating set $X$; for example, when $X$ is a triad in a $3$-connected matroid.
On the other hand, if $X$ is a series class of $M \ba d$ of size at least two, then we say $d$ blocks $X$ if $X$ is not $2$-separating in $M$ (so $X$ is not a series class in $M$), and $d$ fully blocks $X$ if neither $X$ nor $X \cup d$ is $2$-separating in $M$.

Recall that we typically work under the assumption that every triangle or triad of $M$ is \unfortunate.  In this setting, the following lemma shows that an $N$-contractible (or $N$-deletable) element is not in a triangle (or triad, respectively).

\begin{lemma}[{\cite[Lemma~3.1]{paper1}}]
  \label{freegrounded}
  Let $M$ be a $3$-connected matroid with a $3$-connected minor $N$ where 
$|E(N)| \ge 4$.
If $T$ is an \unfortunate\ triangle of $M$ with $x \in T$, 
then $x$ is not $N$-contractible.
\end{lemma}

\section{The triad case}
\label{seclifetriad}

In this section, we prove the following:

\begin{theorem}
  \label{basilica}
  Let $M$ be a $3$-connected matroid with an element $d$ such that $M\ba d$ is $3$-connected.
  Let $N$ be a $3$-connected minor of $M$, % and $M \ba d$,
  where every triangle or triad of $M$ is \unfortunate, and $|E(N)| \ge 4$.
  Suppose that $M\ba d$ has a cyclic $3$-separation $(Y, \{d'\}, Z)$ with $|Y| \ge 4$, where $M\ba d \ba d'$ has an $N$-minor with $|Y \cap E(N)| \le 1$.
  Suppose $Y$ contains a subset $X$ that is $3$-separating in $M \ba d$, where $|X| \ge 4$ and, for each $x \in X$,
  \begin{enumerate}[label=\rm(\alph*)]
    \item $\co(M\ba d \ba x)$ is $3$-connected,\label{deletionpair}
    \item $M\ba d/x$ is $3$-connected, and\label{mixedpair}
    \item $x$ is doubly $N$-labelled in $M\ba d$.\label{doublylabelled}
  \end{enumerate}
  Let $X$ be minimal subject to these conditions.
  If $X$ contains a triad of $M \ba d$,
  then $M$ has an $N$-detachable pair.
\end{theorem}

\subsection*{Some preparatory lemmas}

Let $M$ be a $3$-connected matroid and let $(P_1,P_2,P_3)$ be a partition of $E(M)$ where $P_i$ is $3$-separating for each $i \in \seq{3}$.
If $\lc(P_i,P_j) = 2$ for all distinct $i,j \in \seq{3}$, then we say $(P_1,P_2,P_3)$ is a \emph{paddle}.
The following is proved in \cite[Lemma~7.2]{osw04}.

\begin{lemma}
  \label{petalsarecoclosed}
  Let $(P_1,P_2,P_3)$ be a paddle in a $3$-connected matroid $M$. 
  Then $\cocl(P_i) = P_i$ for each $i \in \seq{3}$.
\end{lemma}

We first handle the following case that arises in the proof of \cref{cathedral}.

\begin{lemma}
  \label{triad-paddle}
  Let $M$ be a $3$-connected matroid with a $3$-connected matroid~$N$ as a minor. %, where every triangle or triad of $M$ is \unfortunate.
  Suppose that $M \ba d$ is $3$-connected. % and has an $N$-minor.
  Let $(S,T,Z)$ be a paddle in $M\backslash d$ such that %the following hold:
  \begin{enumerate}[label=\rm(\alph*)]
    \item $S$ and $T$ are triads of $M\backslash d$ that are blocked by $d$,
    \item $|Z| \ge 3$, and
    \item for all distinct $s,t \in S \cup T$ such that $\{s,t\} \subseteq \cl((S \cup T) - \{s,t\})$, the matroid $M \ba s \ba t$ has an $N$-minor.\label{tpc}
  \end{enumerate}
  Then $M$ has an $N$-detachable pair.
\end{lemma}

\begin{proof}
Let $M'=M\backslash d$.

\begin{sublemma}
\label{justanother}
For $s\in S$, there is at most one element $t'\in T$ such that
$(S-s) \cup (T-t')$ is a circuit in $M'$.
\end{sublemma}

\begin{slproof}
  Let $T=\{t_1,t_2,t_3\}$ and suppose that $(S -s) \cup (T - t')$ is a circuit for each $t' \in \{t_1,t_2\}$.  Then $t_1,t_2 \in \cl((S-s) \cup t_3)$, so $r((S-s) \cup T)=3$.
  But $r(S \cup T) = 4$, so $s \in \cocl(Z)$, contradicting \cref{petalsarecoclosed}.
\end{slproof}

Let $S=\{s,s_2,s_3\}$ and $T=\{t,t_2,t_3\}$. By~\ref{justanother} we may assume that $\{s_2,s_3,t_2,t_3\}$ is independent.
In particular, $\{s,t\} \subseteq \cl_{M'}(\{s_2,s_3,t_2,t_3\})$.
This implies that $M \ba s \ba t$ has an $N$-minor, by \cref{tpc}.
We work towards proving that $\{s,t\}$ is an $N$-detachable pair in $M$.

\begin{sublemma}
\label{justanother2}
$M'\backslash s \ba t$ is connected.
\end{sublemma}

\begin{slproof}
  Suppose that $(P,Q)$ is a separation of $M'\backslash s \ba t$. As $\{s_2,s_3\}$ and $\{t_2,t_3\}$ are series pairs in $M'\backslash s \ba t$,
we may assume that $\{s_2,s_3\}\subseteq P$ and $\{t_2,t_3\}$ is contained in either $P$ or $Q$.
If $\{t_2,t_3\}\subseteq P$, then
$(P,Q)$ is a separation in the $3$-connected matroid $M'$, as $\{s,t\}\subseteq\cl_{M'}(\{s_2,s_3,t_2,t_3\})$; a contradiction.
Therefore, we may assume that $\{s_2,s_3\}\subseteq P$ and $\{t_2,t_3\}\subseteq Q$.
Moreover, since $r(Z) = r(M')-2$, it follows that $|P\cap Z|,|Q\cap Z| \ge 1$.
Let $\lambda=\lambda_{M'\backslash s \ba t}$.
Since $\lambda(P) = \lambda(Q)=0$, by the submodularity of $\lambda$ we have
\begin{align*}
\lambda(P\cap Z)+\lambda(Q\cap Z)
&\le\lambda(P)+\lambda(Q)+2\lambda(Z)-\lambda(P\cup Z)-\lambda(Q\cup Z)\\
&=2\lambda(Z)-\lambda(P\cup Z)    -\lambda(Q\cup Z)    \\
&=4          -\lambda(\{t_2,t_3\})-\lambda(\{s_2,s_3\})=2.
\end{align*}
If either $\lambda(P\cap Z)=0$ or $\lambda(Q\cap Z)=0$, then, as $\{s,t\}\subseteq\cl_{M'}(\{s_2,s_3,t_2,t_3\})$, the set $P \cap Z$ or $Q \cap Z$ is also $1$-separating in $M'$; a contradiction.
Thus $\lambda(P\cap Z)=\lambda(Q\cap Z)=1$. As $|Z|\geq 3$, we may assume, without loss of generality, that $|P\cap Z|\geq 2$.
But it follows that $(P\cap Z,E(M')-(P\cap Z))$ is a contradictory $2$-separation in $M'$.
\end{slproof}

\begin{sublemma}
\label{justanother9}
If $(P,Q)$ is a $2$-separation of $M'\backslash s \ba t$, then, up to swapping $S$ and $T$, and $P$ and $Q$, we have $\{s_2,s_3\} \subseteq P$ and $\{t_2,t_3\} \subseteq \cocl(Q)$.
\end{sublemma}

\begin{slproof}
  Firstly, observe that if $(P,Q)$ is a $2$-separation of $M'\backslash s \ba t$ where $\{s_2,s_3,t_2,t_3\} \subseteq P$, then, as $\{s,t\}\subseteq\cl_{M'}(\{s_2,s_3,t_2,t_3\})$, the partition $(P\cup \{s,s_2\},Q)$ is a $2$-separation in $M'$; a contradiction.
  Thus we may assume that no $2$-separation $(P,Q)$ of $M'\backslash s \ba t$ has $\{s_2,s_3,t_2,t_3\}$ contained in either $P$ or $Q$.

  Let $(P,Q)$ be a $2$-separation of $M'\backslash s \ba t$.
  As $|Z|\geq 3$, we may assume that $|P\cap Z|\geq 2$.
  Since $\{s_2,s_3,t_2,t_3\} \nsubseteq P$, by possibly swapping $S$ and $T$, we may assume that $|Q \cap T| \ge 1$.
  Suppose that $|Q \cap T| = 1$;
  say $P \cap T = \{t'\}$ and $Q \cap T = \{t''\}$ where $\{t',t''\}=T-t$.
  Since $t'\in\cocl_{M'\backslash s \ba t}(\{t''\})$, the partition $(P-\{t'\},Q\cup\{t'\})$ is a $2$-separation of $M'\backslash s \ba t$.
  Since $\{s_2,s_3,t_2,t_3\} \nsubseteq Q \cup t'$, we deduce that $|Q \cap S| \le 1$.
  Now, similarly,
  if $|Q \cap S| = \{s'\}$, then $(P-\{s',t'\},Q\cup\{s',t'\})$ is a $2$-separation of $M'\backslash s \ba t$.
  But then $\{s_2,s_3,t_2,t_3\} \subseteq Q\cup\{s',t'\}$; a contradiction.
  So $Q \cap S = \emptyset$ and $\{t_2,t_3\} \subseteq \cocl(Q)$ when $|Q \cap T| = 1$.
  A similar argument gives that $Q \cap S = \emptyset$ when $|Q \cap T|=2$.
\end{slproof}

Since $M' \ba s\ba t$ is connected, by \cref{justanother2}, $M \ba s \ba t$ is also connected.
Suppose that $(P,Q)$ is a $2$-separation of $M'\backslash s \ba t$. 
By \ref{justanother9}, we may assume that $\{s_2,s_3\}\subseteq P$ and $\{t_2,t_3\}\subseteq \cocl(Q)$. Thus $(P',Q') = (P-\{t_2,t_3\},Q \cup \{t_2,t_3\})$ is also a $2$-separation of $M'\backslash s \ba t$.
As $T\cup d$ is a $4$-element cocircuit in $M$, we have that $\{t_2,t_3,d\}$ is a triad of $M\backslash s \ba t$.
Hence $d \in \cocl_{M\backslash s \ba t}(Q') = \cocl_{M\backslash s \ba t}(Q)$,
so $d\not\in\cl_{M\backslash s \ba t}(P)$.
Likewise, since $d \in \cocl_{M\ba s \ba t}(\{s_2,s_3\})$, we have $d \notin \cl_{M\ba s \ba t}(Q)$.
We conclude that $(P\cup d,Q)$ and $(P,Q\cup d)$ are $3$-separating in $M\backslash s \ba t$.
That is, $d$ fully blocks $(P,Q)$ for each $2$-separation $(P,Q)$ of $M' \ba s \ba t$.
Thus $M \ba s \ba t$ is $3$-connected.
\end{proof}

\begin{lemma}
\label{aquickaside2}
Let $M$ be a $3$-connected matroid with a pair of disjoint triads $S$=$\{s_1,s_2,s_3\}$ and $T$=$\{t_1,t_2,t_3\}$.
If
\begin{enumerate}
\item $\{s_1,s_2,t_1,t_2\}$ is a circuit of $M$, and
\item $s_3$ is not in a triangle of $M$,
\end{enumerate}
then $M/s_3$ is $3$-connected.
\end{lemma}

\begin{proof}
Note that $\lc(S,T)\geq 1$.
Suppose that $(X,Y)$ is a $2$-separation in $M/s_3$ with $|X\cap T|\geq 2$.
Note that $M/s_3$ contains no series pairs or parallel pairs.
It follows, by \cref{aquickaside1}, that $(\fcl_{M/s_3}(X),Y-\fcl_{M/s_3}(X))$ is also a $2$-separation of $M/s_3$; so we may assume that $X$ is fully closed, and thus $T \subseteq X$.
If $\lc(S,T)=2$, then $\{s_1,s_2\}\subseteq\cl_{M/s_3}(T)\subseteq X$, implying that $(X\cup s_3,Y)$ is a $2$-separation of $M$; a contradiction.
So assume that $\lc(S,T)=1$.
If $\{s_1,s_2\}\subseteq X$ or $\{s_1,s_2\}\subseteq Y$, then $(X\cup s_3,Y)$ or $(X,Y\cup s_3)$, respectively, is a contradictory $2$-separation of $M$.
So, without loss of generality, $s_1\in X$ and $s_2\in Y$.
But then, due to the circuit $\{s_1,s_2,t_1,t_2\}$, we have $s_2 \in \cl(X)-X$, contradicting the fact that $X$ is fully closed. 
\end{proof}

\begin{lemma}
\label{therebetriangles}
Let $M$ be a $3$-connected matroid with distinct elements $a_1,a_2,b_1,b_2,p_1,p_2$, such that
  \begin{enumerate}[label=\rm(\alph*)]
    \item $\co(M\backslash p_1\ba p_2)$ is $3$-connected,
    \item $\{a_1,a_2,p_1,p_2\}$ and $\{b_1,b_2,p_1,p_2\}$ are distinct cocircuits of $M$, and
    \item $\{a_1,a_2\}$ and $\{b_1,b_2\}$ are distinct series classes of $M \ba p_1 \ba p_2$.
  \end{enumerate}
Then either
\begin{enumerate}
  \item there exists $x\in\{a_1,a_2,b_1,b_2\}$ such that $M/x$ is $3$-connected, or\label{tbtout1}
  \item up to labelling, $\{a_1,b_1,p_1\}$ and $\{a_2,b_2,p_2\}$ are triangles of $M$.
\end{enumerate}
\end{lemma}

\begin{proof}
  Assume \cref{tbtout1} does not hold.
Suppose that $a_1$ is not in a triangle and consider $M/a_1$. 
Observe that any series class $S$ of $M \ba p_1 \ba p_2$ with size at least two is blocked by $p_1$ or $p_2$; in particular, if $S \neq \{a_1,a_2\}$, then $p_i \notin \cl_{M/a_1}(E(M/a_1)-(S \cup \{p_1,p_2\}))$ for some $i \in \{1,2\}$.
Since $M / a_1$ is not $3$-connected, but $M \ba p_1 \ba p_2/a_1$ is $3$-connected up to series classes, there is a series class $S'$ of $M\ba p_1\ba p_2/a_1$, with $|S'| \ge 2$, that is not fully blocked by both $p_1$ and $p_2$.
By the foregoing, we may assume that $p_1 \in \cl_{M/a_1}(S')$.
Now $p_1$ is in a circuit of $M$ contained in $S' \cup a_1$.
If $S' \neq \{b_1,b_2\}$, then this contradicts orthogonality with the cocircuit $\{b_1,b_2,p_1,p_2\}$. So $S' = \{b_1,b_2\}$.
Let $\{i,j\} = \{1,2\}$.
Now $p_j \in \cocl_{M/a_1 \ba p_i}(\{b_1,b_2\})$, so $p_j \notin \cl_{M/a_1}(E(M/a_1) - \{b_1,b_2,p_1,p_2\})$, where $\{b_1,b_2\}$ is not fully blocked by $p_j$ in $M \ba p_1 \ba p_2 / a_1$.
Hence $\{p_1,p_2\} \subseteq \cl_M(\{b_1,b_2,a_1\})$, so $r_M(\{b_1,b_2,p_1,p_2\}) \le 3$.
Since $M$ is $3$-connected, $r_M(\{b_1,b_2,p_1,p_2\})=3$ and hence $a_1 \in \cl(\{b_1,b_2,p_1,p_2\})$.

Suppose also that $a_2$ is in a triangle. Since $a_1$ is not, this triangle meets $\{p_1,p_2\}$, by orthogonality with the cocircuit $\{a_1,a_2,p_1,p_2\}$.
Again by orthogonality, either the triangle meets $\{b_1,b_2\}$, or it is $\{a_2,p_1,p_2\}$.
In either case, $r(\{a_1,a_2,b_1,b_2\}) \le 3$.
But since $\co(M \ba p_1 \ba p_2 / a_1 / b_1)$ is $3$-connected, $r(\{a_1,a_2,b_1,b_2\})=4$.
We deduce that $a_2$ is not in a triangle of $M$.

Now repeating the argument in the first paragraph with $a_2$ in the place of $a_1$, we deduce that $a_2\in\cl(\{b_1,b_2,p_1,p_2\})$, so $r(\{a_1,a_2,b_1,b_2\})=3$; a contradiction. Thus $a_1$ and $a_2$ are both in triangles of $M$.

Suppose $\{a_1,a_2,x\}$ is a triangle for some $x \in E(M)-\{a_1,a_2\}$.
If $x \in \{p_1,p_2\}$, then this triangle intersects the cocircuit $\{b_1,b_2,p_1,p_2\}$ in one element; so we may assume otherwise.
But then $M \ba p_1 \ba p_2 / a_1$ contains a parallel pair; a contradiction.
So the triangles containing $a_1$ and $a_2$ are distinct, and each either contains $\{p_1,p_2\}$, or meets both $\{b_1,b_2\}$ and $\{p_1,p_2\}$, by orthogonality.
By symmetry, $b_1$ and $b_2$ are also in triangles of $M$, and each either contains $\{p_1,p_2\}$, or meets both $\{a_1,a_2\}$ and $\{p_1,p_2\}$.  It now follows, by circuit elimination and up to relabelling, that $\{a_1,b_1,p_1\}$ and $\{a_2,b_2,p_2\}$ are triangles of $M$.
\end{proof}

\subsection*{A key lemma}

Next, we work towards proving \cref{contractdistincttriads}, which we use not only in the proof of \cref{basilica}, but also in \cref{seclifenontriad}.

In the remainder of \cref{seclifetriad}, we work under the following assumptions.
Let $M$ be a $3$-connected matroid with an element~$d$ such that $M\ba d$ is $3$-connected.
Let $N$ be a $3$-connected minor of $M$, % and $M \ba d$,
where every triangle or triad of $M$ is \unfortunate, and $|E(N)| \ge 4$.
Suppose that $M\ba d$ has a cyclic $3$-separation $(Y, \{d'\}, Z)$ with $|Y| \ge 4$, where $M\ba d \ba d'$ has an $N$-minor with $|Y \cap E(N)| \le 1$.
Note in particular that $r^*(M \ba d) \ge 4$.

Let $X$ be a subset of $Y$ such that $|X| \ge 4$, the set $X$ is $3$-separating in $M \ba d$, and, for each $x \in X$,
\begin{enumerate}[label=\rm(\alph*)]
  \item $\co(M\ba d \ba x)$ is $3$-connected,
  \item $M\ba d/x$ is $3$-connected, and
  \item $x$ is doubly $N$-labelled in $M\ba d$.
\end{enumerate}

The following is proved in \cite[Lemma~7.1]{paper1}.
A \emph{segment} in a matroid $M$ is a subset $S$ of $E(M)$ such that $M|S \cong U_{2,k}$ for some $k \ge 3$, while a \emph{cosegment} of $M$ is a segment of $M^*$.
  
  \begin{lemma}
    \label{triadsnotcoline}
    If $Y$ contains a $4$-element cosegment, then $M$ has an $N$-detachable pair.
  \end{lemma}
  In particular, \cref{triadsnotcoline} implies that if $M$ has no $N$-detachable pairs, then $X$ does not contain a $4$-element cosegment.

  \begin{lemma}
    \label{triadsinside}
    Either each triad of $M \ba d$ that meets $X$ does so in at least two elements, or $M$ has an $N$-detachable pair.
  \end{lemma}
  \begin{proof}
    Assume that $M$ has no $N$-detachable pairs.
    Suppose $T^*$ is a triad of $M \ba d$ with $T^* \cap X = \{t\}$.
    Then $t \in \cocl(E(M\ba d)-X)$.
    Since $|X| \ge 4$ and $X$ does not contain a $4$-element cosegment, it follows that $(X-t,\{t\},E(M\ba d)-X)$ is a cyclic $3$-separation of $M \ba d$, so $\co(M\ba d \ba t)$ is not $3$-connected; a contradiction.
    So each triad of $M \ba d$ that meets $X$ does so in at least two elements.
  \end{proof}

  The next lemma is used, both in the remainder of this section and in \cref{seclifenontriad}, to find $N$-contractible pairs where each element in the pair is in a triad of $M \ba d$ meeting $X$.

\begin{lemma}
  \label{contractdistincttriads}
  Let $S^*$ and $T^*$ be distinct triads of $M \ba d$ meeting $X$, where $S^* \cup T^*$ is not a cosegment.  Suppose $M$ has no $N$-detachable pairs.
  \begin{enumerate}
    \item If $s \in S^*$ and $t \in T^*$ where $s \neq t$ and $d' \notin \{s,t\}$, then $M \ba d/s/t$ has an $N$-minor.\label{cdt1}
    \item If $d' \in S^*$ and $t \in T^* - S^*$, then $M \ba d/d'/t$ has an $N$-minor.\label{cdt2}
  \end{enumerate}
\end{lemma}

\begin{proof}
  By \cref{triadsinside}, the triads $S^*$ and $T^*$ each have at least two elements in $X$.
  If $t \in Z$, then, since $t \in \cocl_{M \ba d}(Y)$, it follows that $(Y \cup t,\{d'\},Z-t)$ is a cyclic $3$-separation of $M \ba d$.
  Either $|(Y \cup t) \cap E(N)| \le 1$ or $|(Z-t) \cap E(N)| \le 1$, but $|Y \cap E(N)| \le 1$ and $|E(N)| \ge 4$, so $|(Y \cup t) \cap E(N)| \le 1$.
  So we may assume that $t \in Y$.
  In case~\cref{cdt1}, we may similarly assume that $s \in Y$.
  %We may also assume that if $d' \in \{s,t\}$, then $d' = s$.
%
  Moreover, if $s \in T^*$, then $t \notin S^*$, since $S^* \cup T^*$ is not a cosegment.
  So we may assume, up to labels, that $t \notin S^*$, and thus $S^* \subseteq Y-t$.
  %We first assume that $t \neq d'$.
  In case~\cref{cdt2}, let $s = d'$.  We now consider both cases together.
  It suffices to prove that $M \ba d/s/t$ has an $N$-minor. %, where $t \neq d'$.

  We first claim that $t$ is $N$-contractible in $M \ba d$.
  By \cref{doublylabelII}\cref{dlIIii}, at most one element of $\cocl_{M \ba d}(Y)-d'$ is not $N$-contractible.
  Suppose that $t$ is this element that is not $N$-contractible.
  Then $t \in \cl(Z')-Z'$, where $Z' = \cocl_{M \ba d}(Z) - d'$.
  Now $t$ is in a circuit contained in $Z' \cup t$, so the triad $T^*$ meets $Z'$, by orthogonality.
  Let $t_2 \in T^* \cap Z'$.
  If $t_2 \in X$, then $\co(M \ba d \ba t_2)$ is $3$-connected, but $t_2 \in \cocl_{M \ba d}(Z)-Z$, which implies that $\co(M \ba d \ba t_2)$ is not $3$-connected; a contradiction.
  So $t_2 \notin X$, implying that $t \in X$.
  But $M \ba d / t$ is not $3$-connected since $t \in \cl(Z')-Z'$; a contradiction.
  So $t$ is $N$-contractible in $M \ba d$. 

  Next we claim that $M \ba d/t$ is $3$-connected.
  This is immediate if $t \in X$, so we assume that $t \notin X$.
  Let $T^* = \{t,t_2,t_3\}$; then $\{t_2,t_3\} \subseteq X$.
  Since $t \in \cocl_{M \ba d}(X)-X$, the matroid $\co(M \ba d \ba t)$ is not $3$-connected, so $\si(M \ba d/t)$ is $3$-connected by Bixby's Lemma.
  If $t$ is in a triangle of $M \ba d$, then, by orthogonality with $T^*$, the triangle also contains either $t_2$ or $t_3$.  But then either $M \ba d / t_2$ or $M \ba d /t_3$ is not $3$-connected; a contradiction.

  Now $M \ba d/t$ is $3$-connected and has an $N$-minor.
  Note that %if $t \in Y$, then
  $(Y-t, \{d'\}, Z)$ is a cyclic $3$-separation of $M \ba d/t$.
  %If $t \in Z$, then $(Y \cup t, \{d'\}, Z-t)$ is a cyclic $3$-separation of $M \ba d$, so $(Y, \{d'\}, Z-t)$ is a cyclic $3$-separation of $M \ba d/t$.
%
  First, suppose that $s \neq d'$.
  Let $Z' =\cocl(Z) - d'$ and $Y' = (Y-t)-Z'$, so that $(Y',\{d'\},Z')$ is a cyclic $3$-separation of $M \ba d/t$.
  Observe that $\cocl_{M \ba d/t}(Y') = \cocl_{M \ba d/t}(Y-t)$.
  By \cref{doublylabelII}\cref{dlIIii}, at most one element of $\cocl(Y-t)-d'$ is not $N$-contractible in $M \ba d/t$,
  and if $s$ is this exceptional element, then $s \in \cl_{M \ba d/t}(Z')$. %, where $Z' = \cocl(Z) - d'$.
  %Let $Y' = (Y-t) - Z'$, and observe that $\cocl_{M \ba d/t}(Y') = \cocl_{M \ba d/t}(Y-t)$.
  But if $s \in \cl_{M \ba d/t}(Z')$, then 
  $s \notin \cocl_{M \ba d/t}(Y') = \cocl_{M \ba d/t}(Y-t)$, implying that
  $s \notin \cocl_{M \ba d/t}(S^*)$; a contradiction.
  So $M \ba d/t/s$ has an $N$-minor if $s \neq d'$.

  Finally, suppose that $s = d'$.
  Then $d'$ is in a triad $\{d',s_2,s_3\}$ of $M \ba d/t$ where $\{s_2,s_3\} \subseteq X$.
  %We have that $M \ba d \ba d'$ has an $N$-minor, and $\{s_2,s_3\}$ is a series pair in this matroid.
  %So $M \ba d \ba s_2$ has an $N$-minor.
  If $\{s_2,s_3\} \subseteq \cocl_{M \ba d}(Z)$, then
  $S^* \subseteq \cocl_{M \ba d}(Y-S^*) \cap \cocl_{M \ba d}(Z-S^*)$,
  and it follows that $\co(M \ba d \ba s_2)$ is not $3$-connected; a contradiction.
  So we may assume $\{s_2,s_3\} \nsubseteq \cocl_{M \ba d/t}(Z)$, and so, by \cref{doublylabelII}\cref{dlIIi} and up to labels,
  $s_2$ is $N$-deletable in $M \ba d/t$.
  Since $\{s_3,d'\}$ is a series pair in $M \ba d \ba s_2/t$, we deduce that $M \ba d \ba s_2 /t/ d'$ has an $N$-minor.
  %Even, $M \ba d \ba d' / s_2 / t$ has an $N$-minor.
  In particular, %$s=d'$ is $N$-contractible in $M \ba d/t$.
  $M \ba d/s/t$ has an $N$-minor.
\end{proof}

\subsection*{Towards the proof of \cref{basilica}}

%In %the remainder of
%this section, we work under the following assumptions.
%%We fix the following setup for the remainder of this section.
  %Let $M$ be a $3$-connected matroid with an element~$d$ such that $M\ba d$ is $3$-connected.
  %Let $N$ be a $3$-connected minor of $M$, % and $M \ba d$,
  %where every triangle or triad of $M$ is \unfortunate, and $|E(N)| \ge 4$.
  %Suppose that $M\ba d$ has a cyclic $3$-separation $(Y, \{d'\}, Z)$ with $|Y| \ge 4$, where $M\ba d \ba d'$ has an $N$-minor with $|Y \cap E(N)| \le 1$.
  %Note in particular that $r^*(M \ba d) \ge 4$.

  We now assume that $X$ is minimal, and $X$ contains a triad of $M \ba d$.

  More specifically, let $X$ be a subset of $Y$ such that $|X| \ge 4$; the set $X$ is $3$-separating in $M \ba d$; for each $x \in X$,
  \begin{enumerate}[label=\rm(\alph*)]
    \item $\co(M\ba d \ba x)$ is $3$-connected,
    \item $M\ba d/x$ is $3$-connected, and
    \item $x$ is doubly $N$-labelled in $M\ba d$;
  \end{enumerate}
  and $X$ is minimal subject to these conditions.
  Furthermore, $X$ contains a triad of $M \ba d$.

  In practice, the following two lemmas are convenient for finding $N$-contractible or $N$-deletable pairs.

  \begin{lemma}
    \label{triadsncontractiblemain}
    Let $S$ and $T$ be distinct triads of $M \ba d$ that meet $X$, where $S \cup T$ is not a cosegment of $M \ba d$.
    Suppose $M$ has no $N$-detachable pairs.
    \begin{enumerate}
      \item If $s \in S$ and $t \in T$, and either $\{s,t\} \subseteq X$, $\{s,t\} \subseteq X \triangle S$, or $\{s,t\} \subseteq S \triangle T$, then $M \ba d / s / t$ has an $N$-minor.\label{triadsncontractible}
      \item If $s \in S - T$ and $t \in T$, and $M\ba d/s/t$ does not have an $N$-minor, then $M \ba d/s'/t'$ has an $N$-minor for any distinct $s' \in S'$ and $t' \in T$ where $S'$ is a triad of $M \ba d$ that meets $X$ with $S' \neq T$ and $s \neq s'$.\label{triadsncontractible2}
      \item If $S$ and $T$ are disjoint, and $X$ is a corank-$3$ circuit contained in $S \cup T$, with %$|S \cap X| \ge 2$ and
        $T \subseteq X$, then $M \ba d/t/t'$ has an $N$-minor for all distinct $t,t' \in T$.\label{triadsncontractible3}
    \end{enumerate}
  \end{lemma}
  \begin{proof}
    Consider \cref{triadsncontractible}.
    If $\{s,t\} \subseteq X$, then $M \ba d / s /t$ has an $N$-minor by \cref{contractdistincttriads}\cref{cdt1}.
    Suppose $\{s,t\} \subseteq X \triangle S$.  Then $s \in S-X$ and $t \in X-S$.  If $s \neq d'$ then $M \ba d/s/t$ has an $N$-minor by \cref{contractdistincttriads}\cref{cdt1}.  On the other hand, if $s=d'$, then $M \ba d / s /t$ has an $N$-minor by \cref{contractdistincttriads}\cref{cdt2}.
    Finally, if $\{s,t\} \subseteq S \triangle T$, then $M \ba d / s /t$ has an $N$-minor by \cref{contractdistincttriads}\cref{cdt1} when $d' \notin \{s,t\}$, or by \cref{contractdistincttriads}\cref{cdt2} when $d' \in \{s,t\}$.

    Now, for \cref{triadsncontractible2}, suppose $M \ba d/s/t$ does not have an $N$-minor, where $s \notin T$.  Then $d' \in \{s,t\}$ by \cref{contractdistincttriads}\cref{cdt1}.
    If $d' =t$, then, as $s \in S-T$, the matroid $M \ba d/s/t$ has an $N$-minor by \cref{contractdistincttriads}\cref{cdt2}; a contradiction.  So $d'=s$, and thus $d' \notin T$.
    It follows, by \cref{contractdistincttriads}\cref{cdt1}, that $M\ba d/s'/t'$ has an $N$-minor for any $s'\in S'$ and $t' \in T$ with $S' \neq T$ and $s \neq s'$.

    Finally, consider \cref{triadsncontractible3}.
    Let $X' = S \cup T$. 
    Since $X$ is a circuit, $T$ is a triad of $M \ba d$, and $t \in T \subseteq X$, we have that $t \notin \cl(\cocl_{M \ba d}(Z))$.
    Hence, by \cref{doublylabelII}\cref{dlIIii}, $M \ba d / t$ has an $N$-minor.
    Moreover, $(Y-t, \{d'\}, Z)$ is a cyclic $3$-separation in the $3$-connected matroid $M \ba d /t$.
    As $r^*_{M \ba d}(X-t)=3$, we have $t' \in \cocl_{M \ba d}(Y-\{t,t'\})$. 
    Now $X-t$ is a circuit in $M \ba d/t$, so \cref{doublylabelII}\cref{dlIIii} implies that $M \ba d/t/t'$ has an $N$-minor, as required.
    %So (ii) holds.
  \end{proof}

  \begin{lemma}
    \label{triadsdelete}
    Let $x$ and $x'$ be distinct elements in $X$.
    If $x' \in \cl(X - \{x,x'\})$, then $M \ba d \ba x \ba x'$ has an $N$-minor. %, for any $x,x' \in X$.
  \end{lemma}
  \begin{proof}
    If $x \in \cocl_{M \ba d}(Z)$, then $(X-x, \{x\}, E(M \ba d)-X)$ is a cyclic $3$-separation in $M \ba d$, implying $\co(M \ba d \ba x)$ is not $3$-connected; a contradiction.
    So $x \notin \cocl_{M \ba d}(Z)$.
    Now it follows from \cref{doublylabelII}\cref{dlIIi} that $x$ is $N$-deletable in $M \ba d$.
    %Moreover, $(Y-x, \{d'\}, Z)$ is a path of $3$-separations $M \ba d \ba x$ where $d' \in \cocl_{M \ba d \ba x}(Y-x) \cap \cocl_{M \ba d \ba x}(Z)$.
    Let $S$ be a set containing all but one element in each series class of $M \ba d \ba x$, with $x' \notin S$.
    Let $Y'=(Y-x) -S$ and $Z'=Z -S$.
    Now $(Y',\{d'\},Z')$ is a cyclic $3$-separation in $\co(M \ba d \ba x)$.
    %$4$-element cocircuit containing $d$ and $x$ and $x'$?
%
    Since $x' \in \cl(X-\{x,x'\})$, we have that $x' \notin \cocl_{\co(M\ba d\ba x)}(Z')$.
    So $x'$ is $N$-deletable in $\co(M \ba d \ba x)$, by \cref{doublylabelII}\cref{dlIIi}.
    In particular, $M \ba d \ba x \ba x'$ has an $N$-minor, as required.
  %\end{slproof}
  \end{proof}

  \begin{lemma}
    \label{closuresubst}
    If there is an element $w \in \cl_{M \ba d}(X) - X$ that is $N$-deletable in $M \ba d$, then $M$ has an $N$-detachable pair.
  \end{lemma}

  \begin{proof}
    We may assume that $r^*_{M \ba d}(X) \ge 3$, otherwise $M$ has an $N$-detachable pair by \cref{triadsnotcoline}.
    We work towards showing $\{d,w\}$ is an $N$-detachable pair.
    As $X$ and $X \cup w$ are exactly $3$-separating in $M \ba d$, the matroid $\co(M \ba d \ba w)$ is $3$-connected, by \cref{openVertSep2} and Bixby's Lemma.
    So the lemma holds unless $w$ is in a triad of $M \ba d$.
    Suppose $\{x,w,y\}$ is such a triad, and let $W = E(M \ba d) -X$.
    We may assume, by \cref{orthogVertSep}, that $x \in X$ and $y \in W$.
    Thus $x \in \cocl_{M \ba d}(W)$.
    Now $X-x$ and $X$ are exactly $3$-separating in $M \ba d$, so, by \cref{gutses2}, $x \in \cocl_{M \ba d}(X-x)$.
    But then $(X-x,\{x\},W)$ is a cyclic $3$-separation of $M \ba d$, so $\co(M \ba d \ba x)$ is not $3$-connected; a contradiction.
    We deduce that $w$ is not in a triad of $M \ba d$, so $M \ba d \ba w$ is $3$-connected and has an $N$-minor.
  \end{proof}

  \begin{lemma}
    \label{triadsintersecting}
    Suppose $M$ has no $N$-detachable pairs.
    If $S$ and $T$ are triads of $M \ba d$ that meet $X$, the set $S \cup T$ is not a cosegment of $M \ba d$, and $|S \cap T| = 1$, then $r_{M \ba d}(S \triangle T) = 4$ and $S \cup T$ is not $3$-separating in $M \ba d$.%\label{triadsintersecting}
  \end{lemma}
  \begin{proof}
    %If $|S \cap T|=2$, then $S \cup T$ is a $4$-element cosegment in $M \ba d$.
    %By \ref{triadsinside}, $S \cup T \subseteq X$, but this 
    %contradicts \cref{triadsnotcoline}.  So
    Let $S\cap T = \{u\}$.
    We claim that $r_{M\ba d}(S\triangle T)=4$.
    First, suppose that $u \in X$. %, and observe that $u \in X$ by \cref{triadsinside}.
    If $r_{M\ba d}(S\triangle T)=3$, then
    $(S \triangle T, E(M \ba d \ba u)-(S \triangle T))$ is a $2$-separation of $M \ba d \ba u$.  % where neither $S \triangle T$ nor its complement is contained in a series class, contradicting \ref{deletionpair}.
    But %, by \cref{triadsnotcoline}, 
    $S \triangle T$ is not contained in a series class in $M \ba d \ba u$, contradicting that $\co(M\ba d \ba u)$ is $3$-connected. 
    Now suppose that $u \notin X$.  
    If $r_{M\ba d}(S\triangle T)=3$, then $X'=S \triangle T$ is a $4$-element subset of $X$ (by \cref{triadsinside}) that is $3$-separating.
    Since $X$ contains a triad, $X' \subsetneqq X$, contradicting the minimality of $X$.
    Thus $r_{M\ba d}(S\triangle T)=4$.

    Suppose that $\lambda_{M\ba d}(S\cup T)=2$.
    If $S\cup T$ contains a $4$-element circuit of $M \ba d$, then, as $r_{M\ba d}(S\triangle T)=4$, this circuit contains one of $S$ or $T$.
    Without loss of generality we may assume that the circuit is $S \cup t$ for $t \in T-u$.
    Then $(S,\{t\},E(M \ba d)-(S \cup t))$ is a vertical $3$-separation of $M \ba d$, so $\si(M\ba d/t)$ is not $3$-connected.
    This implies that $t \notin X$.
    Let $T = \{u,t,t'\}$; then $t' \in X$, by \cref{triadsinside}.
    Moreover, $(S \cup t, \{t'\}, E(M \ba d)-(S \cup T))$ is a cyclic $3$-separation of $M \ba d$, so $\co(M\ba d \ba t')$ is not $3$-connected; a contradiction.
    So $S \cup T$ does not contain a $4$-element circuit of $M \ba d$.

    By uncrossing, $(S \cup T)\cap X$ is $3$-separating in $M \ba d$.
    If $u \notin X$, then $S \triangle T \subseteq X$ by \cref{triadsinside}, so $S \triangle T$ is $3$-separating, and hence $r_{M \ba d}(S \triangle T) = 3$; a contradiction.
    Now, if neither $S$ nor $T$ is contained in $X$, then $|(S \cup T)\cap X|=3$, so $(S \cup T)\cap X$ is a triangle or a triad.  But this is contradictory, since $S \cup T$ is not a cosegment, and no triangle of $M \ba d$ meets $X$ (since $M \ba d/x$ is $3$-connected for each $x \in X$).
    Next, suppose $S \subseteq X$ and $T=\{t_1,t_2,u\}$ where $T-X = \{t_2\}$.
    Then $(S \cup T) \cap X = S \cup t_1$
    is exactly $3$-separating.  
    As $t_1 \notin \cocl_{M \ba d}(S)$, we have $t_1 \in \cl_{M \ba d}(S)$ by \cref{exactSeps2}, so $S \cup t_1$ is a circuit; a contradiction.
    So $S \cup T \subseteq X$.  Moreover, $S \cup T$ is a circuit, since
    $S \cup T$ does not contain a $4$-element circuit, and $r^*_{M \ba d}(S \cup T) = 3$, so $r_{M \ba d}(S \cup T) = 4$.

    Let $s$ and $t$ be distinct elements such that $s \in S$ and $t \in T$.
    By \cref{triadsncontractiblemain}\cref{triadsncontractible}, the matroid $M \ba d / s / t$ has an $N$-minor. 
    Without loss of generality, we may assume that $s \neq u$.
    Since, in $M \ba d$, the set $S \cup T$ is a corank-$3$ circuit, $S$ is a triad, and $s$ is not in a triangle, it follows from the dual of \cref{r3cocirc2} that $M \ba d / s$ is $3$-connected.
    Moreover, $(S \cup T) -s$ is a corank-$3$ circuit in $M \ba d / s$, and $t \in \cocl_{M \ba d/s}((S \cup T) - \{s,t\})$, so we can apply \cref{r3cocirc2} a second time to deduce that $\si(M \ba d / s / t)$ is $3$-connected.
    Now, either $\{s,t\}$ is an $N$-detachable pair, or $\{s,t\}$ is contained in a $4$-element circuit~$C_{s,t}$ of $M$ that could contain $d$.
    As $S \cup T$ does not contain a $4$-element circuit in $M \ba d$, the circuit $C_{s,t}$ either contains $d$ or meets $E(M\ba d)-(S \cup T)$.
    Suppose that $d \notin C_{s,t}$, and let $w \in C_{s,t} - (S \cup T)$.
    Then, by orthogonality $C_{s,t} = \{s,t,x,w\}$ for $x \in (S \cup T) - \{s,t\}$.
    Since $M\ba d /s/t$ has an $N$-minor and $\{x,w\}$ is a parallel pair in this matroid, $w$ is $N$-deletable in $M \ba d$, contradicting \cref{closuresubst}.
    So for all distinct $s \in S$ and $t \in T$, there is a $4$-element circuit containing $\{s,t,d\}$.

  Let $X' = S \cup T$.
  Suppose that $d$ fully blocks $X'$.
  Since $X'$ is a circuit, there are certainly no $4$-element circuits of $M \ba d$ contained in $X'$.
  Moreover, there are no $4$-element circuits of $M$ contained in $X' \cup d$, otherwise $d \in \cl(X')$, contradicting that $d$ fully blocks $X'$.
  Let $S = \{s_1,s_2,t_3\}$ and $T = \{t_1,t_2,t_3\}$.
  %Let $t_3 = u$.
  For each $i \in \seq{3}$, there are elements $v_i, w_i \in \cl(X' \cup d)-(X' \cup d)$ such that $\{s_1, t_i, d, v_i\}$ and $\{s_2, t_i, d, w_i\}$ are circuits.

  Next, we claim that $\{v_1,v_2,v_3,w_1,w_2,w_3\}$ is a $6$-element rank-$3$ set, and if $\{v_1,v_2,v_3,w_1,w_2,w_3\}$ contains a triangle, then this triangle is either $\{v_i,v_j,w_k\}$ or $\{v_i,w_j,w_k\}$ for some $\{i,j,k\}=\{1,2,3\}$.
If $v_i=v_{i'}$ for distinct $i,i' \in \seq{3}$, then $\{s_1, t_i,t_{i'},d\}$ contains a circuit, by the circuit elimination axiom, contradicting the fact that $d$ fully blocks $X'$.
Similarly, the $w_i$ are pairwise distinct for $i \in \seq{3}$.
Say $v_i = w_j$ for some $i, j \in \seq{3}$. Then, again by circuit elimination, there is a circuit contained in $\{s_1, s_2, t_i, t_j, v_i\}$.
If $v_i \in \cl(\{s_1, s_2, t_i, t_j\})$, then $d \in \cl(X')$; a contradiction.
So $\{s_1, s_2, t_i, t_j\}$ is a circuit of $M$, but this contradicts the fact that $X'$ is a circuit.
Hence the elements $v_1, v_2, v_3, w_1, w_2, w_3$ are pairwise distinct.
Now $\cl(X' \cup d)-(X' \cup d)$ has rank at most 3.
If $r(\{v_1,v_2,v_3\}) \le 2$, then $\{s_1,d,v_1,v_2,v_3\}$ has rank at most four, but spans the rank-$5$ set $X' \cup d$; a contradiction.
A similar argument applies if $r(\{w_1,w_2,w_3\}) \le 2$, or, for some distinct $i,j \in [3]$ either $r(\{v_i,v_j,w_i\}) \le 2$ or $r(\{v_i,w_i,w_j\}) \le 2$.
It now follows from \cite[Lemma~7.2]{paper1} that $M$ has an $N$-detachable pair; a contradiction.

  Now suppose $d$ does not fully block $X'$.
  Then $d \in \cl(X')$, and, for each of the $4$-element circuits containing $\{s,t,d\}$, the fourth element is in $\cl(X')$.
  Let $S = \{s_1,s_2,u\}$ and $T = \{t_1,t_2,u\}$, and let the $4$-element circuits be $\{s_1,t_i,d,x_i\}$, $\{s_2,t_i,d,w_i\}$, $\{u,t_i,d,p_i\}$, and $\{s_i,u,d,q_i\}$, for $i \in \{1,2\}$.
  %Since $d' \in \cocl(Y)$, it follows from \cref{presingle2} that $|\cl(Y)-Y| \le 1$.
  %So at most one of $\{p_1,p_2,q_1,q_2,w_1,w_2,x_1,x_2\}$ is not in $Y$.
  Let $e \in \{p_i,q_i,w_i,x_i\}$ for some $i \in \{1,2\}$.
  Since $d \in \cl(X')$, we have $e \in \cl_{M \ba d}(Y-e)$, so $e \notin \cocl_{M \ba d}(Z)$.
  It follows, by \cref{doublylabelII}\cref{dlIIi}, that $e$ %each such element %of $Y-\cocl(Z)$
  is $N$-deletable in $M \ba d$.

  Suppose there exists some $e \in \{p_i,q_i,w_i,x_i\}-X'$ for $i \in \{1,2\}$.
  Then $(Y-e, \{e\}, Z \cup d')$ is a vertical $3$-separation, so $\co(M \ba d \ba e)$ is $3$-connected by Bixby's Lemma.
  In the case that $\{d,e\}$ is not an $N$-detachable pair,
  $\{d,e\}$ is contained in a $4$-element cocircuit~$C^*$ of $M$.
  Since $e \notin \cocl_{M \ba d}(X')$, the cocircuit contains at most one element of $X'$.
  But since $X'$ is a circuit in $M$, any element $x \in X'$ is not in $\cocl_{M \ba d}(E(M\ba d)-X')$.  So $C^* \subseteq E(M)-X'$, implying $d \notin \cl(X')$; a contradiction.
  So $\{p_i,q_i,w_i,x_i\} \subseteq X'$ for each $i \in \{1,2\}$.

  Now, for each pair of distinct elements $s \in S$ and $t \in T$, the set $\{d,s,t\}$ is contained in a $4$-element circuit that is contained in $X' \cup d$.  Moreover, any two of these circuits intersect in at most two elements, otherwise, by circuit elimination, $X'$ properly contains a circuit; a contradiction.
  %But there are at most two $3$-element subsets of $\seq{5}$ that intersect in at most one element, so .. contradiction
  Suppose $\{d,s,t,u\}$ is a circuit for some labelling $\{s,t,u,v,w\}$ of $X'$ with $s \in S$ and $t \in T$.
  Then, up to relabelling $\{s,t,u\}$, there is a circuit $\{d,u,v,w\}$.
  Up to swapping the labels on $v$ and $w$, there is also a $4$-element circuit containing $\{d,s,v\}$.  But any such circuit intersects either $\{d,s,t,u\}$ or $\{d,u,v,w\}$ in three elements; a contradiction.
\end{proof}

We now prove the main result of this section: \cref{basilica}.
\setcounter{theorem}{0}

\begin{theorem}
  \label{cathedral}
  If $X$ contains a triad of $M \ba d$, then $M$ has an $N$-detachable pair.
\end{theorem}
\begin{proof}
  Let $x \in X$.  Since $\co(M \ba d \ba x)$ is $3$-connected and $M \ba d \ba x$ has an $N$-minor, either $\{d,x\}$ is an $N$-detachable pair or $x$ is in a triad of $M \ba d$.
  So we may assume $x$ is in a triad of $M \ba d$ for every $x \in X$.

  \begin{sublemma}
    \label{triadfans}
    Let $R$ and $S$ be disjoint triads of $M\ba d$ that meet $X$.
    If $M$ has no $N$-detachable pairs, then $\lc(R,S) =1$, and there exists some $r \in R$ such that $S$ is not contained in a $4$-element fan in $M\ba d/r$.
  \end{sublemma}
  \begin{slproof}
    By \cref{triadsnotcoline,triadsinside}, $\lc(R,S) \neq 0$.
    Suppose that $\lc(R,S) = 2$.  Then $(R,S,E(M\ba d)-(R \cup S))$ is a paddle.  If $|R \cap X| = 2$ and $|S \cap X| = 2$, then $X \cup R$ and $X \cup R \cup S$ are $3$-separating, by uncrossing, and it follows that $(X, \{r\}, \{s\}, E(M\ba d)-(X \cup \{r,s\}))$ is a path of $3$-separations where $R-X=\{r\}$ and $S-X = \{s\}$.
    But then $s \in \cocl(E(M\ba d)-(R \cup S))$, contradicting \cref{petalsarecoclosed}.
    So, by \cref{triadsinside}, at least one of $R$ and $S$ is contained in $X$; in fact, by a similar argument, $R \cup S \subseteq X$.
    Now it follows that
    %in which case
    $M$ has an $N$-detachable pair by \cref{triad-paddle,triadsdelete}.
    So $\lc(R,S) =1$.

    Suppose $S$ is contained in a $4$-element fan of $M\ba d/r'$ for some $r' \in R$.
    Since each $s \in S\cap X$ is $N$-contractible in $M \ba d$, %by \cref{doublylabelled},
    it follows from \cref{triadsinside} and orthogonality that $s$ is not contained in an \unfortunate\ triangle.
    Thus, %up to relabelling, and
    there is a $4$-element circuit~$C$ of $M\ba d$ with $r' \in C$ and, by orthogonality, $|C \cap R|=2$ and $|C \cap S| = 2$.
    Let $R - C = \{r\}$.
    If $S$ is also contained in a $4$-element fan of $M\ba d/r$, then there is a $4$-element circuit~$C'$ with $r \in C'$ and $|C' \cap R|=2$ and $|C' \cap S|=2$, implying that $\lc(R,S)=2$; a contradiction.
    Thus, %up to a potential relabelling of the elements of $R$, we may assume that
    the triad $S$ is not contained in a $4$-element fan in $M\ba d/r$ for some $r \in R$.
  \end{slproof}

  We now consider,
  in \cref{triads3disjoint,triadsseparate,triadsatu,triadsdirty,triadsfinal}, %we consider all
  each possible %configuration
  arrangement
  of three distinct triads in $X$, and, in each case, we prove the existence of an $N$-detachable pair.
  These configurations, as they appear in $(M\ba d)^*$, are illustrated in \cref{lifeinxtriadfigs}.
  Note that the intention is to show how these triads interact, and the illustrations are not indicative of the rank of these sets in $(M \ba d)^*$ (in particular, the union may have rank more than three in $(M \ba d)^*$).

  Note also that every triad of $M \ba d$ that meets $X$ is blocked by $d$.
  To see this, let $\Gamma$ be a triad of $M \ba d$ that meets $X$ and is not blocked by $d$.
  Then $\Gamma$ is a triad of $M$, and $\Gamma$ contains an element $x \in X$ that is $N$-deletable.
  But this implies that $\Gamma$ is not $N$-grounded; a contradiction.
  %So each triad of $M \ba d$ that meets $X$ is blocked by $d$.

  \begin{figure}
  \begin{subfigure}{0.45\textwidth}
      \centering
    \begin{tikzpicture}[rotate=90,scale=0.55,line width=1pt]
      %\SetVertexNoLabel
      \draw (0,0) -- (4.0,0.0);
      \draw (0,-2) -- (4.0,-2); 
      \draw (0,-4) -- (4,-4); 

      \Vertex[x=4.0,y=0,LabelOut=true,L=$r_1$]{a1}
      \Vertex[x=2.0,y=0,LabelOut=true,L=$r_2$]{a2}
      \Vertex[x=0.0,y=0,LabelOut=true,L=$r_3$]{a3}

      \Vertex[x=4,y=-2,LabelOut=true,L=$s_1$]{b1}
      \Vertex[x=2,y=-2,LabelOut=true,L=$s_2$]{b3}
      \Vertex[x=0,y=-2,LabelOut=true,L=$s_3$]{b2}

      \Vertex[x=4,y=-4,LabelOut=true,L=$t_1$]{c1}
      \Vertex[x=2,y=-4,LabelOut=true,L=$t_2$]{c2}
      \Vertex[x=0,y=-4,LabelOut=true,L=$t_3$]{c3}
    \end{tikzpicture}
      \caption{\cref{triads3disjoint}}
      \label{triadfigs1}
  \end{subfigure}
  \begin{subfigure}{0.45\textwidth}
      \centering
	\begin{tikzpicture}[rotate=90,scale=0.55,line width=1pt]
	  %\SetVertexNoLabel
	  \draw (4.0,-2.0) -- (0,-3) -- (4.0,-4); 
	  \draw (0,0) -- (4,0); 

      \Vertex[x=4,y=0,LabelOut=true,L=$r_1$,Lpos=180]{c1}
      \Vertex[x=2,y=0,LabelOut=true,L=$r_2$,Lpos=180]{c2}
      \Vertex[x=0,y=0,LabelOut=true,L=$r_3$,Lpos=180]{c3}

      \Vertex[x=4.0,y=-2,LabelOut=true,Lpos=180,L=$s_1$]{a1}
      \Vertex[x=2.0,y=-2.5,LabelOut=true,Lpos=180,L=$s_2$]{a3}

      \Vertex[x=4,y=-4,LabelOut=true,L=$t_1$]{b1}
      \Vertex[x=2,y=-3.5,LabelOut=true,L=$t_2$]{b3}
      \Vertex[x=0,y=-3,LabelOut=true,L=$u$]{b2}
	\end{tikzpicture}
      \caption{\cref{triadsseparate}}
      \label{triadfigs2}
  \end{subfigure}

  \begin{subfigure}{0.3\textwidth}
      \centering
	\begin{tikzpicture}[rotate=90,scale=0.55,line width=1pt]
	  %\SetVertexNoLabel
	  \draw (0,-2) -- (4.0,0.0);
	  \draw (0,-2) -- (4.0,-2); 
	  \draw (0,-2) -- (4,-4); 

      \Vertex[x=4.0,y=0,LabelOut=true,Lpos=180,L=$r_1$]{a1}
      \Vertex[x=2.0,y=-1,LabelOut=true,Lpos=180,L=$r_2$]{a3}

      \Vertex[x=4,y=-2,LabelOut=true,L=$s_1$]{b1}
      \Vertex[x=2,y=-2,LabelOut=true,L=$s_2$,Lpos=45]{b3}
      \Vertex[x=0,y=-2,LabelOut=true,L=$u$]{b2}

      \Vertex[x=4,y=-4,LabelOut=true,L=$t_1$]{c1}
      \Vertex[x=2,y=-3,LabelOut=true,L=$t_2$]{c2}
	\end{tikzpicture}
      \caption{\cref{triadsatu}}
      \label{triadfigs3}
  \end{subfigure}
  \begin{subfigure}{0.3\textwidth}
      \centering
    \begin{tikzpicture}[rotate=90,scale=0.55,line width=1pt]
      %\SetVertexNoLabel
      \draw (4,0) -- (0,0) -- (0,-4) -- (4,-4); 

      \Vertex[x=4.0,y=0,LabelOut=true,L=$r_1$,Lpos=180]{a1}
      \Vertex[x=2.0,y=0,LabelOut=true,L=$r_2$,Lpos=180]{a2}
      \Vertex[x=0.0,y=0,LabelOut=true,L=$r_3$,Lpos=180]{a3}

      \Vertex[x=0,y=-2,LabelOut=true,L=$s$,Lpos=-90]{b2}

      \Vertex[x=4,y=-4,LabelOut=true,L=$t_1$]{c1}
      \Vertex[x=2,y=-4,LabelOut=true,L=$t_2$]{c2}
      \Vertex[x=0,y=-4,LabelOut=true,L=$t_3$]{c3}
    \end{tikzpicture}
      \caption{\cref{triadsdirty}}
      \label{triadfigs4}
  \end{subfigure}
  \begin{subfigure}{0.3\textwidth}
      \centering
    \begin{tikzpicture}[rotate=90,scale=0.55,line width=1pt]
      %\SetVertexNoLabel
      \draw (-4,0) -- (0,-2) -- (-4,-4) -- (-4,0);

      \Vertex[x=-4.0,y=0,LabelOut=true,Lpos=180,L=$s_1$]{a1}
      \Vertex[x=-2.0,y=-1,LabelOut=true,Lpos=180,L=$r$]{a3}

      \Vertex[x=-4,y=-2,LabelOut=true,L=$s_2$,Lpos=-90]{b1}
      \Vertex[x=0,y=-2,LabelOut=true,L=$t_1$]{b2}

      \Vertex[x=-4,y=-4,LabelOut=true,L=$t_3$]{c1}
      \Vertex[x=-2,y=-3,LabelOut=true,L=$t_2$]{c2}
    \end{tikzpicture}
      \caption{\cref{triadsfinal}}
      \label{triadfigs5}
  \end{subfigure}
  \caption{Each configuration of three distinct triads in $X$ as they appear in $(M \ba d)^*$, and the claim in which the configuration is considered.}
  \label{lifeinxtriadfigs}
  \end{figure}
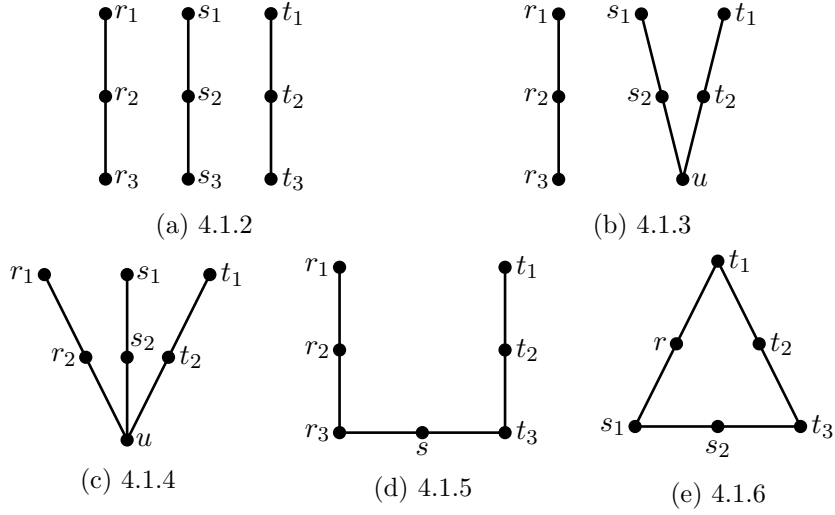

  \begin{sublemma}
    \label{triads3disjoint}
    Let $R$, $S$, and $T$ be distinct triads of $M\ba d$ that meet $X$.
    If these three triads are pairwise disjoint, then $M$ has an $N$-detachable pair.
  \end{sublemma}
  \begin{slproof}
    Suppose $R$, $S$, and $T$ are pairwise disjoint.
    Let $R=\{r_1,r_2,r_3\}$, $S=\{s_1,s_2,s_3\}$, and $T=\{t_1,t_2,t_3\}$ (see \cref{triadfigs1}).
    By \cref{triadfans}, we may assume that the triad~$S$ is not contained in a $4$-element fan in $M\ba d/r_1$, say.
    Now, by Tutte's Triangle Lemma, both $M\ba d/r_1/s_1$ and $M\ba d/r_1/s_2$ are $3$-connected, up to relabelling the elements of $S$.
    By \cref{triadsncontractiblemain}\ref{triadsncontractible}, either $M$ has an $N$-detachable pair, or %$M/r_1/s_1$ and $M/r_1/s_2$ have $N$-minors, so
    there are elements $\alpha$ and $\beta$ such that $\{d,r_1,s_1,\alpha\}$ and $\{d,r_1,s_2,\beta\}$ are circuits of $M$.
    Moreover, $\alpha,\beta\in T$, as otherwise $d$ does not block the triad~$T$ of $M \ba d$.
    So we may assume that $\{d,s_1,t_1\}$ and $\{d,s_2,t_2\}$ are triangles in $M/r_1$, and
    it follows, by circuit elimination, that $\{s_1,s_2,t_1,t_2\}$ contains a circuit of $M/r_1$.
    Since $r_1 \in \cocl(\{r_2,r_3,d\})$, this circuit is also a circuit of $M$.
    As each element of $S \cup T$ is not in a triangle, %an \unfortunate\ triangle,
    $\{s_1,s_2,t_1,t_2\}$ is a circuit of $M/r_1$ and $M$.

    If $\{d,t_3\}$ is also contained in a triangle of $M / r_1$, then the triangle must contain an element $s$ in $\{s_1,s_2,s_3\}$, by orthogonality.
    Thus, by circuit elimination with the triangle $\{d,s_1,t_1\}$, we see that $\{s,s_1,t_1,t_3\}$ contains a circuit of $M / r_1$, and $r_1 \notin \cl(\{s,s_1,t_1,t_3\})$, so it follows that $\{s_1,s,t_1,t_3\}$ is a circuit of $M$.
    But then $S$ is contained in a $4$-element fan in $M \ba d/t$ for each $t \in T$, which, by \cref{triadfans}, implies that $M$ has an $N$-detachable pair.

    Recall that $\{s_1,s_2,t_1,t_2\}$ is a circuit of $M/r_1$, and $S$ is not in a $4$-element fan of $M \ba d / r_1$, so, in particular, $s_3$ is not in a triangle of $M\ba d / r_1$.
    Suppose $s_3$ is in a triangle of $M / r_1$.
    By orthogonality, and the previous paragraph, this triangle is $\{d,s_3,t_i\}$ for some $i \in \{1,2\}$.
    %Then this triangle contains $d$ and an element in $T$, by orthogonality. 
    %But $\{t_3,d\}$ is not contained in a triangle of $M/r_1$, so the triangle is $\{d,s_3,t_i\}$ for some $i \in \{1,2\}$.
    By circuit elimination with the triangle $\{d,s_i,t_i\}$, we deduce that $\{s_3,t_i,s_i\}$ contains a circuit in $M/r_1$; a contradiction.
    So $s_3$ is not in a triangle of $M/r_1$.
    By \cref{aquickaside2}, $M/r_1/s_3$ is $3$-connected, so $M$ has an $N$-detachable pair by \cref{triadsncontractiblemain}\ref{triadsncontractible}.
  \end{slproof}
    
  \begin{sublemma}
    \label{triadsseparate}
    Let $R$, $S$, and $T$ be distinct triads of $M\ba d$ that meet $X$.
    If $|S\cap T|=1$, the set $S \cup T$ is not a cosegment of $M \ba d$, and $R\cap(S\cup T) = \emptyset$, then $M$ has an $N$-detachable pair.
  \end{sublemma}
  \begin{slproof}
    Let $R=\{r_1,r_2,r_3\}$, $S=\{s_1,s_2,u\}$ and $T=\{t_1,t_2,u\}$, where the elements $r_1,r_2,r_3,s_1,s_2,t_1,t_2,u$ are pairwise distinct (see \cref{triadfigs2}).
    Suppose that $M$ has no $N$-detachable pairs.
    By \cref{triadfans}, $S$ is not contained in a $4$-element fan in $M\ba d/r_1$, say.
    By Tutte's Triangle Lemma, there are distinct elements $s,s' \in \{s_1,s_2,u\}$ such that $M\ba d/r_1/s$ and $M\ba d/r_1/s'$ are $3$-connected.
    Since $R$ and $S \cup T$ are disjoint, $M \ba d/r_1/x$ has an $N$-minor for each $x \in S \cup T$, by \cref{triadsncontractiblemain}\cref{triadsncontractible}.

    First, suppose that $M\ba d/r_1/s_1$ and $M\ba d/r_1/s_2$ are $3$-connected.
    Since $M/r_1/s_1$ and $M/r_1/s_2$ have $N$-minors, either $M$ has an $N$-detachable pair, or 
    $M/r_1$ has a triangle containing $\{s_1,d\}$ and a triangle containing $\{s_2,d\}$.
    As $T\cup d$ is a cocircuit of $M/r_1$, each of these triangles meets $T$.
    Suppose these triangles are $\{s_1,d,t\}$ and $\{s_2,d,t'\}$, for $t,t' \in T$.
    Then $\{s_1,s_2,t,t'\}$ contains a circuit $C$ of $M/r_1$, by circuit elimination.
    Since $r_1 \in \cocl_M(\{r_2,r_3,d\})$, the circuit $C$ is also a circuit of $M$, so $C = \{s_1,s_2,t,t'\}$.
    Now $r_{M \ba d}(S \cup T) \le 4$, and $r^*_{M \ba d}(S \cup T) = 3$, so $S \cup T$ is $3$-separating in $M \ba d$, contradicting \cref{triadsintersecting}.

    Now suppose that $M\ba d/r_1/s_1$ and $M\ba d/r_1/u$ are $3$-connected.
    In this case, either $M$ has an $N$-detachable pair, %by \cref{triadsncontractible,triadsncontractible2},
    or $M/r_1$ has triangles containing $\{d,s_1\}$ and $\{d,u\}$.
    By orthogonality, the first of these triangles meets $\{t_1,t_2,u\}$.
    Suppose $\{d,s_1,u\}$ is a triangle of $M/r_1$.
    It follows that $S \cup d$ is $3$-separating, and
    $(\{d,s_1,u\},\{s_2\},E(M/r_1)-(S\cup d))$ is a cyclic $3$-separation of $M/r_1$, implying that $\si(M/r_1/ s_2)$ is $3$-connected.
    Since $s_2$ is not in an \unfortunate\ triangle, and 
    $S$ is not contained in a $4$-element fan in $M\ba d/r_1$,
    either $\{r_1,s_2\}$ is an $N$-detachable pair, %by \cref{triadsncontractiblemain}\cref{triadsncontractible},
    or $\{d,s_2\}$ is contained in a triangle of $M/r_1$ that, by orthogonality, meets $T$.
    If this triangle is $\{d,s_2,u\}$, then $\{s_1,s_2,u\}$ contains a triangle by circuit elimination; a contradiction.  So $\{d,s_2,t\}$ is a triangle of $M/r_1$ for some $t \in \{t_1,t_2\}$.

    Now we may assume, up to swapping $s_1$ and $s_2$, and $t_1$ and $t_2$, that $\{d,s_1,t_1\}$ is a triangle of $M/r_1$.
    If $M/r_1/s_i$ or $M/r_1/t_i$ is $3$-connected for some $i \in \{1,2\}$, then $M$ has a contradictory $N$-detachable pair, % by \cref{triadsncontractiblemain}\cref{triadsncontractible},
    so we may assume otherwise.
    Observe that $\{s_1,s_2\}$ and $\{t_1,t_2\}$ are distinct series classes of $M/r_1\ba d \ba u$, otherwise $X$ contains a $4$-element cosegment of $M \ba d$, contradicting \cref{triadsnotcoline}.
    Applying \cref{therebetriangles} on $M/r_1$, 
    the element $s_2$ is also in a triangle that meets both $\{d,u\}$ and $\{t_1,t_2\}$.
    If this triangle contains $u$, then $S$ is contained in a $4$-element fan of $M/r_1$; a contradiction.  But otherwise we have that $r_{M/r_1}(\{s_1,s_2,u,t_1,t_2\})=4$.
    Since $r_1 \notin \cl(E(M)-\{r_2,r_3,d\})$, the set $\{s_1,s_2,u,t_1,t_2\}$ also has rank four in $M$,
    implying that $S\cup T$ is $3$-separating in $M\ba d$,
    which contradicts \cref{triadsintersecting}.
  \end{slproof}
    
  \begin{sublemma}
    \label{triadsatu}
    Let $R$, $S$, and $T$ be distinct triads of $M\ba d$ that meet $X$, where the union of any two of these triads is not a cosegment of $M \ba d$.
    %If $S\cap T=\{u\}$, and $R\cap(S\cup T) = \{u\}$,
    If $|R \cap S \cap T| = 1$,
    then $M$ has an $N$-detachable pair.
  \end{sublemma}
  \begin{slproof}
    Suppose that $R=\{r_1,r_2,u\}$, $S=\{s_1,s_2,u\}$ and $T=\{t_1,t_2,u\}$
    (see \cref{triadfigs3}).
    Consider $M/t_1\ba d\ba u$.
    Since $\{t_1,t_2\}$ is a series pair in $M\ba d \ba u$, the matroid $\co(M/t_1\ba d\ba u)$ is $3$-connected.
    Observe that $\{r_1,r_2\}$ and $\{s_1,s_2\}$ are distinct series classes of $M/t_1\ba d\ba u$. 
    Now applying \cref{therebetriangles} to $M/t_1$, either $M$ has an $N$-detachable pair by \cref{triadsncontractiblemain}\cref{triadsncontractible}, or $M/t_1$ has triangles $\{r_1,s_1,u\}$ and $\{r_2,s_2,d\}$, up to relabelling.
    So $\{r_1,s_1,t_1,u\}$ and $\{r_2,s_2,t_1,d\}$ are circuits in $M$.
    By \cref{triadsintersecting}, we may assume that $r(\{r_1,r_2,t_1,t_2\}) =4$.  Since $s_1 \in \cocl(\{s_2,u,d\})$, it follows that $r(\{r_1,r_2,s_1,t_1,t_2\}) =5$.
    By symmetry, any $5$-element subset of $\{r_1,r_2,s_1,s_2,t_1,t_2\}$ is independent.

    Consider now $M/t_2$. 
    Applying \cref{therebetriangles}, we see that $M$ contains circuits $\{r_i,s_j,t_2,u\}$ and $\{r_{i'},s_{j'},t_2,d\}$, where $\{i,i'\}=\{1,2\}$ and $\{j,j'\}=\{1,2\}$.
    Suppose $j=1$. Then $\{r_1,s_1,t_2,u\}$ or $\{r_2,s_1,t_2,u\}$ is a circuit.
    In either of these cases $\{t_1,t_2\}\subseteq\cl(\{r_1,r_2,s_1,u\})$, thus $r(\{r_1,r_2,s_1,t_1,t_2\})\leq 4$; a contradiction.
    Similarly, if $i=1$, then $\{t_1,t_2\} \subseteq \cl(\{s_1,s_2,r_1,u\})$; a contradiction.
    So $\{r_2,s_2,t_2,u\}$ and $\{r_1,s_1,t_2,d\}$ are circuits.

    Now consider $M/r_1$.
    Applying \cref{therebetriangles} once more, we arrive at the conclusion that $\{r_1,s_2\}$ is an $N$-detachable pair unless $s_2$ is in a triangle with exactly one element from each of $\{d,u\}$ and $\{t_1,t_2\}$.
    But the existence of this triangle implies that $s_2\in\cl_{M/r_1}(\{s_1,d,u,t_1,t_2\})$, and, due to the circuits $\{r_1,s_1,t_1,u\}$ and $\{r_1,s_1,t_2,d\}$ of $M$, that $r_{M/r_1}(\{s_1,s_2,t_1,t_2\})=3$.
    Hence, by \cref{triadsintersecting}, $M$ has an $N$-detachable pair.
  \end{slproof}

  \begin{sublemma}
    \label{triadsdirty}
    Let $R$, $S$, and $T$ be distinct triads of $M\ba d$ that meet $X$, where the union of any two of these triads is not a cosegment of $M \ba d$.
    If $|S\cap T|=1$, and $|R \cap (S\triangle T)| = 1$, then $M$ has an $N$-detachable pair.
  \end{sublemma}
  \begin{slproof}
    %Suppose \cref{triadsdirty} does not hold.
    We may assume that $T\cap R=\emptyset$ and $|S\cap R|=1$.
    Thus, we have triads $R=\{r_1,r_2,r_3\}$, $T=\{t_1,t_2,t_3\}$ and $S=\{r_3,s,t_3\}$ (see \cref{triadfigs4}).
    We begin by handling the case where $X$ consists of more than these three triads.

    \begin{subsublemma}
      \label{thatsallfolks}
      %If $X \neq R\cup S\cup T$,
      If $X - (R\cup S\cup T) \neq \emptyset$,
      then $M$ has an $N$-detachable pair.
    \end{subsublemma}

    \begin{sslproof}
      Suppose that $q\in X-(R\cup S\cup T)$.
      Recall that every $x \in X$ is in a triad of $M \ba d$, so $q$ is in a triad~$Q$. 
      Suppose $Q$ intersects one of the triads $R$, $S$, or $T$, in two elements.
      Then $R \cup S \cup T \cup q$ is the union of three triads that meet $X$: two triads that intersect in one element, and a third triad disjoint from the other two.
      In this case \cref{triadsseparate} implies that $M$ has an $N$-detachable pair.
      So we may assume that $Q$ intersects each of the triads $R$, $S$ and $T$ in no more than a single element.
      By \cref{triadsnotcoline,triadsinside}, we may also assume that $Q$ is not in a cosegment with $R$, $S$, or $T$, otherwise $M$ has an $N$-detachable pair.

      By \cref{triadsatu}, we have that $\{r_3,t_3\}\cap Q=\emptyset$.
      If $Q\cap(R\cup S\cup T)=\{s\}$, or $Q\cap(R\cup S\cup T)=\emptyset$, then $Q$, $R$ and $T$ are disjoint, so $M$ has an $N$-detachable pair by \cref{triads3disjoint}.
      So we may assume, without loss of generality, that $Q \cap \{t_1,t_2\} \neq \emptyset$.
      Now, if $Q \cap R = \emptyset$, then $R$ is disjoint from $Q \cup T$, so $M$ has an $N$-detachable pair by \cref{triadsseparate}.
      Thus we may assume, without loss of generality, that $Q=\{q,r_2,t_2\}$.
      %Now \cref{triadsinside} implies that $R \cup S \cup T \cup q \subseteq X$.
      If there exists some other $q'\in X-(R\cup S\cup T\cup q)$, with corresponding triad $Q'$, then again we deduce that $Q' = \{q',r,t\}$ for $r \in \{r_1,r_2\}$ and $t \in \{t_1,t_2\}$.
      Suppose $\{r,t\} = \{r_2,t_2\}$.
      Then $\{q,q',t_2\}$ is a triad, so, by applying \cref{triadsseparate} using the triads $R$, $\{q,q',t_2\}$, and $T$, we deduce that $M$ has an $N$-detachable pair.
      So we may assume that $\{r,t\} \neq \{r_2,t_2\}$.
      But now, due to the disjoint triad $S$, we obtain an $N$-detachable pair at the hands of either \cref{triads3disjoint} or \cref{triadsseparate}.
      Therefore $X \subseteq R\cup S\cup T\cup q$.

      \medskip
      Next, we claim that, up to a cyclic shift on the labels given to $R$, $S$, $T$, and $Q$, we may assume that $S$ is not in a $4$-element fan in $M \ba d /q$.
      Observe that $R$ and $T$ are disjoint triads, as are $S$ and $Q$.
      By \cref{triadfans}, we may assume that $\lc(R,T) = \lc(S,Q) = 1$.
      If $S$ is in a $4$-element fan in $M \ba d /q$, then, by orthogonality, there is a $4$-element circuit $C=\{q,q_2,s',s''\}$ in $M \ba d$ for $q_2 \in Q-q$ and $s',s'' \in S$.
      Moreover, $s \in \{s',s''\}$, for otherwise $C$ intersects either $R$ or $T$ in a single element, contradicting orthogonality.
      %If $s \notin \{s',s''\}$, then $Q$ is not in a $4$-element fan in $M \ba d/s$, since $\lc(S,Q)=1$, and so, up to swapping labels on $S$ and $Q$, the claim holds.  So we may assume that $s \in \{s',s''\}$.
      Again by orthogonality, $C$ contains either $\{t_2,t_3\}$ or $\{r_2,r_3\}$.

      Due to symmetry,
      if $R$ is not in a $4$-element fan in $M \ba d / t_1$, then, after a cyclic shift on the labels $R$, $S$, $T$, and $Q$, the claim holds.  So, repeating the argument used on $S$ and $Q$, but this time for $R$ and $T$, we reveal a circuit containing $\{r_1,t_1\}$ and either $\{r_2,t_2\}$ or $\{r_3,t_3\}$.
      Without loss of generality, we may assume that $M \ba d$ has circuits
      $\{s,q,t_2,t_3\}$ and $\{t_1,t_2,r_1,r_2\}$.

      By \cref{triadsintersecting} we may assume that $r_{M \ba d}(Q \triangle T)=4$.
      So $r_{M \ba d}(Q \cup T) \le 5$, and it follows, due to the existence of the circuits $\{s,q,t_2,t_3\}$ and $\{t_1,t_2,r_1,r_2\}$, that $r_{M \ba d}(X'-r_3) \le 5$ where $X' = R \cup S \cup T \cup Q$.
      But since each element of $\{s,t_1,r_1,q\}$ is in a triad of $M \ba d$ where the other elements are in $X'-\{s,t_1,r_1,q\}$, we have $r(E(M\ba d)-X') \le r(M\ba d) - 4$.
      Since $\lambda_{M \ba d}(X)=2$, it follows, by uncrossing, that $\lambda_{M \ba d}(X')=2$, and we deduce that $(X'-r_3,\{r_3\},E(M\ba d)-X')$ is a cyclic $3$-separation of $M \ba d$.
      Thus $\co(M \ba d \ba r_3)$ is not $3$-connected, implying that $r_3 \notin X$.
      Now, for every $x \in X'-r_3$, we have $x \in \cl(X'-\{r_3,x\}) \cap \cocl_{M \ba d}(X'-\{r_3,x\})$, so $X'-\{r_3,x\}$ is not $3$-separating.
      As $X$ is $3$-separating, $|X| < |X'-r_3|-1 = 6$.
      By \cref{triadsinside}, $X = \{r_1,r_2,s,t_3,t_i\}$ for some $i \in \{1,2\}$.  But then $X$ does not contain a triad; a contradiction.
      %By \cref{triadsinside}, we may assume that $X = \{r_1,r_2,s,t_3,t_2\}$.
      %Now $r^*_{M \ba d}(X) =r^*_{M \ba d}(X') = 4$, and $r_{M \ba d}(X) =r_{M \ba d}(X'-r_3) = 5$.
      %Since $|X|=5$, we have that $\lambda_{M \ba d}(X) = 4$; a contradiction.
%
      This proves the claim, % that, up to a cyclic shift on the labels given to $R$, $S$, $T$, and $Q$, we may
      so henceforth we
      assume that $S$ is not in a $4$-element fan in $M \ba d /q$.

      \medskip

      Now we are in a position where we can apply Tutte's Triangle Lemma on the $3$-connected matroid $M \ba d/q$.
      We have two possible scenarios: either $M\ba d/q/r_3$ and $M\ba d/q/t_3$ are $3$-connected, or $M\ba d/q/s$ is $3$-connected.
      Assume that the first of these possibilities holds.
      Then $M/q$ contains triangles $\{d,r',r_3\}$ and $\{d,t_3,t'\}$, say.

      Suppose that the triangles $\{d,r',r_3\}$ and $\{d,t_3,t'\}$ coincide; that is, $r'=t_3$ and $t'=r_3$.
      Now $\{d,r_3,t_3\}$ is a triangle and a triad of $M/q \ba s$, so it is $2$-separating and, by Bixby's Lemma, $M/q/s$ is $3$-connected up to parallel pairs.
      However, $s$ is not in a triangle of $M/q$ and thus $M/q/s$ is $3$-connected.
      Hence, by \cref{triadsncontractiblemain}\cref{triadsncontractible}, $M$ has an $N$-detachable pair.
      Therefore we may assume that $\{d,r',r_3\}$ and $\{d,t_3,t'\}$ are distinct. 

      Now, by the circuit elimination axiom, there is a circuit of $M/q$ contained in $\{r_3,t_3,r',t'\}$.
      By orthogonality with $R$ and $T$, we deduce that $\{r_3,t_3,r',t'\}$ is a circuit of $M \ba d/q$ where $r' \in \{r_1,r_2\}$ and $t' \in \{t_1,t_2\}$.

      We now switch to the dual: let $M' = M^*/d$ and consider $M' \ba q \ba s$.
      Since the triangle $S$ is not in a $4$-element fan in $M'\ba q$, the matroid $M'\ba q\ba s$ does not have any series pairs.
      Suppose that $(A,B)$ is a $2$-separation of $M'\ba q\ba s$.
      We may assume that $r_3\in A$ and $t_3\in B$, for otherwise we would have $s\in\cl_{M'\ba q}(A)$ or $s\in\cl_{M'\ba q}(B)$, which would imply that $M'\ba q$ has a $2$-separation; a contradiction.
      So $|R \cap A| \ge 2$ and $|T \cap B| \ge 2$, for otherwise $r_3 \in \cl_{M'\ba q}(B)$ or $t_3 \in \cl_{M'\ba q}(A)$ in which case again $M'\ba q$ has a $2$-separation.
      By \cref{aquickaside1}, %Since $M'\ba q$ has no series pairs or parallel pairs,
      $(\fcl(A),B-\fcl(A))$ and $(A-\fcl(B),\fcl(B))$ are also $2$-separations.
      Hence we may assume that $R\subseteq A$ and $T\subseteq B$.

      Let $X' = X \cup (R \cup S \cup T)$, and observe that $X'$ is $3$-separating in $M'$ by uncrossing.
      Let $W=E(M'\ba q\ba s)-X'$.
      As $|W|\geq 3$, we may assume that $|A\cap W|\geq 2$.
      Denote $\lambda_{M'\ba q\ba s}$ by $\lambda$.
      By the submodularity of $\lambda$, we have
      \begin{align*}
        \lambda(A\cap W)+\lambda(A\cup W) 
        \leq \lambda(A)+\lambda(W) %\leq 1+2 =
        = 3.
      \end{align*}
      So either $A\cap W$ or $A\cup W$ is $2$-separating in $M'\ba q\ba s$.
      The first possibility implies that $(X'-\{q,s\})\cup B$ is $2$-separating, but since $s\in\cl_{M'\ba q}(X'-\{q,s\})$, this implies that %$M'\ba q$ has a contradictory $2$-separation, by virtue of the fact that $|A\cap W|\geq 2$.
      $A \cap W$ is a $2$-separation in $M'\ba q$; a contradiction.
      So $A\cup W$ is $2$-separating in $M'\ba q\ba s$.

      Now $B \cap X'$ is a triangle in $M'\ba q \ba s$, implying that $r_{M'}(A\cup W) = r(M')-1$.
      Note that since $X'-q$ is $3$-separating in the $3$-connected matroid $M'\ba q$, and $s \in \cl_{M'}(X'-\{q,s\})$, we have that $X'-\{q,s\}$ is exactly $3$-separating in $M'\ba q\ba s$, otherwise $M'\ba q$ is not $3$-connected.
      %Note that, in $M'\ba q\ba s$, the $3$-separating set $X'-\{q,s\}$ is exactly $3$-separating, otherwise $M'\ba q$ is not $3$-connected.
      If $X'-\{q,s\}$ has rank three in $M'$, then $r_{M'}(R \cup S) =3$, contradicting \cref{triadsintersecting}.
      So $X'-\{q,s\}$ has rank four, and $r_{M'}(W) = r(M')-2$.
      Pick $r$ so that $\{r,r'\} = \{r_1,r_2\}$.  
      If $r \notin X$, then $r \in \cocl_{M'}(X)$, so $r \notin \cl_{M'}(W)$.
      On the other hand, if $r \in X$, then
      $r\not\in\cl_{M'}(W)$, for otherwise $\si(M'/r)$ fails to be $3$-connected; a contradiction.
      So $r_{M'}(W\cup r) = r(M')-1$.
      As $r_3$ is in the cocircuit $\{r',t',r_3,t_3\}$, % of $M'\ba q$, and this cocircuit is disjoint from $W\cup r$,
      it then follows that $r_{M'}(W\cup\{r,r_3\})=r(M')$.
      But this is contradictory, since %$B\cap X$ is a triangle in $M'\ba q$, implying that
      $r_{M'}(A\cup W) = r(M')-1$.
      We are left to conclude that $M'\ba q\ba s$ is $3$-connected.

      Returning to the application of Tutte's Triangle Lemma,
      we now have that $M\ba d/ q/ s$ is $3$-connected.
      So either $M/ q/ s$ is $3$-connected, in which case $M$ has an $N$-detachable pair by \cref{triadsncontractiblemain}\cref{triadsncontractible}, or $\{d,s,r'\}$ is a triangle in $M/q$.
      But since $\{d,r_1,r_2,r_3\}$ and $\{d,t_1,t_2,t_3\}$ are cocircuits of $M/q$, orthogonality implies that $r' \in R \cap T$; a contradiction.
    \end{sslproof}

    Let $X' = R \cup S \cup T$ and $W'=E(M\ba d)-X'$.
    With \ref{thatsallfolks} in hand, we henceforth assume that $X \subseteq X'$.
    The next step is to show the following:

    \begin{subsublemma}
      \label{gross}
      Suppose that $M \ba d / t_2/t_3$ has an $N$-minor.
      If the triad $S$ is closed in $M\ba d/t_2$,
      then $M$ has an $N$-detachable pair.
    \end{subsublemma}

    \begin{sslproof}
      Suppose that $S$ is closed in $M\ba d/t_2$.
      Clearly $M \ba d / t_2$ is $3$-connected when $t_2 \in X$.
      Suppose $t_2 \notin X$.
      We may assume that $T-t_2 \subseteq X$, by \cref{triadsinside}.
      Then $(X, \{t_2\}, E(M \ba d) - (X \cup t_2))$ is a cyclic $3$-separation of $M \ba d$, so $\si(M \ba d / t_2)$ is $3$-connected, by Bixby's Lemma.
      But $t_2$ is $N$-contractible in $M \ba d$, so it is not contained in an \unfortunate\ triangle.
      Since every triangle of $M$ is \unfortunate, we deduce that $M \ba d / t_2$ is $3$-connected.

      Now, by Tutte's Triangle Lemma, either
      \begin{enumerate}[label=\rm(\Roman*)]
        \item $M\ba d/t_2/t_3$ and $M\ba d/t_2/s$ are $3$-connected, or\label{cI}
        \item $M\ba d/t_2/r_3$ and $M\ba d/t_2/s'$ are $3$-connected, but $M \ba d/t_2/s''$ is not $3$-connected, for some $\{s',s''\} = \{s,t_3\}$.\label{cII}
      \end{enumerate}
      We first establish some properties that hold in either case.
      Observe that $M / t_2$ is $3$-connected, since $M \ba d / t_2$ is $3$-connected and $\{t_2,d\}$ is not contained in a triangle of $M$.
      We may assume that $r_{M \ba d}^*(X) = r_{M \ba d}^*(X')$, by \cref{triadsinside}.
      Hence $r_{M \ba d}^*(W') \ge 3$.

      \medskip
      Our first claim is that we may assume that $d$ is not in a triangle of $M/t_2$ with two elements from $S$.
      Suppose $d$ is in such a triangle $U$.
      By orthogonality, $r_3 \in U$.
      So let $S - U = \{s'\}$ and let $U = \{d,r_3,s''\}$, where $\{s',s''\} = \{s,t_3\}$.
      Since $E(M/t_2) - (S\cup d)$ is a hyperplane of $M/t_2\ba s'$, we have that $(U,E(M/t_2\ba s')-U)$ is a $2$-separation in $M/t_2\ba s'$, so $\co(M/t_2\ba s')$ is not $3$-connected.
      By Bixby's Lemma, $\si(M/t_2/s')$ is $3$-connected.
      If $M/t_2/s'$ is $3$-connected, then $M$ has an $N$-detachable pair as required (using \cref{triadsncontractiblemain}\cref{triadsncontractible} when $s' \neq t_3$).
      Otherwise, if $M/t_2/s'$ has a parallel pair that does not contain $d$, then, by orthogonality, $S$ is not closed in $M \ba d / t_2$; a contradiction.
      Suppose $M/t_2/s'$ has a parallel pair $\{d,q\}$.
      Then $M/t_2$ has triangles $U$ and $\{s',d,q\}$, so, by circuit elimination, $S \cup q$ contains a circuit of $M/t_2$.
      Since $S$ is closed in $M \ba d / t_2$, the set $S$ contains a circuit of $M/t_2$.
      But then $S$ is a triangle and a triad of $M \ba d/t_2$, so this matroid is not $3$-connected; a contradiction.
      This proves the first claim.
      \medskip

      Let $s' \in \{t_3,s\}$ such that $M \ba d / t_2 / s'$ is $3$-connected.
      Then either $M / t_2 / s'$ is $3$-connected, in which case $M$ has an $N$-detachable pair (using \cref{triadsncontractiblemain}\cref{triadsncontractible} when $s'=s$), or there exists a triangle $\{s',d,\alpha\}$ in $M/t_2$.
      As $d$ blocks the triad $R$ of $M\ba d$, the set $R\cup d$ is a cocircuit in $M/t_2$, and so, by orthogonality, $\alpha\in R$.
      Since $d$ is not in a triangle of $M/t_2$ with $\{s',r_3\} \subseteq S$, we have that $\alpha\neq r_3$.  Thus $\{d,s',\alpha\}$ is a triangle of $M/t_2$ for some $\alpha \in \{r_1,r_2\}$.
      \medskip

      Suppose (I) holds.
      Since $M \ba d/t_2/t_3$ and $M \ba d/t_2/s$ are $3$-connected, we may assume that $M/t_2$ has triangles $\{d,t_3,\alpha\}$ and $\{d,s,\beta\}$ for some $\alpha,\beta \in \{r_1,r_2\}$.
      By circuit elimination, $\{\alpha,\beta,s,t_3\}$ contains a circuit of $M/t_2$.
      It follows by orthogonality that $\{\alpha,\beta\} = \{r_1,r_2\}$, so we may assume that $\{d,t_3,r_1\}$, $\{d,s,r_2\}$, and $\{r_1,r_2,s,t_3\}$ are circuits in $M/t_2$.
      As $S$ and $R$ are triads of $M/t_2\ba d$, we have that $r_{M/t_2\ba d}(W'\cup t_1)\leq r(M/t_2\ba d)-2$.
      But now $r_{M/t_2}(W'\cup t_1)\leq r(M/t_2\ba r_3)-2$ and $r_{M/t_2}(\{r_1,r_2,s,t_3,d\})=3$, so that $\lambda_{M/t_2\ba r_3}(\{r_1,r_2,s,t_3,d\}) \le 1$.

      By Bixby's Lemma, $\si(M/t_2/r_3)$ is $3$-connected, hence either $M/t_2/r_3$ is $3$-connected, in which case $M$ has an an $N$-detachable pair by \cref{triadsncontractiblemain}\cref{triadsncontractible}, or $r_3$ is contained in some triangle~$U$ of $M/t_2$.
      Since $R\cup d$ and $S\cup d$ are both cocircuits of $M/t_2$ containing $r_3$, orthogonality and the fact that $S$ is closed in $M \ba d / t_2$ implies that $U$ also contains $d$.
      The final element of $U$ cannot be in $\{s,t_3\}$, and also cannot be in $\{r_1,r_2\}$, otherwise $(R \cup S \cup d, W' \cup t_1)$ is a $2$-separation of $M/t_2$.
      So either $U$ contains $t_1$, or $U$ meets $W'$.

      We first consider the latter case.  Let $U = \{d,r_3,w\}$ for $w \in W'$.
      Then $\{t_2,d,r_3,w\}$ is a circuit of $M$, and, by circuit elimination with $\{t_2,d,s,r_2\}$, we have that $\{r_3,r_2,d,s,w\}$ contains a circuit.
      But $d \in \cocl(T)$, so $\{r_3,r_2,s,w\}$ is a circuit.
      Since $\{w,r_3\}$ is a parallel pair in $M\ba d/s/r_2$, which has an $N$-minor by \cref{triadsncontractiblemain}\cref{triadsncontractible}, $w$ is $N$-deletable in $M \ba d$.
      As $\co(M \ba d \ba w)$ is $3$-connected and $M$ has no $N$-detachable pairs, $w$ is in a triad of $M \ba d$ that contains an element $y \in \{s,r_2,r_3\}$ and an element $w' \in W'-w$.
      If $y \in X$, then $\co(M\ba d \ba y)$ is $3$-connected; a contradiction.
      So $y \in \{s,r_2,r_3\}-X$.
      By orthogonality between the triad $\{y,w,w'\}$ and the circuit $\{r_1,r_2,s,t_3\}$ of $M/t_2$, we deduce that $y=r_3$.
      Now $(R \cup S)-r_3 \subseteq X$.
      If $R\cup T$ contains a $4$-element cosegment, then by orthogonality with the circuit $\{r_3,r_2,s,w\}$, the cosegment is $T \cup r_1$.  But then $\{d,r_1,t_2,t_3\}$ is a cocircuit of $M$ that intersects the triangle $\{d,s,r_2\}$ in a single element; a contradiction.
      As $X$ contains a triad, it now follows that $T \subseteq X$, and hence $X=X'-r_3$.

      Now $M / r_1 /r_3$ has an $N$-minor by \cref{triadsncontractiblemain}\cref{triadsncontractible}.
      In what follows, we frequently use the fact that no triangle of $M$ meets $X$, and $r_3 \notin \cl(X)$, for otherwise $\lambda_{M \ba d}(X \cup r_3)=1$.
      Observe that $(X-r_1, \{r_3\}, W')$ is a cyclic $3$-separation in the $3$-connected matroid $M \ba d / r_1$, so $\si(M \ba d / r_1/r_3)$ is $3$-connected.
      If $\{r_1,r_3\}$ is contained in a $4$-element circuit~$C$ of $M$, then, by orthogonality, either $d \in C$, or $C$ contains an element in $\{w,w'\}$ and an element in $\{s,t_3\}$.
      If $w \in C$, then by circuit elimination with $\{r_2,r_3,s,w\}$, the set $R \cup S$ contains a circuit, in which case, by orthogonality, $R \cup s$ is a circuit; a contradiction.
      If $w' \in C$, then $w' \in \cl(X' \cup w) \cap \cocl_{M \ba d}(X' \cup w)$, where $X' \cup w$ is $3$-separating since $w \in \cl(X')$; a contradiction.
      So $d \in C$.
      Now $C = \{r_1,r_3,d,t_i\}$ for some $i \in [3]$.
      If $i=2$, then $R \cup \{s,t_2\}$ contains a circuit, by circuit elimination with $\{d,s,r_2,t_2\}$, in which case, by orthogonality, this circuit is $R \cup s$; a contradiction.
      On the other hand, if $i=3$, then, by circuit elimination with $\{d,t_3,r_1,t_2\}$, the set $\{r_1,r_3,t_2,t_3\}$ is a circuit, which is again contradictory.
      So $C = \{r_1,r_3,d,t_1\}$, in which case $T \cup \{r_1,r_3\}$ is a circuit, by circuit elimination with $\{d,t_3,r_1,t_2\}$.
      This circuit cannot contain $r_3$, so $T \cup r_1$ is a circuit, contradicting orthogonality.
      We deduce that $\{r_1,r_3\}$ is not contained in a $4$-element circuit of $M$, so $M/r_1/r_3$ is $3$-connected, and hence $M$ has an $N$-detachable pair.

      Now we may assume that the triangle~$U$ is $\{d,t_1,r_3\}$.
      Note that $W'$ and $W'\cup t_1$ are each exactly $3$-separating in $M/t_2$, and thus $t_1\in\cl_{M/t_2}(R \cup S \cup d)\cap\cl_{M/t_2}(W')$.
      Suppose that $M/t_2\ba r_2$ is not $3$-connected. Then, as $M/t_2$ has no series pairs, $M/t_2\ba r_2$ has a non-trivial $2$-separation $(P,Q)$.
      By \cref{aquickaside1}, we may assume that the triad $\{d,r_1,r_3\}$ is contained in $P$.
      Likewise, as $\{d,t_3,r_1\}$ is a triangle in $M/t_2\ba r_2$, we may assume that $t_3\in P$, and, as $S\cup d$ is a cocircuit, that $s \in P$. 
      But $r_2\in\cl_{M/t_2}(\{s,d\})$, so $(P\cup r_2,Q)$ is a $2$-separation in the $3$-connected matroid $M/t_2$; a contradiction.
      So $M/t_2\ba r_2$ is $3$-connected.

      By Bixby's Lemma, $M/t_2\ba r_2\ba t_1$ is now $3$-connected unless $t_1$ is in a triad~$\Gamma$ of $M/t_2\ba r_2$ that meets both $W'$ and $\{r_1,r_3,d,s,t_3\}$.
      Let $\Gamma \cap W' = \{w\}$.
      If $d \notin \Gamma$, then, as $\Gamma \cup r_2$ is a cocircuit of $M/t_2$, orthogonality implies that this cocircuit intersects the triangles $\{d,s,r_2\}$ and $\{d,t_1,r_3\}$ in at least two elements; a contradiction. 
      So we may assume that $\Gamma = \{w,d,t_1\}$.
      But, recalling that $\{d,t_3,r_1\}$ is a triangle of $M/t_2$, this also contradicts orthogonality.
      Therefore $M/t_2\ba r_2\ba t_1$ is $3$-connected.

      Suppose that $M\ba r_2\ba t_1$ is not $3$-connected.
      Then $M$ has a cocircuit $C^*=\{t_1,t_2,r_2,\delta\}$ where $\delta \in \{r_1,t_3,d\}$, by orthogonality.
      Recall that $M$ has the following cocircuits: $C_T=T\cup d$, $C_R=R\cup d$, and $C_S=S\cup d$.
      If $\delta\in C_S$, then $r_1\not\in C^*\cup C_S\cup C_T$ so that $W'=E(M)-(C^*\cup C_S\cup C_T\cup C_R)$ is a flat of rank at most $r(M)-4$.
      If $\delta\not\in C_S$, then $\delta=r_1$, so that $s\not\in C^*\cup C_R\cup C_T$, again implying that $W'$ is a flat of rank at most $r(M)-4$.
      But $W'$ is exactly $3$-separating in $M$, and, due to the triangles $\{d,t_3,r_1\}$, $\{d,s,r_2\}$, and $\{d,t_1,r_3\}$ of $M/t_2$, we have $r_M(X') \le 5$, contradicting the fact that $M$ is $3$-connected.
      So $M \ba r_2 \ba t_1$ is $3$-connected.

      It remains to show that $M \ba r_2 \ba t_1$ has an $N$-minor.
      First we show that $X=X'$.
      Recall that $\{r_1,r_2,s,t_3\}$ is a circuit of $M/t_2$.
      By orthogonality, $\{r_1,r_2,s,t_3,t_2\}$ is a circuit of $M$.
      Recall also that $\{d,t_1,r_3\}$ and $\{d,s,r_2\}$ are triangles of $M/t_2$.
      By circuit elimination, $M$ has a circuit contained in $\{t_1,t_2,s,r_2,r_3\}$.
      By orthogonality, and since each triangle of $M$ is \unfortunate, $\{t_1,t_2,s,r_2,r_3\}$ is a circuit of $M$.
      So $r(X') \le 5$.
      Now 
      \begin{align*}
          \lambda_{M \ba d}(X') &= r(X') + r^*_{M \ba d}(X') - |X'| \\
          &\le 5 + 4 - 7 = 2.
      \end{align*}
      Hence $r(X') = 5$ and $r^*_{M \ba d}(X') = 4$.
      Moreover, for each $x \in X'$, we have $x \in \cl(X'-x) \cap \cocl_{M \ba d}(X'-x)$, so $X'-x$ is not $3$-separating.
      Suppose $X \subsetneqq X'$.
      Then, by \cref{triadsinside} and since $X$ contains a triad, we may assume that $|X| = 5$, and hence $r(X) \ge 4$.
      But $r^*(X) = r^*(X')=4$, so $\lambda_{M \ba d}(X) \ge 3$; a contradiction.
      We deduce that $X = X'$.

      Now $r_2,t_1 \in X$ and $r_2 \in \cl(X-\{r_2,t_1\})$.
      Hence $\{r_2,t_1\}$ is an $N$-detachable pair by \cref{triadsdelete}.
      This proves \cref{gross} when (I) holds.
      \medskip

      Now suppose (II) holds.
      Since $M\ba d/t_2/r_3$ is $3$-connected, either $M$ has an $N$-detachable pair, or there exists a triangle $\{d,r_3,\gamma\}$ in $M/t_2$.
      For some $\{s',s''\}=\{s,t_3\}$, the matroid $M\ba d/t_2/s'$ is $3$-connected but $M\ba d/t_2/s''$ is not $3$-connected.
      Since $M\ba d/t_2/s'$ is $3$-connected we may assume that $M/t_2$ has a triangle $\{d,s',\alpha\}$ for some $\alpha \in \{r_1,r_2\}$.  Without loss of generality let $\alpha = r_1$.

      Consider $M/t_2/s''$.
      If this matroid is $3$-connected, then $M$ has an $N$-detachable pair (by \cref{triadsncontractiblemain}\cref{triadsncontractible} when $s'' = s$).
      So we may assume that $M/t_2/s''$ is not $3$-connected.
      Note that there are no triangles in $M/t_2$ that contain $s''$ but not $d$, by orthogonality and since $S$ is closed in $M \ba d /t_2$.
      Thus, if $M/t_2/s''$ is $3$-connected up to parallel classes, then this matroid has only a single parallel pair, which contains $d$, so $M\ba d/t_2/s''$ is $3$-connected; a contradiction.
      So $M/t_2/s''$ has a $2$-separation $(P,Q)$ where we may assume that either $P$ or $Q$ is fully closed, by \cref{aquickaside1}.
      Thus, to begin with, we may assume that the triangle $\{d,s',r_1\}\subseteq P$ and $P$ is fully closed.
      Now $r_3\notin P$, as otherwise $s''\in\cl^*_{M/t_2}(P)$, which would result in a $2$-separation $(P\cup s'',Q)$ in $M/t_2$.
      So $r_3\in Q$ and thus $\{\gamma,r_2\} \subseteq Q$ as well, since $\{d,r_3,\gamma\}$ is a triangle and $R\cup d$ is a cocircuit.
      But now $d\in\cl_{M/t_2/s''}(Q)$, and $r_1\in\cl^*_{M/t_2/s''}(Q\cup d)$, %and thus $X-\{t_2,s''\} \subseteq\fcl_{M/t_2/s''}(Q)$,
      so that $(P',Q') = (P-\fcl_{M/t_2/s''}(Q),\fcl_{M/t_2/s''}(Q))$ is a $2$-separation of $M/t_2/s''$ in which
      $R \cup \{s',\gamma,d\} \subseteq Q'$.
      But now $s''\in\cl^*_{M/t_2}(Q')$ and $(P',Q'\cup s'')$ is a $2$-separation in the $3$-connected matroid $M/t_2$; a contradiction.
      This completes the proof of \cref{gross}.
    \end{sslproof}

    \begin{subsublemma}
      \label{newfix}
      Suppose that $M /t_1/t_3$ and $M/t_2/t_3$ have $N$-minors, and $M$ has no $N$-detachable pairs.
      Then, for each $t \in \{t_1,t_2\}$, there exists $w_{t} \in E(M \ba d)$ such that $S \cup \{t,w_t\}$ contains a circuit.
      Moreover, if $w_{t_2} \notin X'$, then 
      \begin{itemize}
        \item $\{s,t_2,t_3,w_{t_2}\}$ is a circuit,
        \item $\{y,w_{t_2},w'\}$ is a triad of $M \ba d$ for some $y \in \{s,t_2,t_3\}-X$ and $w' \in W'-w_{t_2}$,
        \item $w_{t_1} \in X'$, and
        \item either $R \cup t_1$ is a cosegment of $M \ba d$, or $X = X'-y$.
      \end{itemize}
    \end{subsublemma}
    \begin{sslproof}
      By \cref{gross} and symmetry, there exists $w_t \in E(M \ba d)-(S \cup t)$ such that $S \cup \{w_t,t\}$ contains a circuit, for each $t \in \{t_1,t_2\}$.
      Suppose that $w_{t_2} \notin X'$.
      Then, by orthogonality, $\{s,t_2,t_3,w_{t_2}\}$ is a circuit.
      Note that $M\ba d/s/t_2$ has an $N$-minor by \cref{triadsncontractiblemain}\cref{triadsncontractible}, so $w_{t_2}$ is $N$-deletable in $M \ba d$.
      Since $\co(M \ba d \ba w_{t_2})$ is $3$-connected and $M$ has no $N$-detachable pairs, $w_{t_2}$ is in a triad of $M \ba d$ that contains an element $y \in \{s,t_2,t_3\}$ and an element $w' \in W'-w_{t_2}$.
      If $y \in X$, then $\co(M\ba d \ba y)$ is $3$-connected; a contradiction.
      So $y \in \{s,t_2,t_3\}-X$.

      Suppose that $w_{t_1} \notin X'$.
      Then, by symmetry, $\{s,t_1,t_3,w_{t_1}\}$ is a circuit, and $w_{t_1}$ is in a triad $\{y'',w_{t_1},w''\}$ for some $y'' \in \{s,t_1,t_3\}$ and $w'' \in W'-w_{t_1}$.
      If $w_{t_1} = w_{t_2}$, then $T \cup s$ contains a circuit, by circuit elimination.  By orthogonality with the triad $\{y,w_{t_2},w'\}$, this circuit is a triangle.  But this triangle meets $X$; a contradiction.  So $w_{t_1} \neq w_{t_2}$.
      If $y \neq y''$, then, as $y,y'' \notin X$ and by \cref{triadsinside}, we may assume that $\{y,y''\} = \{s,t_2\}$, so $y = t_2$ and $y'' = s$.
      But then the triad $\{s, w_{t_1}, w''\}$ meets the circuit $\{s,t_2,t_3,w_{t_2}\}$, so $w_{t_2} = w''$ by orthogonality.
      Now $w_{t_1},w'' \in \cl(X')$.
      Hence $X' \cup w_{t_1}$ is $3$-separating, but then, due to the triad $\{y'',w_{t_1},w''\}$, we have that $w'' \in \cl(X' \cup w_{t_1}) \cap \cocl_{M \ba d}(X' \cup w_{t_1})$.
      Since $M \ba d$ is $3$-connected, this implies $|W'| = 3$, but then $W'$ is a triangle, contradicting that $r^*(W') \ge 3$.
      We deduce that $y = y''$, and hence $y \in \{s,t_3\}$.
      Now, by orthogonality between the circuit $\{s,t_1,t_3,w_{t_1}\}$ and the triad $\{y,w_{t_2},w'\}$, we have $w_{t_1} = w'$.
      As before, since $w_{t_1},w_{t_2} \in \cl(X')$ and $w_{t_1} \in \cocl_{M \ba d}(X' \cup w_{t_2})$, this is contradictory.
      This proves that $w_{t_1} \in X'$.

      Recall that $r^*_{M \ba d}(X)=r^*_{M \ba d}(X')$.
      Suppose that $r^*_{M \ba d}(X) \ge 4$.
      Then the triad $\Gamma$ contained in $X$ is either $R$, $S$, or $T$.
      Suppose also that $|X| \le 5$.
      Then $r(X) \le 3$, since $\lambda_{M \ba d}(X) = 2$.
      Since $X$ does not contain any triangles, $\Gamma \cup x$ is a circuit for every $x \in X-\Gamma$; but this contradicts orthogonality.
      We deduce that $|X|=6$, and hence $X=X'-y$.

      Now suppose $r^*_{M \ba d}(X) = 3$.
      Then $\{t_3,r_3,r_1\}$ cospans $X'$ in $M \ba d$, so $\{t_3,r_3,r_1,t_1\}$ contains a cocircuit.
      By orthogonality, this cocircuit does not meet the circuit $\{s,t_2,t_3,w_{t_2}\}$, so $\{r_1,r_3,t_1\}$ is a triad, and hence $R \cup t_1$ is a cosegment of $M \ba d$ as required.
    \end{sslproof}

    \begin{subsublemma}
      \label{newfix2}
      Either $M$ has an $N$-detachable pair, or, up to swapping $R$ and $T$, for each $i \in \{1,2\}$ there exists $w_{t_i} \in X'$ such that $S \cup \{t_i,w_{t_i}\}$ contains a circuit.
    \end{subsublemma}
    \begin{sslproof}
      Suppose that $M/t_i/t_3$ does not have an $N$-minor for some $i \in \{1,2\}$.
      Then, by \cref{triadsncontractiblemain}\cref{triadsncontractible2}, we may assume that $M/r_1/r_3$ and $M/r_2/r_3$ have $N$-minors, for otherwise $M$ has an $N$-detachable pair.
      So, up to swapping $R$ and $T$, we may assume that $M/t_1/t_3$ and $M/t_2/t_3$ have $N$-minors.

      Now, by \cref{gross}, we may assume that for each $i \in \{1,2\}$ there exists $w_{t_i}$ such that $S \cup \{t_i,w_{t_i}\}$ contains a circuit.
      If $\{w_{t_1},w_{t_2}\}\subseteq X'$, then \cref{newfix2} holds; so assume that $w_{t_2} \notin X'$.
      Then, by \cref{newfix}, $w_{t_1} \in X'$, and either $R \subseteq X$, or $R \cup t_1$ is a cosegment of $M \ba d$.
      In either case, \cref{triadsncontractiblemain}\cref{triadsncontractible} implies that $M/r_i/r_3$ has an $N$-minor for each $i \in \{1,2\}$.
      By \cref{gross} and symmetry, for each $i \in \{1,2\}$ there exists $w_{r_i}$ such that $S \cup \{r_i,w_{r_i}\}$ contains a circuit.
      If $\{w_{r_1},w_{r_2}\}\subseteq X'$, then \cref{newfix2} holds, after swapping $R$ and $T$.
      So we may assume, without loss of generality, that $w_{r_2} \notin X'$.

      We can now apply \cref{newfix} a second time, with $R$ in the role of $T$.
      Then $\{s,r_2,r_3,w_{r_2}\}$ is a circuit and $\{y'',w_{r_2},w''\}$ is a triad of $M \ba d$ for some $y'' \in \{s,r_2,r_3\}-X$ and $w'' \in W'-w_{r_2}$.
      By orthogonality, $R \cup t_1$ is not a cosegment of $M \ba d$, so $R \subseteq X$, and hence $y'' = s$.
      If $w_{r_2} \neq w_{t_2}$, then, as $\{s,t_2,t_3,w_{t_2}\}$ is a circuit, $w'' = w_{t_2}$, so $\{s,w_{r_2},w_{t_2}\}$ is a triad.
      Then $X' \cup w_{r_2}$ is $3$-separating, and $w_{t_2} \in \cl(X' \cup w_{r_2}) \cap \cocl_{M\ba d}(X' \cup w_{r_2})$; a contradiction.
      So $w_{r_2} = w_{t_2}$.
      Now, by circuit elimination, $S \cup \{t_2,r_2\}$ contains a circuit.
      Since $w_{t_1} \in X'$, this completes the proof of \cref{newfix2}.
    \end{sslproof}

    \begin{subsublemma}
      \label{theressomecocircuits}
      We may assume that both $S\cup\{r_1,t_1\}$ and $S\cup\{r_2,t_2\}$ contain circuits of $M \ba d$.
    \end{subsublemma}

    \begin{sslproof}
      Suppose that $S \cup t_1$ contains a circuit in $M\ba d/t_2$.
      As $r_3\in\cl^*_{M\ba d/t_2}(\{r_1,r_2\})$, the set $\{t_1,t_3,s\}$ is a triangle in $M\ba d/t_2$.
      Hence $\{t_1,t_2,t_3,s\}$ is a circuit of $M\ba d$.
      But now $(T, \{s\}, E(M\ba d)-(T \cup s))$ is a vertical $3$-separation, and $M\ba d /s$ is not $3$-connected; so $s \notin X$.
      By \cref{triadsinside}, we may assume that $S-s \subseteq X$ and $|T \cap X| \ge 2$.  Then, by uncrossing, $X \cup T$ is $3$-separating.
      But $s \in \cl(X \cup T) \cap \cocl_{M \ba d}(X \cup T)$; a contradiction.

      By \cref{newfix2}, we may assume that, for each $i \in \{1,2\}$, there exists $w_{t_i} \in X'$ such that $S \cup \{t_i,w_{t_i}\}$ contains a circuit.
      From the previous paragraph, and symmetry, $w_{t_1},w_{t_2} \in \{r_1,r_2\}$.
      Now, up to swapping the labels on $r_1$ and $r_2$, \cref{theressomecocircuits} holds unless $r_1$ is in a circuit $C_1$ contained in $S \cup \{t_1,r_1\}$ and a circuit $C_2$ contained in $S \cup \{t_2,r_1\}$.
      Suppose we are in the exceptional case.
      Note that $S \cup r_1$ does not contain a circuit, by orthogonality and since no element of $X$ is in a triangle.
      So $C_1 \neq C_2$.
      By circuit elimination, there is a circuit contained in $S \cup \{t_1,t_2\}$; a contradiction.
      This completes the proof. % of \cref{theressomecocircuits}.
    \end{sslproof}

    By \ref{theressomecocircuits}, $\{r_1,r_3,t_1,t_3\}$ and $\{r_2,r_3,t_2,t_3\}$ each contain circuits in $M\ba d/ s$.
    In fact, by orthogonality and since no triangles meet $X$, $\{r_1,r_3,t_1,t_3\}$ and $\{r_2,r_3,t_2,t_3\}$ are circuits of $M\ba d/ s$.
    If $\{r_1,r_3,t_1,t_3\}$ and $\{r_2,r_3,t_2,t_3\}$ are circuits of $M \ba d$, then $R$ and $T$ are disjoint triads of $M \ba d$ with $\lc(R,T)=2$, so $M$ has an $N$-detachable pair by \cref{triadfans}.
    So we may also assume that $s \in \cl(X'-s)$.

    Observe now that $x \in \cl(X'-x)$ for each $x \in X'$.
    If $X \neq X'$, then by \cref{triadsinside} and uncrossing, there exists some element $x \in X'-X$ for which $X'-x$ is $3$-separating.
    But $X'$ is $3$-separating and $x \in \cl(X'-x) \cap \cocl_{M \ba d}(X'-x)$ for each $x \in X'$; a contradiction.
    We deduce that $X=X'$.

    \begin{subsublemma}
      \label{therebemorecocircuits}
      We may assume that
      %$S\cup d$, $\{r_1,t_1,s,d\}$ and $\{r_2,t_2,s,d\}$
      $\{r_1,t_1,s,d\}$, $\{r_2,t_2,s,d\}$ and $\{r_3,t_3,s,d\}$ are circuits of $M$.
    \end{subsublemma}

    \begin{sslproof}
      Suppose $C$ is a $4$-element circuit of $M \ba d$ containing $s$.
      By orthogonality, $|C \cap X| \ge 3$.  Suppose $C - X = \{w\}$.  Then $w \in \cl_{M \ba d}(X)-X$.  By \cref{triadsncontractiblemain}\cref{triadsncontractible}, $w$ is $N$-deletable in $M \ba d$, so $M$ has an $N$-detachable pair by \cref{closuresubst}.
      So we may assume that $C \subseteq X$.
      If $T \subseteq C$, then $T$ is a triangle and a triad of $M \ba d / s$, contradicting that this matroid is $3$-connected.  So $T \nsubseteq C$ and, similarly, $R \nsubseteq C$.
      But now either $|T\cap C|=1$ or $|R\cap C|=1$, contradicting orthogonality.
      Hence, for each $x\in X-s$, the matroid $M \ba d/s/x$ has no parallel pairs.

      Suppose $M\ba d/s/t_2$ is not $3$-connected.
      Then it has a non-trivial $2$-separation $(P,Q)$.
      In what follows, \cref{aquickaside1} will be used freely.
      We may assume that $R\subseteq P$ and $P$ is fully closed in $M\ba d / s / t_2$.
      Now $t_3\in Q$, as otherwise $(P\cup s,Q)$ is $2$-separating in $M\ba d / t_2$.
      Moreover, $t_1\in Q$ as $\{t_1,t_3,r_1,r_3\}$ is a circuit in $M \ba d/s/t_2$.
      But now $t_2\in\cl_{M\ba d/s}(Q)$ and $(P,Q\cup t_2)$ is $2$-separating in $M\ba d/s$.
      Therefore $M\ba d/s/t_2$ is $3$-connected.

      By symmetry, $M\ba d/s/t_1$, $M\ba d/s/r_1$ and $M\ba d/s/r_2$ are $3$-connected.
      A similar argument also gives that both $M\ba d/s/r_3$ and $M\ba d/s/t_3$ are $3$-connected.
      Thus, by \cref{triadsncontractiblemain}\cref{triadsncontractible} and since $X = X'$, the element $d$ is in some triangle with every element from $X-s$ in the matroid $M/s$.
      By orthogonality, these triangles intersect $R$ and $T$ in a single element each.
      As $\{r_1,t_1,r_3,t_3\}$ and $\{r_2,t_2,r_3,t_3\}$ are circuits in $M/s$, the only possible arrangement is that $\{r_1,t_1,d\}$, $\{r_2,t_2,d\}$ and $\{r_3,t_3,d\}$ are triangles of $M/s$.
    \end{sslproof}

    We now work towards showing that $M \ba t_1 \ba r_2$ is an $N$-detachable pair.
    First, suppose that $M \ba t_1 \ba r_2$ has a series pair.  Then there is a $4$-element cocircuit of $M$ containing $\{t_1,r_2\}$.  By orthogonality, either this cocircuit meets $\{d,s\}$, in which case the other two elements are from the circuit $\{d,s,t_3,r_3\}$, or the cocircuit is $\{t_1,t_2,r_1,r_2\}$.
    In the latter case, $X \subseteq \cocl_{M \ba d}(\{t_1,t_2,r_1\})$, so $r^*_{M \ba d}(X) =3$.  But as $r(X) = 5$ and $|X| = 7$, the set $X$ is $2$-separating in $M$; a contradiction.
    Similarly, if $\{t_1,r_2,z_1,z_2\}$ is a cocircuit for distinct $z_1,z_2 \in \{s,t_3,r_3\}$, then again $r^*_{M \ba d}(X) =3$; a contradiction.
    So we may assume that $\{t_1,r_2,d,z\}$ is a cocircuit for $z \in \{s,t_3,r_3\}$.
    By orthogonality with the circuits $\{r_1,r_3,t_1,t_3\}$ and $\{r_2,r_3,t_2,t_3\}$, we see that $z \neq s$.
    Now $\{s,r_1,t_1,t_2\} \subseteq \cocl_{M \ba d}(\{r_2,r_3,t_3\})$, so $r^*_{M \ba d}(X) =3$; a contradiction.
    %In either case, \ldots $r^*_{M \ba d}(X) < 4$ \ldots

    Now we may assume that if $M \ba t_1 \ba r_2$ is not $3$-connected, it has a non-trivial $2$-separation $(P,Q)$.
    By \cref{aquickaside1}, we may assume that $\{r_1,r_3,d\} \subseteq P$.
    If $t_3 \in P$, then $\{s,t_2\} \subseteq P$, and it follows that $(P \cup \{t_1,r_2\},Q)$ is a $2$-separation in $M$; a contradiction.
    So $t_3 \in Q$, and, similarly, $\{s,t_2\} \subseteq Q$.
  But now, $(Q', P')= (\fcl(Q), P-\fcl(Q))$ is also a $2$-separation, where $r_3 \in Q'$, hence $r_1 \in Q'$, so $(Q' \cup \{t_1,r_2\}, P')$ is a $2$-separation of $M$; a contradiction.  We deduce that $M \ba t_1 \ba r_2$ is $3$-connected.
    By \cref{triadsdelete}, $\{t_1,r_2\}$ is an $N$-detachable pair.
    This completes the proof of \cref{triadsdirty}.
  \end{slproof}

  \begin{sublemma}
    \label{triadsfinal}
    Let $R$, $S$, and $T$ be distinct triads of $M\ba d$ that meet $X$, where the union of any two of these triads is not a cosegment of $M \ba d$.
    If $|S\cap T|=1$, and $|R\cap(S\cup T)|= 2$,
    %If $|R \cap S| = |S\cap T| = |R \cap T|=1$,
    then $M$ has an $N$-detachable pair.
  \end{sublemma}
  \begin{slproof}
    Since the union of any two of $R$, $S$, and $T$ is not a cosegment, $|R \cap S| = |S\cap T| = |R \cap T|=1$ and $R \cap S \cap T = \emptyset$.
    Let $S=\{s_1,s_2,t_3\}$, $T=\{t_1,t_2,t_3\}$ and $R=\{s_1,t_1,r\}$ (see \cref{triadfigs5}).

    \begin{subsublemma}
      \label{case4sublemma}
      %If $X \neq R\cup S\cup T$,
      If $X - (R\cup S\cup T) \neq \emptyset$,
      then $M$ has an $N$-detachable pair.
    \end{subsublemma}

    \begin{sslproof}
      Suppose $x\in X-(R\cup S\cup T)$.
      As $\co(M\ba d\ba x)$ is $3$-connected, we may assume $x$ is in a triad $\Gamma\subseteq X$, otherwise $\{d,x\}$ is an $N$-detachable pair.
      By \cref{triadsseparate,triadsdirty}, we may assume that $\Gamma$ intersects each of $R$, $S$ and $T$.
      If $\Gamma$ intersects $R$, $S$, or $T$ in two elements, then $M$ has an $N$-detachable pair by \cref{triadsdirty}. Now, $\Gamma\in\{\{x,t_3,r\},\{x,s_1,t_2\},\{x,t_1,s_2\}\}.$
      By \cref{triadsatu}, $M$ has an $N$-detachable pair in each case.
    \end{sslproof}

    Now suppose that $M$ has no $N$-detachable pairs.

    \begin{subsublemma}
      \label{case4sublemma1}
      Either $X = R \cup S \cup T$ or $X = (R \cup S \cup T) - z$ for some $z \in \{t_1,t_3,s_1\}$.
    \end{subsublemma}
    \begin{slproof}
      By \cref{case4sublemma}, $X \subseteq R \cup S \cup T$.
      %We claim that either $X = R \cup S \cup T$ or $X = (R \cup S \cup T) - z$ for $z \in \{t_1,t_3,s_1\}$.
      Suppose $X \subsetneqq R \cup S \cup T$, and $r \notin X$.  Then $X \cup S \cup T$ is $3$-separating, by \cref{triadsinside} and uncrossing, but this $3$-separating set is just $S \cup T$, contradicting \cref{triadsintersecting}.
      By symmetry, we deduce $\{r,s_2,t_2\} \subseteq X$.
      Now, if $z \notin X$ for some $z \in \{t_1,t_3,s_1\}$, then $X = (R \cup S \cup T)-z$ by \cref{triadsinside}, thus proving %the claim.
      \cref{case4sublemma1}
    \end{slproof}

    Let $X' = R \cup S \cup T$, and observe that $\lambda_{M \ba d}(X')=2$.
    \begin{subsublemma}
      \label{case4sublemma2}
      For each $x\in X$, %no element of $X-x$ is in a triad of $M^*/d\backslash x$.
      if $M \ba d / x$ contains a triangle that meets $X-x$, then this triangle is $\{x',z,w\}$ where $x' \in X-x$, $z \in X' - X$, and $w \notin X'$.
    \end{subsublemma}

    \begin{sslproof}
      Suppose that for some $x\in X$, there is a triangle~$U$ of $M \ba d / x$ that meets $X-x$.
      Then $U \cup x$ is a $4$-element circuit $C$ of $M\ba d$.
      By orthogonality with $R$, $S$, and $T$, we have $|C\cap X'|\geq 3$.
      Suppose $|C\cap X'|=3$ and let $C -X' = \{w\}$.
      If $C \cap X' \subseteq X$, then $w \in \cl(X)-X$, and $w$ is $N$-deletable in $M \ba d$ by \cref{triadsncontractiblemain}\cref{triadsncontractible}, contradicting \cref{closuresubst}.
      By \cref{case4sublemma1}, $C \cap X' = \{x',z,x\}$ where $x' \in X-x$ and $z \in X'-X$, as required.

      Now suppose that $C \subseteq X'$.
      If $C \nsubseteq X$, then
      $z \in C$ where $z \in X'-X$, and $z \in \cl_{M \ba d}(X)-X$.  %It follows, by \cref{triadsncontractible}, that $z is $N$-deletable.
      But $z \in \cocl_{M \ba d}(X)-X$, contradicting \cref{gutsstayguts2}.
      So $C \subseteq X$.
      By orthogonality and \cref{triadsintersecting}, $C$ contains one of $R$, $S$, or $T$.
      Let $y \in X-C$.
      Since $\lambda_{M \ba d}(X')=2$, we have $r_{M \ba d}(X') = 5$.
      So $y \notin \cl_{M \ba d}(X'-y)$.
      As $y \in \cocl_{M \ba d}(X'-z) \cap\cocl_{M \ba d}(E(M\ba d)-X')$, we see that
      $\co(M\backslash d\ba y)$ is not $3$-connected; a contradiction.
      So $C \nsubseteq X'$.
    \end{sslproof}

    Since $r^*_{M\ba d}(X)=3$ and $\lambda_{M \ba d}(X)=2$, the set $X$ contains a circuit of $M \ba d$.
    Suppose that $X$ properly contains a circuit~$C$.
    By \cref{case4sublemma2}, $|C| \ge 5$, so $|X| = 6$ and $|C|=5$.
    Let $X-C = \{y\}$.
    Then $y \in \cocl(C)$ and $y \notin \cl(C)$, so $(C,\{y\},E(M\ba d)-X)$ is a cyclic $3$-separation of $M \ba d$.
    Hence $\co(M\ba d \ba y)$ is not $3$-connected; a contradiction.
    We deduce that $X$ is a corank-$3$ circuit.

    Combining \cref{case4sublemma2} and two applications of the dual of \cref{r3cocirc2}, it now follows that,
    for all distinct $x,x'\in X$, either $M\ba d/ x/ x'$ is $3$-connected, or there is a $4$-element circuit $\{x,x',z,w\}$ of $M\ba d$, where $z \in X'-X$ and $w \in E(M \ba d)-X'$.
    By symmetry, we may assume that $X=X'-s_1$, so $z = s_1$. %$X'-s_1 \subseteq X$.
    By orthogonality, $\{t_1,t_2,z,w\}$ is not a circuit of $M \ba d$ for any $w \in E(M \ba d)-X'$.
    Similarly, neither $\{t_1,r,z,w\}$ nor $\{t_2,r,z,w\}$ is a circuit of $M \ba d$ for any $w \in E(M \ba d)-X'$.
    Since neither $\{t_1,t_2\}$, $\{t_1,r\}$, nor $\{t_2,r\}$ is an $N$-detachable pair, there are distinct $4$-element circuits $C_1$, $C_2$, and $C_3$ of $M$ containing $\{d,t_1,t_2\}$, $\{d,t_1,r\}$, and $\{d,t_2,r\}$, respectively.
    By orthogonality with the cocircuit $\{s_1,s_2,t_3,d\}$ of $M$, the circuits $C_1$, $C_2$, and $C_3$ each meet $\{s_1,s_2,t_3\}$. %, and $C_2$ meets $\{s_1,s_2,t_3\}$.
    There exists an element $y \in \{s_2,t_3\}$ that is in at most one of these three circuits.  By circuit elimination on two circuits not containing $y$, the set $X'-y$ contains a circuit of $M \ba d$, so $r_{M\ba d}(X'-y) \le 4$.  
    As $r_{M \ba d}(X') = 5$, it follows that $y \notin \cl_{M\ba d}(X'-y)$, so $y \in \cocl_{M \ba d}(X'-y)$ by \cref{exactSeps2}. %\cap\cocl_{M \ba d}(E(M\ba d)-X')$.
    Now $(X'-y,\{y\},E(M\ba d)-X')$ is a cyclic $3$-separation of $M \ba d$.
    Hence $\co(M\backslash d\ba y)$ is not $3$-connected, where $y \in X$; a contradiction.
  \end{slproof}

  We now return to the proof of \cref{cathedral}.
  Suppose that $M$ has no $N$-detachable pairs.
  Then every $x \in X$ is in a triad of $M \ba d$.
  As $|X|\geq 4$, there are distinct triads~$S$ and~$T$ that meet $X$, and $S \cup T$ is not a cosegment, by \cref{triadsnotcoline}.
  Suppose %that $S$ and $T$ are such triads, %that meet $X$
  %with $S \neq T$, where
  $S$ and $T$ meet at an element in $X$.
  As $|S\cap T|=1$, \cref{triadsintersecting} implies that $\lambda_{M\ba d}(S\cup T)>2$.
  By uncrossing and \cref{triadsinside}, the set $X \cup S \cup T$ is $3$-separating.
  Thus, there exists some $r\in X-(S\cup T)$, where $r$ is in a triad $R$.

  First, suppose that every such $r$ is such that either $S \cup r$ or $T \cup r$ is a cosegment.
  %If $|R \cap S| = 2$ say, then %$X \cup S \cup T \cup r$ is $3$-separating, so there exists some $r' \in X-(S\cup T \cup r)$.
  %$R \cup S$ is a $4$-element cosegment that meets $T$ in one element $z$ say.
  Without loss of generality, let $S \cup r$ be a cosegment. %, and let $z$ be the element in the intersection of $T$ and $R \cup S$.
  Now $S \cup r \nsubseteq X$, by \cref{triadsnotcoline}, so $S$ contains an element $z$ not in $X$.
  %If $z \in X$, then $(R \cup S) \cap X$ is a triad that meets $T$ in one element, $z$.
  Since $(S - z) \cup r$ and $T$ are triads that intersect in one element, 
  $T \cup (S-z) \cup r$ is not $3$-separating by \cref{triadsintersecting}.
  As the union of this set and $X$ is $3$-separating, by uncrossing,
  there exists some $r'\in X-(S \cup T \cup r)$, where either $S \cup r'$ or $T \cup r'$ is a cosegment.  If $S \cup r'$ is a cosegment, then $(S-z) \cup \{r,r'\}$ is a $4$-element cosegment contained in $X$, contradicting \cref{triadsnotcoline}.  So $T \cup r'$ is a cosegment.
  Now $T \cup r'$ is not contained in $X$, by \cref{triadsnotcoline}, so there is an element $z' \in T-X$.
  As $(T-z') \cup r'$ and $(S-z) \cup r$ are triads that intersect in one element, repeating the argument above we deduce an element $r'' \in X$ such that either $(T-z') \cup \{r',r''\}$ or $(S-z) \cup \{r,r''\}$ is a $4$-element cosegment contained in $X$; a contradiction.

  Now we may assume that neither $S \cup r$ nor $T \cup r$ is a cosegment.
  So $r$ is in a triad whose intersection with $S$ or $T$ has size at most one.
  By \cref{triadsseparate}, $R$ intersects $S \cup T$; then
  by \cref{triadsdirty}, $R$ intersects both $S$ and $T$.
  Now $|R\cap S|=|R\cap T|=1$, so
  $|R \cap (S \cup T)| \neq 1$ by \cref{triadsatu}, and $|R \cap (S \cup T)| \neq 2$ by \cref{triadsfinal}.
  This contradiction implies there are no two triads $S$ and $T$ that meet $X$, and intersect at a single element in $X$.

  Next, we claim that either $X$ is the disjoint union of two triads, or $X$ is a $5$-element subset of the disjoint union of two triads.
  Certainly $X$ contains a triad $S$ of $M \ba d$, and there is a triad $T$ that meets $X$ and is disjoint from $S$.
  By \cref{triadsinside}, $|T \cap X| \ge 2$.
  If $X-(S\cup T)=\emptyset$, then the claim holds.
  So suppose that $X-(S\cup T)$ is non-empty. 
  Then there is a triad $R$, distinct from $S$ and $T$, that meets $X$.
  So $|R \cap X| \ge 2$.
  If $R$ and $T$ are disjoint, then $M$ has an $N$-detachable pair by \cref{triads3disjoint}; whereas if $R$ intersects $T$ in one element not in $X$, then $R \cup T$ is not a cosegment, by \cref{triadsnotcoline}, so $M$ has an $N$-detachable pair by \cref{triadsseparate}.
  Hence $|R \cap T| = 2$, and $R \cup T$ is a $4$-element cosegment.
  By \cref{triadsinside,triadsnotcoline}, $|(R \cup T)-X|=1$.
  Let $T' = (R \cup T) \cap X$.
  Now $T'$ and $S$ are disjoint triads contained in $X$.
  If $X-(S\cup T')$ is non-empty, then there is another triad $R'$ that meets $X$, and neither $R' \cup S$ nor $R' \cup T'$ is a cosegment, by \cref{triadsnotcoline}.
  So if $R'$ meets $S$ or $T'$, it does so at a single element in $X$; a contradiction.  On the other hand, if this triad is disjoint from $S$ and $T'$, then this contradicts \cref{triads3disjoint}.  We deduce that $S \cup T' = X$, as required.

So we may assume that $X$ is contained in the disjoint union of two triads~$S$ and $T$, and $|X| \in \{5,6\}$.
By \cref{triadfans}, $\sqcap(S,T)=1$.
Let $X' = S \cup T$, let $W'=E(M\ba d)-X'$, and observe that $\lambda_{M \ba d}(X')=2$ and $r(X')=5$.

\begin{sublemma}
\label{triadslocalconn}
For all $2$-element subsets $S' \subseteq S$ and $T' \subseteq T$ such that $S' \neq S \cap X$ and $T' \neq T \cap X$,
\[\sqcap(S',T)=\sqcap(S',W')=\sqcap(T',S)
=\sqcap(T',W')=0.\]
\end{sublemma}

\begin{slproof}
  Let $S-S'=\{s\}$, and note that $s \in X$.
  Suppose $\sqcap(S',T)=1$, so $r(T\cup S')=4$.
  As $r(W'\cup s)=r(W')+1$, the set
  $W'\cup s$ is $3$-separating in $M \ba d$, implying that $s\in\cl^*_{M \ba d}(W')$.
  But then %$(T\cup S',W')$ is a $2$-separation of $M\ba d \ba s$, contradicting that $\co(M \ba d \ba s)$ is $3$-connected.
$\co(M \ba d \ba s)$ is not $3$-connected; a contradiction.
  So $\sqcap(S',T)=0$.

  Similarly, suppose $\sqcap(S',W')=1$, so $r(W' \cup S')=r(W')+1$.
  As $r(T \cup s) = 4$ and $r(X')=5$, we have $\lambda_{M \ba d}(T \cup s) = \lambda_{M \ba d}(X') = 2$, implying that $s\in\cl^*_{M \ba d}(T)$; a contradiction.
  %A similar argument, with the roles of $W$ and $T$ swapped, shows that $\sqcap(S',W)=0$, and
  By symmetry, %we also obtain
  $\sqcap(T',S)=\sqcap(T',W')=0$.
\end{slproof}

Now, if $|X| = 6$, then \cref{triadslocalconn} implies that $X$ is a corank-$3$ circuit in $M \ba d$.
Suppose that $|X|=5$.  Without loss of generality, let $T \subseteq X$ and $s\in S-X$.
If $X$ is not a circuit, it contains a $4$-element circuit, since no element of $X$ is in a triangle.
It follows that $(S-s) \cup (T-t)$ is a circuit for some $t \in T$.
But then $((S-s) \cup (T-t), \{t\}, E(M \ba d)-X)$ is a cyclic $3$-separation of $M \ba d$, implying that $\co(M \ba d \ba t)$ is not $3$-connected; a contradiction.
Moreover, note that $s \notin \cl(X)$ in this case.
So, if $|X| \in \{5,6\}$, then $X$ is the only circuit contained in $X'$.

Let $S = \{s_1,s_2,s_3\}$ and let $T = \{t_1,t_2,t_3\}$, where $T \subseteq X$ and $S-s_2 \subseteq X$.
Since $X$ is a circuit,
$T$ is not contained in a $4$-element fan in $M\ba d/s$, for each $s \in S$.
By Tutte's Triangle Lemma,
at least two of $M\ba d/s/t_1$, $M\ba d/s/t_2$, and $M\ba d/s/t_3$ are $3$-connected, for each $s \in S$.
Up to relabelling the elements of $T$, we may assume that $M \ba d/s_1/t_1$, $M \ba d/s_1/t_2$, and $M \ba d/s_2/t_1$ are $3$-connected.
As each of these matroids also has an $N$-minor, by \cref{triadsncontractiblemain}\ref{triadsncontractible}, we see that either $M$ has an $N$-detachable pair, or there are $4$-element circuits $\{d,s_1,t_1,\alpha\}$, $\{d,s_1,t_2,\beta\}$, and $\{d,s_2,t_1,\gamma\}$ in $M$.

By circuit elimination, $\{t_1,t_2,\alpha,\beta\}$ contains a circuit of $M/s_1$,
so $\sqcap_{M/s_1}(\{t_1,t_2\},\{\alpha,\beta\})=1$.
%
%\begin{sublemma}
%\label{111}
%$\{\alpha,\beta\}\cap(S\cup T)\neq\emptyset$ and $\{\alpha,\beta\}\neq\{s_2,s_3\}$.
%\end{sublemma}
%
%\begin{slproof}
If $\{\alpha,\beta\}\subseteq W'$, then $\sqcap_M(\{t_1,t_2\},\{\alpha,\beta\})=1$, and so
$\sqcap(\{t_1,t_2\},W')\geq 1$, contradicting \ref{triadslocalconn}.
On the other hand, if $\{\alpha,\beta\}=\{s_2,s_3\}$,
then $\sqcap(S,\{t_1,t_2\})\geq 1$, which again contradicts \ref{triadslocalconn}.
So $\{\alpha,\beta\}$ meets both $\{s_2,s_3\}$ and $W'$.

Now, by a similar argument, $\{s_1,s_2,\alpha,\gamma\}$ contains a circuit of $M/t_1$, where $\{\alpha,\gamma\}$ meets both $\{t_2,t_3\}$ and $W'$,
by \cref{triadslocalconn} and since $S \cap X \neq \{s_1,s_2\}$.
So $\alpha \in W'$, $\beta \in \{s_2,s_3\}$ and $\gamma \in \{t_2,t_3\}$.
Again by circuit elimination, $\{s_1,s_2,t_1,t_2,\beta,\gamma\}$ contains a circuit, where $\beta \in \{s_2,s_3\}$ and $\gamma \in \{t_2,t_3\}$.
Since the only circuit contained in $X'$ is $X$,
we deduce that $\beta = s_3$ and $\gamma = t_3$. %, and $S \cup T$ is a circuit.

Now, either $M \ba d/s_2/t_2$ or $M \ba d/s_2/t_3$ is also $3$-connected.
If $M \ba d/s_2/t_3$ is $3$-connected, then, as $M$ has no $N$-detachable pairs, $\{d,s_2,t_3,\zeta\}$ is a circuit, and, by circuit elimination with $\{d,s_2,t_1,t_3\}$, the set $\{s_2,t_1,t_3,\zeta\}$ contains a circuit in $M$.  By orthogonality, $\zeta \in \{s_1,s_3\}$, but this contradicts \cref{triadslocalconn}.
So we may assume that $M \ba d/s_2/t_2$ is $3$-connected.

As $M$ has no $N$-detachable pairs, there is a circuit $\{d,s_2,t_2,\eta\}$ in $M$.
By circuit elimination, and since $\gamma = t_3$, the set $\{t_1,t_2,t_3,\eta\}$ contains a circuit in $M/s_2$.
Since $X$ is the only circuit contained in $X'$, %By the foregoing,
we have $\eta \in W'$.
By orthogonality with $S$, the set $\{t_1,t_2,t_3,\eta\}$ is a circuit of $M$.
But $M\ba d/t_1/t_3$ has an $N$-minor, by \cref{triadsncontractiblemain}\cref{triadsncontractible3}, %(since $X$ is has corank three and is the union of disjoint triads in $M \ba d$),
so $\eta$ is $N$-deletable in $M \ba d$, contradicting \cref{closuresubst}.
This final contradiction completes the proof of \cref{cathedral}.
\end{proof}

\section{The non-triad case}
\label{seclifenontriad}

In this section, we prove \cref{mosque}.
We first prove a lemma that guarantees either the existence of a detachable pair, or specific structured outcomes.
We then consider these structured outcomes relative to an $N$-minor later in the section.

\subsection*{Preserving $3$-connectivity}

\begin{lemma}
  \label{temple}
  Let $M$ be a $3$-connected matroid with an element $d$ such that $M\ba d$ is $3$-connected.
  Let $(X,W)$ be a $3$-separation of $M\ba d$ with $|X|\geq 4$, $r(W) \ge 3$, $r^*_{M \ba d}(W) \ge 4$, and, for each $x \in X$,
  \begin{enumerate}[label=\rm(\alph*)]
    \item $\co(M\ba d\ba x)$ is $3$-connected, 
    \item $M\ba d/x$ is $3$-connected, and 
    \item $x$ is not contained in a triangle or triad of $M$.
  \end{enumerate}
  Suppose $X$ is minimal subject to these conditions, and $X$ does not contain a triad of $M \ba d$.
  Then either:
  \begin{enumerate}
    \item $M \ba d \ba x$ is $3$-connected for some $x \in X$;\label{templefirstcase}
    \item $M / s/t$ is $3$-connected for distinct $s,t \in \cocl_{M \ba d}(X)$ such that $s \in S^*$ and $t \in X \cap (T^*-S^*)$ for distinct triads $S^*$ and $T^*$ of $M \ba d$ that meet $X$;\label{templecontractpair}
    \item $\{x,x',c,w\}$ is a $4$-element circuit of $M \ba d$ where $\{x,x'\} \subseteq X$, $c \in \cocl_{M \ba d}(X)-X$, $w \in W-c$, and $x$ and $x'$ are in distinct triads of $M \ba d$ contained in $X \cup c$; or\label{templehardcase}
    \item $X=\{x_1,x_1',x_2,x_2'\}$ is a quad in $M\ba d$,
      there exists an element $c \in W$ such that $\{x_1,x_1',c\}$ and $\{x_2,x_2',c\}$ are triads of $M \ba d$, and
      for each $x \in X$ there is a $4$-element circuit of $M$ containing $\{x,c,d\}$.\label{templequadcase}
    %\item $X\cup\{c,d\}$ is a \spikelike\ of $M$, where $c \in \cocl_{M \ba d}(X)-X$.\label{templespikecase}
  \end{enumerate}
\end{lemma}

\begin{proof}
  We assume that \cref{templefirstcase} does not hold, and show that one of \cref{templecontractpair,templehardcase,templequadcase} holds.
  We consider two cases: \cref{coclosed} and \cref{notcoclosed}. 
  We first prove two claims that hold in either case.

  \begin{sublemma}
    \label{j1}
    $W$ is fully closed in $M \ba d$.
  \end{sublemma}

  \begin{slproof}
    If $W$ is not closed, then there exists some $x\in X$ such that $M\ba d/x$ fails to be $3$-connected; a contradiction.
    %As $X$ contains no triads, it contains a circuit of $M\backslash d$.
    Suppose $x\in X \cap \cocl_{M \ba d}(W)$.
    %Since $r^*_{M \ba d}(X) \ge 3$,
    Since $|X| \ge 4$ and $X$ does not contain a triad,
    $X$ contains a circuit, but
    $x\notin \cl_{M \ba d}(X-x)$,
    so this circuit does not contain $x$.
    %Then $X$ cannot be a circuit and so $X-\{x\}$ contains a circuit of $M\backslash d$.
    Thus $(X-x,\{x\},W)$ is a cyclic $3$-separation, implying that $\co(M \ba d \ba x)$ is not $3$-connected; a contradiction.
    %But this implies that $\co(M\backslash d,x)$ has a $2$-separation $(X-\{x\},W)$.
  \end{slproof}

  \begin{sublemma}
    \label{j2}
    In $M\backslash d$, every element of $X$ is in a triad, and every triad that meets $X$ contains exactly one element of $W$.
  \end{sublemma}

  \begin{slproof}
    It is clear that every element of $X$ is in a triad, as otherwise \cref{templefirstcase} holds.
    Let $T^*$ be a triad that contains some $x\in X$.
    Then $T^* \nsubseteq X$. %, by (a).
    If $\{x\}=T^*\cap X$, then $x\in\cocl_{M \ba d}(W)$, which contradicts \cref{j1}.
    So $|T^* \cap X| = 2$ and $|T^* \cap W| = 1$, as required.
  \end{slproof}

  %Consider a triad that meets $X$ and contains the element $c \in W$.
  We now consider two cases. 
  We begin by analysing the situation where %$X\cup c$ is coclosed.
  there is some element $c \in W$ such that every element of $X$ is in a triad contained in $X \cup c$.

  \begin{sublemma}
    \label{coclosed}
    Suppose there exists some $c \in W$ such that each element in $X$ is in a triad contained in $X \cup c$. Then one of \cref{templecontractpair,templequadcase,templehardcase} holds.
  \end{sublemma}

  \begin{slproof}
    %Assume that $\cocl_{M\backslash d}(X\cup c)=X\cup c$.
    Since $X$ does not contain a triad, every element of $X$ is
    %By \cref{j2}, each element of $X$ is contained
    in a triad with $c$ and exactly one other element of $X$.
    The intersection of any two such triads is $\{c\}$, otherwise $X \cup c$ contains a $4$-element cosegment, implying that $X$ contains a triad.
    It follows that $|X|$ is even, and there is a partition of $X$ into pairs $\{x_i,x_i'\}$ such that $\{x_i,x_i',c\}$ is a triad, for $i \in \seq{|X|/2}$.
    The element $d$ blocks each of these triads, % in $M$,
    by (c), and so this partition of $X$ extends to a collection of cocircuits $\{x_i,x_i',c,d\}$ of $M$.
    Moreover, by the cocircuit elimination axiom and (c), $\{x_i,x_i',x_j,x_j'\}$ is a cocircuit in $M\backslash d$ for any distinct $i,j \in \seq{|X|/2}$.

    \begin{subsublemma}
      \label{subby0}
      Either \cref{templecontractpair} or \cref{templehardcase} holds, or, for each $x\in X$, there is a $4$-element circuit of $M$ containing $\{x,c,d\}$.
    \end{subsublemma}

    \begin{sslproof}
      Let $x \in X$.
      The matroid $M\backslash d/x$ is $3$-connected. % for all $x\in X$.
      As $|W| \ge 3$ and
      $c \in \cocl_{M \ba d/x}(X-x)$, it follows from \cref{exactSeps2,openVertSep2} that $\co(M \ba d/x\ba c)$ is not $3$-connected.
      By Bixby's Lemma, $\si(M\backslash d/x/c)$ is $3$-connected. % for all $x\in X$ unless $c$ is in a triangle of $M\backslash d/x$.
      But 
      %such a triangle would imply that $X\cup c$ is not fully closed in $M\backslash d$.
      if $c$ is in a triangle~$T$ of $M\backslash d/x$, then $T$ meets $X$ and $W-c$.
      Let $T=\{x',c,w\}$ with $x' \in X$ and $w \in W-c$.
      As $M \ba d$ has at least one triad contained in $(X-x) \cup c$ that also meets $T$, by orthogonality $\{x',c\}$ is contained in a triad contained in $(X-x) \cup c$.
      So \ref{templehardcase} holds if $c$ is in a triangle of $M \ba d /x$.
      Thus we may assume that $M\backslash d/x/c$ is $3$-connected. % for all $x\in X$.
      Now, if $M/x/c$ is $3$-connected, then \cref{templecontractpair} holds.
      As $x$ was chosen arbitrarily, $d$ is in a triangle with every element of $X$ in $M/c$.
      %Thus, for all $x\in X$, there is a $4$-element circuit $C_x$ of $M$ containing $\{x,c,d\}$.
      Since each element of $X$ is not in a triangle, \cref{subby0} follows.
      %This completes the proof of \cref{subby0}.
    \end{sslproof}

    Suppose $|X| = 4$.  Then $X$ is a $3$-separating cocircuit in $M \ba d$.
    It follows that $X$ is also a circuit, so
    %Now if $|X| = 4$, then
    $X$ is a quad of $M \ba d$, and \ref{templequadcase} holds by \cref{subby0}.
    Thus, in what follows, we may assume that $|X| \ge 6$.
    We also assume that \cref{templecontractpair} does not hold.

    \begin{subsublemma}
      \label{subby2}
      If %$|X| \ge 6$ and
      %Assuming \ref{templehardcase} doesn't hold, if
      $\{a,b\}$ is contained in a $4$-element circuit of $M \ba d$ for $a\in\{x_i,x_i'\}$ and $b\in\{x_j,x_j'\}$ where $i\neq j$,
      then this circuit is $\{x_i,x_i',x_j,x_j'\}$.
    \end{subsublemma}

    \begin{sslproof}
      Suppose $a\in\{x_1,x_1'\}$, $b\in\{x_2,x_2'\}$, and $\{a,b\}$ is contained in a $4$-element circuit $C$ of $M \ba d$ that is not $\{x_1,x_1',x_2,x_2'\}$.
      It follows from orthogonality that %this circuit must contain $c$, and, consequently, one of $\{x_3,x_3'\}$ as well.
      $c \in C$.
      %The claim now holds if $|X| = 4$, %since $X \cup c$ is closed,
      %otherwise \ref{templehardcase} holds,
      %so we may assume that $|X| \ge 6$.
      By orthogonality, $C$ meets $\{x_3,x_3'\}$.
      But now %$c \in \cl_{M \ba d}(X)$, contradicting the fact that $c\in\cocl_{M\ba d}(W-c)$.
      $c \in \cl_{M \ba d}(X) \cap \cocl_{M \ba d}(X)$, contradicting \cref{gutsstayguts2}.
    \end{sslproof}

    \begin{subsublemma}
      \label{subby3}
      If $|X|=6$, then
      %either \cref{templecontractpair} holds or
      $r^*_{M\backslash d}(X)=4$.
    \end{subsublemma}

    \begin{sslproof}
      Clearly $X=\{x_1,x_1',x_2,x_2',x_3,x_3'\}$ and $r_{M^*/d}(X)\in\{3,4\}$. Assume that $r_{M^*/d}(X)=3$.
      Then, as $X$ is $3$-separating in $M^*/d$, and $W$ is closed, $W$ is a hyperplane and $X$ is a rank-$3$ cocircuit in $M^*/d$.
      Take any $a\in\{x_1,x_1'\}$ and $b\in\{x_2,x_2'\}$.
      By \cref{r3cocirc2}, the matroid $\co(M^*/d\backslash a\ba b)$ is $3$-connected.
      Since $X$ is a cocircuit of $M^* /d$, 
      \cref{subby2} implies that $M^*/d\backslash a\ba b$ has no series pairs, thus $M^*/d\backslash a\ba b$ is $3$-connected. 
      It follows that %either \cref{templecontractpair} holds, or
      since \cref{templecontractpair} does not hold,
      there exists a $4$-element cocircuit $C_{ab}$ of $M^*$ containing $\{a,b,d\}$ for each $a\in\{x_1,x_1'\}$ and $b\in\{x_2,x_2'\}$.

      Consider $C_{x_1x_2}$. This cocircuit meets the circuit $\{c,d,x_3,x_3'\}$, and so, by orthogonality,
      $C_{x_1x_2}\subseteq X\cup\{c,d\}$ with $C_{x_1x_2}\cap\{x_1,x_1',x_2,x_2'\}=\{x_1,x_2\}$.
      %Similarly, $C_{x_1x_2'}\subseteq X\cup\{c,d\}$ %with $C_{x_1x_2'}\cap\{x_1,x_1',x_2,x_2'\}=\{x_1,x_2'\}$
      %and $C_{x_1'x_2}\subseteq X\cup\{c,d\}$, % with $C_{x_1'x_2}\cap\{x_1,x_1',x_2,x_2'\}=\{x_1',x_2\}$.
      %and these three cocircuits are distinct.
      Similarly, $C_{x_1x_2'}$ and $C_{x_1'x_2}$ are cocircuits contained in $X\cup\{c,d\}$, and of these three cocircuits, only $C_{x_1x_2'}$ contains $x_2'$, and only $C_{x_1'x_2}$ contains $x_1'$.
      Now
      \begin{align*}
        E(M^*)-(C_{x_1x_2}\cup C_{x_1x_2'}\cup C_{x_1'x_2})
      \end{align*}
      is a flat in $M^*$ of rank at most $r(M^*)-3$, and so
      \begin{align*}
        r_{M^*}\big(E(M^*)-(X\cup\{c,d\})\big)\leq r(M^*)-3.
      \end{align*}
      But then $\lambda(X\cup\{c,d\})\leq r_{M^*}(X\cup\{c,d\})-3=1$, contradicting the fact that $M$ is $3$-connected.
      We conclude that $r^*_{M\backslash d}(X)=4$.
    \end{sslproof}

    \begin{subsublemma}
      \label{subby1}
      $\{x_1,x_1',x_2,x_2'\}$ is a circuit of $M\backslash d$. % for each $i\neq j$.
    \end{subsublemma}

    \begin{sslproof}
      Suppose $|X|=6$ and that $\{x_1,x_1',x_2,x_2'\}$ is independent.
      %Now $r_{M\backslash d}(X)= \lambda_{M \ba d}(X) + |X| - r^*_{M \ba d}(X) = 4$, by \ref{subby3},
      Now $r_{M\backslash d}(X)= 4$, by \cref{subby3} and since $X$ is $3$-separating in $M \ba d$.
      %Since $r_{M\backslash d}(X)=4$, a fact which follows from~\ref{subby3}.
      Then $x_3\in\cl_{M\backslash d}(\{x_1,x_1',x_2,x_2'\})$.
      But this contradicts that $\{x_3,x_3',c\}$ is a triad in $M\backslash d$. %So the result holds in the case that $|X|=6$ and

      So we may assume that $|X|\geq 8$. Again, suppose that $\{x_1,x_1',x_2,x_2'\}$ is independent in $M\backslash d$.
      By \cref{subby2}, each element $b\in\{x_2,x_2'\}$ is not contained in a triangle of $M \ba d / x_1$.
      Thus, the triad $\{x_2,x_2',c\}$ of $M \ba d / x_1$ is not contained in a $4$-element fan. 
      It follows, by Tutte's Triangle Lemma, that either $M \ba d / x_1 / x_2$ or $M \ba d / x_1 / x_2'$ is $3$-connected.
      Assume without loss of generality that $M \ba d / x_1 / x_2$ is $3$-connected.
      Now $M$ has a $4$-element circuit $C_1=\{x_1,x_2,d,\alpha\}$ for some $\alpha$, since $M/x_1/x_2$ is not $3$-connected.
      As $\{c,d,x_3,x_3'\}$ and $\{c,d,x_4,x_4'\}$ are cocircuits of $M$, we deduce that $\alpha =c$, by orthogonality.
      By repeating this argument in $M \ba d / x_1'$, we obtain a distinct circuit $C_2$ of $M$ which is either $\{c,d,x_1',x_2\}$ or $\{c,d,x_1',x_2'\}$.
      By circuit elimination on $C_1$ and $C_2$, there is a circuit contained in $\{x_1,x_2,x_1',x_2',c\}$.
      By orthogonality with $\{c,d,x_3,x_3'\}$, and since no element in $X$ is contained in a triangle of $M$, the circuit is $\{x_1,x_2,x_1',x_2'\}$; a contradiction.
    \end{sslproof}

    It now follows from \cref{subby1} that $\{x_1,x_2,x_1',x_2'\}$ is $3$-separating in $M \ba d$.  But $|X| \ge 6$, contradicting the minimality of $X$.
    This completes the proof of \cref{coclosed}.
  \end{slproof}

  We now turn our attention to the case where, for every $c \in W$, some element of $X$ is not in a triad of $M \ba d$ that is contained in $X \cup c$.
  Recall that we are under the assumption that (i) does not hold.

  \begin{sublemma}
    \label{notcoclosed}
    Suppose that for each $c \in W$, there is some element $x \in X$ such that $x$ is not in a triad of $M \ba d$ contained in $X \cup c$.
    Then either \cref{templecontractpair} or \cref{templehardcase} holds.
  \end{sublemma}

  \begin{slproof}
    We start by showing the following:

    \begin{subsublemma}
      \label{jj0}
      Let $c$ and $c'$ be distinct elements in $W$ such that there are two triads of $M \ba d$ that meet $X$, one containing $c$, and the other containing $c'$.
      Then either \cref{templecontractpair} holds, or there is a $4$-element circuit of $M$ containing $\{d,c,c'\}$.
    \end{subsublemma}
    \begin{sslproof}
      Let $T_c$ and $T_{c'}$ be the triads containing $c$ and $c'$ respectively.
      By \cref{j2}, $T_c-c \subseteq X$, so $(X, \{c\}, W-c)$ is a cyclic $3$-separation of $M \ba d$.
      By \cref{orthogVertSep}, either $M \ba d/c$ is $3$-connected, or $c$ is in a triangle that meets $X$.  But each $x \in X$ is not in a triangle, so $M \ba d/c$ is $3$-connected.
      Now $T_{c'}$ is a triad of $M \ba d/c$, so, similarly, $(X,\{c'\},W-\{c,c'\})$ is a vertical $3$-separation.
      By \cref{orthogVertSep} again, if $M \ba d/c/c'$ is not $3$-connected, then $\{c,c'\}$ is contained in a $4$-element circuit of $M \ba d$ that meets $X$ and $W - \{c,c'\}$. 
      But this contradicts \cref{j1}, which says that $W$ is fully closed.
      So $M \ba d/c/c'$ is $3$-connected.
      Now, either \cref{templecontractpair} holds, or there exists some $\alpha$ such that $\{d,c,c',\alpha\}$ is a $4$-element circuit of $M$.
    \end{sslproof}

    \begin{subsublemma}
      \label{jjnew0}
      Either \cref{templecontractpair} holds, or, for each $c \in W$ in a triad $T_c^*$ that meets $X$, and each $x \in X - T_c^*$, there is a $4$-element circuit of $M$ containing $\{x,c\}$.
    \end{subsublemma}
    \begin{sslproof}
      Let $c$ be an element in a triad $T^*$ that meets $X$, and consider the $3$-connected matroid $M \ba d/x$ for any $x \in X - T^*$.
      Since $T^*$ is a triad in $M \ba d/x$, and $T^*-c \subseteq X$ by \cref{j2}, $c \in \cocl_{M \ba d/x}(X-x)$, and hence $(X-x, \{c\}, W-c)$ is a cyclic $3$-separation of $M \ba d/x$.
      Thus $\si(M\ba d/x/c)$ is $3$-connected, by Bixby's Lemma. 
      Suppose there is no $4$-element circuit containing $\{x,c\}$.
      If $c$ is in a triangle, then this triangle meets $X$; a contradiction.
      Since neither $x$ nor $c$ is in a triangle, $M/x/c$ is $3$-connected. 
      Thus \cref{templecontractpair} holds.
    \end{sslproof}

    We now assume that \cref{templecontractpair} does not hold.
    Let $W' = \cocl_{M \ba d}(X)-X$.  Observe that, for any $c \in W'$, the partition $(X \cup (W'-c), \{c\}, W-W')$ is a cyclic $3$-separation, so %$c \notin \cl(\cocl_{M \ba d}(X)-c)$.
    $c \notin \cl(X \cup (W' - c))$.
    We use this often in what follows.

    \begin{subsublemma}
      \label{jjnew1}
      There are distinct elements $c_1,c_2 \in W$ such that every element $x \in X$ is in a triad of $M \ba d$ contained in $X \cup \{c_1,c_2\}$.
    \end{subsublemma}
    \begin{sslproof}
      Suppose \cref{jjnew1} does not hold.
      Let $T^*_1$ and $T^*_2$ be triads of $M \ba d$ that meet $X$, with $c_1 \in T^*_1$ and $c_2 \in T^*_2$ for distinct $c_1,c_2 \in W$, and let $x \in X-(T^*_1 \cup T^*_2)$ where $x$ is in a triad~$T_3^*$ of $M \ba d$ and there is an element $c_3 \in T_3^* \cap (W-\{c_1,c_2\})$.
      Note that $T_i^*-c_i \subseteq X$ for each $i \in \{1,2,3\}$, by \cref{j2}.
      We may assume that for distinct $i,j \in \{1,2,3\}$, the set $T^*_i \cup T^*_j$ is not a cosegment, for otherwise we can let $c_i = c_j$; in particular, $|T^*_i \cap T^*_j| \le 1$.

      By \cref{jjnew0}, there are $4$-element circuits of $M$ containing $\{x,c_1\}$ and $\{x,c_2\}$.
      Suppose neither of these circuits contains $d$.
      If $T_1^*$ and $T_3^*$ are disjoint, then, by orthogonality, the $4$-element circuit containing $\{x,c_1\}$ is contained in $X \cup \{c_1,c_3\}$, so $c_1 \in \cl(X \cup c_3)$; a contradiction.
      So $T^*_3$ meets $T^*_1$ and, similarly, $T^*_3$ meets $T_2^*$.
      Since $|T^*_1 \cap T^*_2| \le 1$, observe that $T_3^* = \{x,x',c_3\}$ for some $x' \in X-x$ such that $T_1^* \cap T_2^* \cap T_3^* = \{x'\}$.
      Then, by orthogonality, the circuit containing $\{x,c_1\}$ is contained in $X \cup \{c_1,c_2,c_3\}$, so $c_1 \in \cl(X \cup \{c_2,c_3\})$; a contradiction.
      So we may assume that $M$ has a $4$-element circuit containing $\{x,c_1,d\}$.

      Now, for some choice of $\{c',c''\} = \{c_2,c_3\}$, the matroid $M$ has $4$-element circuits $\{x,c_1,d,\beta\}$ and $\{d,c_1,c',\alpha\}$ where $\beta \neq c'$, by \cref{jj0}.
      By circuit elimination, $\{x,c_1,c',\alpha,\beta\}$ contains a circuit.  It follows that $\{\alpha,\beta\} \nsubseteq X$, otherwise $c' \in \cl(X \cup c_1)$, $c_1 \in \cl(X \cup c')$, or $x$ is in a triangle of $M$.  Moreover, $\{\alpha,\beta\} \nsubseteq W$, otherwise $W$ is not closed, $c_1 \notin \cocl_{M \ba d}(X)$, or $c' \notin \cocl_{M \ba d}(X)$.
      So $\{\alpha,\beta\}$ meets $X$ and $W$.

      By orthogonality, $\alpha \in X \cup c''$.
      Suppose that $\alpha = c''$.  Then $\beta \in X$, and $\{x,c_1,c',c'',\beta\}$ contains a circuit.
      Since this circuit meets $\{c_1,c',c''\}$, we obtain a contradiction.
      So $\alpha \in X$. 
      Now $X \cup c''$ also contains a triad of $M \ba d$, so there is a $4$-element circuit $\{d,c_1,c'',\gamma\}$ of $M$, by \cref{jj0}, where $\gamma \in X \cup c'$, by orthogonality.
      By circuit elimination with $\{d,c_1,c',\alpha\}$, we again obtain a contradictory circuit contained in $X \cup \{c_1,c',c''\}$ and meeting $\{c_1,c',c''\}$.
    \end{sslproof}

    Let $c,c' \in W$ be distinct elements such that every element of $X$ is in a triad of $M \ba d$ contained in $X \cup \{c,c'\}$.

    \begin{subsublemma}
      \label{jjnew2}
      If $|X| > 4$, then \cref{templehardcase} holds.
    \end{subsublemma}
    \begin{sslproof}
      Suppose that $|X| > 4$.
      Then there are at least three distinct triads contained in $X \cup \{c,c'\}$, and it follows that, up to labels, there are distinct triads $T^*_1$ and $T_2^*$ containing $c$.
      Let $T^*_1 = \{x_1,x_1',c\}$ and $T^*_2 = \{x_2,x_2',c\}$.
      Since $X$ does not contain a triad of $M \ba d$, the elements $x_1,x_1',x_2,x_2'$ are distinct.
      There exists $x_3  \in X - \{x_1,x_1',x_2,x_2'\}$ such that $x_3$ is not in a triad contained in $X \cup c$.  It follows that $\{x_3,c'\}$ is contained in a triad~$T_3^*$.

      At least one of $T_1^*$ and $T_2^*$ does not meet $T_3^*$, so we may assume that $T_1^* \cap T_3^* = \emptyset$.
      By \cref{jjnew0}, there is a $4$-element circuit containing $\{x_1,c'\}$.
      If this circuit does not contain $d$, then, by orthogonality, $c' \in \cl(X \cup c)$; a contradiction.
      So the circuit is $\{x_1,c',d,p\}$, where $p \in T_2^*$, by orthogonality.
      Similarly, there is a circuit $\{x'_1,c',d,p'\}$, where $p' \in T_2^*$.
      By \cref{jj0}, $M$ also has a $4$-element circuit $\{d,c,c',\alpha\}$ for some $\alpha \in E(M) - \{d,c,c'\}$.

      We consider two cases depending on whether or not $T_2^* \cap T_3^* = \emptyset$.
      First consider the case where $T_2^* \cap T_3^* = \emptyset$.
      Suppose $\{x_1,c,c',d\}$ is not a circuit.
      Then, by circuit elimination on the circuits $\{x_1,c',d,p\}$ and $\{d,c,c',\alpha\}$, there is a circuit contained in $\{x_1,p,c,c',\alpha\}$, where $p \in T_2^*$ and $\alpha \neq d$. 
      If $\alpha \in X$, then either $c \in \cl(X \cup c')$ or $c' \in \cl(X \cup c)$; a contradiction.
      So $\alpha \in W$.
      Now, by orthogonality with $T_3^*$, the circuit does not contain $c'$.
      It follows that $\{x_1,p,c,\alpha\}$ is a circuit for $p \in \{x_2,x_2'\}$, so \cref{templehardcase} holds.
      So we may also assume that $\{x_1,c,c',d\}$ is a circuit.
      By the same argument with $x_1'$ in the role of $x_1$, we deduce that $\{x_1',c,c',d\}$ is a circuit.
      But then $\{x_1,x_1',c,c'\}$ contains a circuit; a contradiction.

      Now we may assume that $T_2^* \cap T_3^* = \{x_2\}$, so $T_3^* = \{x_2,x_3,c'\}$.
      By \cref{jjnew0}, $M$ has a $4$-element circuit $\{x_3,c,\beta,z\}$, for some $\{\beta,z\} \subseteq E(M) - \{x_3,c\}$.
      By orthogonality, $\{\beta,z\}$ meets $\{x_1,x_1',d\}$ and $\{x_2,x_2',d\}$.
      Thus, if $d \notin \{\beta,z\}$, then $c \in \cl(X)$; a contradiction.
      So let $z = d$.
      Now $\{d,c,\beta,x_3\}$ and $\{d,c,c',\alpha\}$ are circuits, so if $\{\beta,x_3\} \neq \{c',\alpha\}$, then, by circuit elimination, there is a circuit contained in $\{x_3,c,c',\alpha,\beta\}$, where $d \notin \{\alpha,\beta\}$.  If this circuit contains $c$, then, by orthogonality, $\{\alpha,\beta\} \subseteq \{x_1,x_1',x_2,x_2'\}$.
      But now $c \in \cl(X \cup c')$; a contradiction.
      So either $\{d,c,c',x_3\}$ or $\{x_3,c',\alpha,\beta\}$ is a $4$-element circuit of $M$.

      If $\{d,c,c',x_3\}$ is a circuit, then, by circuit elimination with $\{x_1,c',d,p\}$, the set $\{x_1,x_3,p,c,c'\}$ contains a circuit.
      Since each element in $X$ is not in a triangle, the circuit has at least four elements.
      Thus $c \in \cl(X \cup c')$ or $c' \in \cl(X \cup c)$; a contradiction.
      So $\alpha \neq x_3$ and $\{x_3,c',\alpha,\beta\}$ is a $4$-element circuit of $M$, where $d \notin \{\alpha,\beta\}$.

      Observe that $\{\alpha,\beta\}$ meets both $X$ and $W$, by \cref{j1} and since $c' \notin \cl(X)$.  
      Let $\{\alpha,\beta\} \cap X = \{x_4\}$. 
      By orthogonality between $\{x_3,c',\alpha,\beta\}$ and either $T_1^*$ or $T_2^*$, we have $x_4 \notin \{x_1,x_1',x_2,x_2'\}$. 
      Moreover $x_4 \neq x_3$, since $x_3 \notin \{\alpha,\beta\}$.

      Now $\{x_4,c\}$ contains a circuit by \cref{jjnew0}, and this circuit contains $d$, by orthogonality with $T_1^*$ and $T_2^*$, and since $c \notin \cl(X)$.
      By orthogonality with $T_3^*$, either $\{x_4,c,d,c'\}$ or $\{x_4,c,d,x_j\}$ is a $4$-element circuit for $j \in \{2,3\}$.
      Recall that $\{x_1,c',d,p\}$ is a circuit for some $p \in T_2^*$.
      By circuit elimination, either $\{x_1,p,x_4,c,c'\}$ or $\{x_1,p,x_j,x_4,c,c'\}$ contains a circuit, respectively.
      In the former case, $c \in \cl(X \cup c')$ or $c' \in \cl(X \cup c)$; a contradiction.
      In the latter case, the only other possibility is that $\{x_1,p,x_j,x_4\}$ is a $4$-element circuit contained in $X$.
      But then this set intersects $T_1^*$ in a single element, contradicting orthogonality.
    \end{sslproof}

    \begin{subsublemma}
      \label{jjnew3}
      $|X| \neq 4$.
    \end{subsublemma}
    \begin{sslproof}
      Let $X = \{x_1,x_1',x_2,x_2'\}$, where
      %Now, up to labels,
      $\{x_1,x_2,c\}$ and $\{x_1',x_2',c'\}$ are triads of $M \ba d$.
      Since $\lambda_{M \ba d}(X) = 2$, we have $r(X) + r^*_{M \ba d}(X) = 6$, so $X$ is a quad in $M \ba d$.
      By \cref{jjnew0}, there are $4$-element circuits containing $\{x,c'\}$ for $x \in \{x_1,x_2\}$, and $4$-element circuits containing $\{x',c\}$ for $x' \in \{x_1',x_2'\}$.  It follows, by orthogonality and since $c \notin \cl(X \cup c')$ and $c' \notin \cl(X \cup c)$, that any such %$4$-element
      circuit must contain $d$.

      Suppose that $X$ is a cocircuit of $M$.
      %If $X$ is a cocircuit,
      Then each of the $4$-element circuits containing $\{x,c,d\}$ or $\{x,c',d\}$ for $x \in X$ is contained in $X \cup \{c,d\}$ or $X \cup \{c',d\}$, by orthogonality.
      So $M$ has distinct circuits $\{x_1,\alpha,c',d\}$ and $\{x_1',\alpha',c,d\}$ for some $\alpha,\alpha' \in X$.
      %The circuits are either $\{x_1,x_2,c',d\}$ and $\{x_1',x_2',c,d\}$, or $\{x_1,x_1',c',d\}$ and $\{x_2,x_2',c,d\}$. (there cannot be any more meeting $X$ and containing $d$ and either $c$ or $c'$).
%
      Now these two circuits, together with the circuit $X$, imply that $r(X \cup \{c,c',d\})\le 4$.
      But $r^*(X \cup \{c,c',d\}) \le r^*(X) + 1 = 4$, so $\lambda_M(X \cup \{c,c',d\}) \le 1$; a contradiction, since $|W| \ge 4$.

      So we may assume that $d \in \cocl(X)$, and $X \cup d$ is a cocircuit of $M$.
      By \cref{jj0}, $\{x_0,c,c',d\}$ is a circuit for some $x_0 \in X$.
      Recall that each element in $X$ is in a $4$-element circuit containing $d$ and either $c$ or $c'$.
      Suppose one of these circuits is contained in $X \cup \{c,c',d\}$.
      Then, by circuit elimination, there is a circuit of $M \ba d$ contained in $X \cup \{c,c'\}$, and containing at most three elements of $X$; a contradiction.
      So, for each $x \in X$, there is a $4$-element circuit containing $x$, $d$, either $c$ or $c'$, and an element in $W$. % cannot be in any other $4$-element circuit contained in $X \cup \{c,c',d\}$.

      Without loss of generality, suppose that $\{x_1,c,c',d\}$ is a circuit.  Then, by orthogonality, $\{x_2,c',d,y\}$, $\{x_1',c,d,y_1'\}$, and $\{x_2',c,d,y_2'\}$ are circuits for some $y \in W$ and $y_1',y_2' \in W-c$.
      %Clearly $c \notin \{y_1',y_2'\}$.
      Note that if $c' \in \{y_1',y_2'\}$, then, by circuit elimination on $\{x_1,c,c',d\}$ and either $\{x_1',c,d,y_1'\}$ or $\{x_2',c,d,y_2'\}$, there is a circuit contained in $\{x_1,x_1',x_2',c,c'\}$; a contradiction.

      Since $X \cup d$ and $\{x_1',x_2',c',d\}$ are cocircuits of $M$, there is a cocircuit $C^*$ contained in $X \cup c'$, by cocircuit elimination. %, and this cocircuit certainly contains $c$.
      %Likewise, there is a cocircuit contained in $X \cup c'$ that contains $c'$.
      Since $c' \notin \{y_1',y_2'\}$, orthogonality with the circuit $\{x_1',c,d,y_1'\}$ implies that $x_1' \notin C^*$, and orthogonality with the circuit $\{x_2',c,d,y_2'\}$ implies that $x_2' \notin C^*$.  But then $\{x_1,x_2,c'\}$ contains a cocircuit; a contradiction.
    \end{sslproof}
    It now follows from \cref{jjnew2,jjnew3} that \cref{notcoclosed} holds.
  \end{slproof}

  With that, the proof of \cref{temple} is complete.
\end{proof}

\subsection*{Retaining an $N$-minor}

In this section, we consider specific outcomes of \cref{temple}, relative to a cyclic $3$-separation $(Y,\{d'\},Z)$ for which a $3$-connected $N$-minor is known to lie primarily in $Z$, with the goal of finding an $N$-detachable pair.

For the entirety of the section we work under the following assumptions.
Let $M$ be a $3$-connected matroid with an element $d$ such that $M\ba d$ is $3$-connected.
Let $N$ be a $3$-connected minor of $M$, % and $M \ba d$,
where every triangle or triad of $M$ is \unfortunate, and $|E(N)| \ge 4$.
Suppose that $M\ba d$ has a cyclic $3$-separation $(Y, \{d'\}, Z)$ with $|Y| \ge 4$, where $M\ba d \ba d'$ has an $N$-minor with $|Y \cap E(N)| \le 1$.
Let $X$ be a $3$-separating subset of $Y$ with $|X|\geq 4$, where
$X$ does not contain a triad of $M\ba d$, and, for each $x \in X$,
both $\co(M\ba d\ba x)$ and $M\ba d/x$ are $3$-connected,
and $x$ is doubly $N$-labelled in $M \ba d$.
Let $W = E(M \ba d)-X$, and observe that $r^*_{M \ba d}(W) \ge 3$.

Since every triangle or triad of $M$ is \unfortunate, each element in $X$ is not in a triangle or triad of $M$, by \cref{freegrounded}.
In particular, note that as $|X| \ge 4$ and $X$ does not contain any triangles, $r(X) \ge 3$ and therefore $r(M\ba d) \ge 4$.

We now consider the case where \cref{temple}\ref{templehardcase} holds.

\begin{lemma}
  \label{altar}
  Suppose that there are elements $c \in \cocl_{M \ba d}(X) \cap W$ and $w \in W-c$ such that 
  \begin{enumerate}[label=\rm(\alph*)]
    \item $\{x_1,x_2,c,w\}$ is a $4$-element circuit of $M \ba d$ where $\{x_1,x_2\} \subseteq X$, and $x_1$ and $x_2$ are in distinct triads of $M \ba d$ contained in $X \cup c$; and
    %\item $r(W-\{c,w\}) \ge 2$ and $r^*_{M \ba d}(W-\{c,w\}) \ge 2$.
    \item $|W-\{c,w\}| \ge 2$ and $W$ contains a circuit.
  \end{enumerate}
  Then either
  \begin{enumerate}
    \item $M$ has an $N$-detachable pair, or\label{altari}
    \item there exists a set $Q\subseteq W \cup d$ with $\{c,d\}\subseteq Q$ such that $X\cup Q$ is a \spider\ of $M$ with associated partition $\{X,Q\}$.\label{altarii}
  \end{enumerate}
\end{lemma}
\begin{proof}
  Note that $(X,\{c\},W-c)$ is a cyclic $3$-separation of $M \ba d$.
  Let $x_1',x_2' \in X$ be such that $\{x_1,x_1',c\}$ and $\{x_2,x_2',c\}$ are distinct triads of $M \ba d$.
  Observe that the elements $x_1,x_1',x_2,x_2'$ are distinct, since otherwise the union of the two triads is a cosegment, in which case $X$ contains a contradictory triad.
  Moreover, by cocircuit elimination on the triads $\{x_1,x_1',c\}$ and $\{x_2,x_2',c\}$ of $M \ba d$, and since $X$ does not contain a triad, $\{x_1,x_2,x_1',x_2'\}$ is a cocircuit of $M \ba d$.
  By \cref{contractdistincttriads}, $M \ba d /x_1/x_2$ has an $N$-minor.
  Since $\{c,w\}$ is a parallel pair in this matroid, $M \ba d \ba w/x_1/x_2$ and $M \ba d \ba c/x_1/x_2$ have $N$-minors.

  %We first use that $M \ba d \ba w$ has an $N$-minor.
  We claim that $\co(M \ba d \ba w)$ is $3$-connected.
  Since $X \cup c$ is exactly $3$-separating, \cref{aggregatelemma} implies that $X \cup \{c,w\}$ is also exactly $3$-separating, and $w \in \cl(W-\{c,w\})$.
  If $r(W-\{c,w\}) \ge 3$, then $(X \cup c, \{w\}, W-\{c,w\})$ is a vertical $3$-separation, and, by Bixby's Lemma, $\co(M \ba d \ba w)$ is $3$-connected, as required.
  On the other hand, if $r(W-\{c,w\}) = 2$, then $W-c$ is a segment.  If $|W-c| \ge 4$, then $M \ba d \ba w$ is $3$-connected by \cref{rank2Remove2}.  So we may assume that $W-c$ is a triangle. But $W$ contains a cocircuit of $M \ba d$ that contains $c$, and $c$ is not in a triad, as $c$ is $N$-deletable.  So $W$ is a $4$-element cocircuit of $M \ba d$. Then $\co(M \ba d \ba w)$ is $3$-connected by \cref{r3cocirc2}, thus proving the claim.
  Since $M \ba d \ba w$ has an $N$-minor,
  either $\{d,w\}$ is an $N$-detachable pair and, in particular, \cref{altari} holds, or $w$ is in a triad of $M \ba d$.

  So we may assume that $w$ is in a triad $T^*$ of $M \ba d$.
  By orthogonality, $T^*$ meets $\{x_1,x_2,c\}$.
  Recall that $w \in \cl(W-\{c,w\})$, so that
    %Since $X \cup c$ and $X \cup \{c,w\}$ are exactly $3$-separating, 
    %and $w \in \cl_{M \ba d}(X \cup c)$, it follows that
  $w \notin \cocl_{M \ba d}(X \cup c)$.
  If, for some $x \in X$, we have $x \notin \cl(X-x)$, then $x \in \cocl_{M \ba d}(X-x)$ by \cref{exactSeps2}, in which case $(X-x,\{x\},W)$ is a cyclic $3$-separation of $M \ba d$.
  But then $\co(M \ba d \ba x)$ is not $3$-connected; a contradiction.  So each $x \in X$ is in a circuit contained in $X$.
  Thus, it follows from orthogonality that if $T^*$ meets $X$, then $w \in \cocl_{M \ba d}(X)$; a contradiction.
  So $(X \cup \{c,w\}) \cap T^* = \{c,w\}$.

  Let $T^*= \{c,w,c'\}$, where $c' \in W-\{c,w\}$.
  Recall that $M \ba d \ba c/x_1/x_2$ has an $N$-minor, and observe that $\{w,c'\}$ is a series pair in this matroid.
    %Thus $M \ba d/ c' /x_1$ and $M \ba d/ c' /x_2$ have an $N$-minor.
  Thus each of $M/c'/x_1$ and $M/c'/x_2$ has an $N$-minor, and,
   as $c'$ (and $w$) are $N$-contractible, $c'$ is not in a triangle (and neither is $w$).
  %Since $|W-c| \ge 3$, it follows that $r(W-c) \ge 3$.
  Thus $r(W-\{c,w\})=r(W-c) \ge 3$.
  Observe that $X \cup \{c,w\}$ is exactly $3$-separating and $|W-\{c,w\}| \ge 3$, so the dual of \cref{aggregatelemma} implies that $X \cup \{c,w,c'\}$ is also exactly $3$-separating and $c' \in \cocl_{M \ba d}(W-\{c,w,c'\})$.

  \begin{sublemma}
    \label{contractcandidates}
    $M \ba d/c'/x$ is $3$-connected for all $x \in \{x_1,x_2,x_1',x_2'\}$.
  \end{sublemma}
  \begin{slproof}
    Let $x \in \{x_1,x_2,x_1',x_2'\}$.
    Since $c' \in \cocl_{M \ba d}(W-\{c,w,c'\})$, we have that
    %Observe that $X \cup \{c,w\}$ is exactly $3$-separating and $|W-\{c,w\}| \ge 3$, so the dual of \cref{aggregatelemma} implies that $X \cup \{c,w,c'\}$ is also exactly $3$-separating and $c' \in \cocl_{M \ba d}(W-\{c,w,c'\})$.  Hence
    $r^*_{M \ba d}(W-\{c,w,c'\}) = r^*_{M \ba d}(W-\{c,w\}) \ge 2$.
    First we show that \cref{contractcandidates} holds when we have equality.
    In this case, $W-\{c,w\}$ is a cosegment in $M \ba d / x$.
    %Since $r^*_{M \ba d/x}(W-\{c,w\})=2$,
    If $|W-\{c,w\}| \ge 4$ then $M \ba d/c'/x$ is $3$-connected by the dual of \cref{rank2Remove2}, as required. 
    On the other hand, if $W-\{c,w\}$ is a triad, then it follows that $W-c$ is a corank-$3$ circuit, since $w$ is not in a triangle, and $\si(M \ba d/c'/x)$ is $3$-connected by applying the dual of \cref{r3cocirc2} in the matroid $M \ba d/x$.
    Moreover, if $\{c',x\}$ is in a $4$-element circuit in $M\ba d$, then, by orthogonality, this circuit meets $W-\{c,w,c'\}$ and $\{c,w\}$, implying $x \in \cl(W)$.
    But this contradicts that $\{x_1,x_1',x_2,x_2'\}$ is a cocircuit of $M \ba d$.

    So we may assume that $r^*(W-\{c,w,c'\}) \ge 3$.
    Since $M \ba d/x$ is $3$-connected, 
    it follows that $((X-x) \cup \{c,w\}, \{c'\}, W-\{c,w,c'\})$ is a cyclic $3$-separation.
    Hence $\si(M \ba d/c'/x)$ is $3$-connected by Bixby's Lemma.
    Recall that $c'$ is $N$-contractible, so it is not in an \unfortunate\ triangle in $M$.
    Thus, if $c'$ is in a triangle~$T$ in $M \ba d/x$, then it is in a $4$-element circuit $T \cup x$ in $M \ba d$.
    As $x$ is in a triad of $M \ba d$ contained in $X \cup c$,
    orthogonality implies that $T \cap ((X-x) \cup c)$ is non-empty.
    Since $c' \notin \cl_{M \ba d / x}((X-x) \cup \{c,w\})$, the triangle~$T$ also contains an element in $W-\{c,w,c'\}$.
    It then follows that $c \notin T$, as otherwise $x \in \cl(W)$, contradicting that $\{x_1,x_1',x_2,x_2'\}$ is a cocircuit of $M \ba d$.
    So $\{x,x',c',w'\}$ is a circuit for some $x' \in X-x$ and $w' \in W-\{c,w,c'\}$.
    But $\{c,w,c'\}$ is also a triad of $M \ba d$, so we obtain a contradiction to orthogonality.
    %Moreover, as $\{c,w,c'\}$ is a triad of $M \ba d$, we have that
    %$\{c,w\} \cap T$ is non-empty, again by orthogonality.
    %$\{c,w\} \cap T$ is non-empty, again by orthogonality.
    %But now $c' \in \cl(X \cup \{c,w\})$, which contradicts that $x \in \cocl(X \cup \{c,w\})$.
    We deduce that $M \ba d/c'/x$ is $3$-connected.
  \end{slproof}

    %Next we show that either $M \ba d /c/ x_1'$ is $3$-connected, or $\{x_1',x_2',c,w\}$ is a circuit.
  \begin{sublemma}
    \label{slcases}
    Either
    \begin{enumerate}[label=\rm(\Roman*)]
      \item $M \ba d /c/ x_1'$ and $M \ba d /c/ x_2'$ are $3$-connected, or\label{slcase1}
      \item $\{x_1',x_2',c,w\}$ is a circuit.\label{slcase2}
    \end{enumerate}
  \end{sublemma}
  \begin{slproof}
    Let $i \in \{1,2\}$.
    Since $M \ba d / x_i'$ is $3$-connected, and $c \in \cocl_{M \ba d}(X) \cap \cocl_{M \ba d}(W-c)$, the matroid $\si(M \ba d/c /x_i')$ is $3$-connected.
    If $c$ is neither in a triangle in $M \ba d/x_1'$, nor in $M \ba d/x_2'$, then \cref{slcase1} holds.
    So, without loss of generality, suppose that $c$ is in a triangle~$T$ in $M \ba d /x_1'$.
    Then $T$ meets $X$ and $W-c$.
    Moreover, as $\{c,w,c'\}$ is a triad in $M \ba d$, the triangle $T$ contains one of $w$ or $c'$.
    But if $c' \in T$, then $c' \in \cl(X\cup c)$, contradicting that $c' \in \cocl_{M \ba d}(W-\{c,w,c'\})$.
    So $T = \{x,c,w\}$ for some $x \in X$.
    Now $\{x_1',x,c,w\}$ is a circuit in $M \ba d$, since $x$ is not in a triangle. 
    As $\{x_1,x_2,c,w\}$ is also a circuit, the set $\{x_1',x,x_1,x_2,c\}$ contains a circuit.  But $c \notin \cl(X)$, and $X$ does not contain a triangle of $M$, so $\{x_1',x,x_1,x_2\}$ is a circuit for some $x \in X-\{x_1,x_2,x_1'\}$.
    By orthogonality, $x = x_2'$.
    This completes the proof of \cref{slcases}.
  \end{slproof}

  \begin{sublemma}
    \label{anextrasl}
    There is a cocircuit $C^*$ of $M$ such that $\{x_1,x_1',x_2,x_2'\} \subseteq C^* \subseteq \{x_1,x_1',x_2,x_2',c\}$.
  \end{sublemma}
  \begin{slproof}
    Observe that $\{x_1,x_1',c,d\}$ and $\{x_2,x_2',c,d\}$ are cocircuits of $M$ so, by cocircuit elimination, $\{x_1,x_1',x_2,x_2',c\}$ contains a cocircuit $C^*$ of $M$.
    Since no element of $X$ is in a triad of $M$, we have $|C^*| \ge 4$.
    Suppose $x_1 \notin C^*$, say.  Then $r^*_M(\{x_1',x_2,x_2',c\}) = 3$.  Due to the cocircuits $\{x_2,x_2',c,d\}$ and $\{x_1,x_1',c,d\}$, it follows that $r^*_M(\{x_1,x_1',x_2,x_2',c,d\}) = 3$, so $r^*_{M\ba d}(\{x_1,x_1',x_2,x_2'\}) = 2$; a contradiction.
    \cref{anextrasl} follows by symmetry.
  \end{slproof}

  \begin{sublemma}
    When \cref{slcases}\cref{slcase1} holds, $M$ has an $N$-detachable pair.
  \end{sublemma}
  \begin{slproof}
    We are now in the case where $M \ba d/c/x_1'$ and $M \ba d/c/x_2'$ are $3$-connected.
    For each $i \in \{1,2\}$, the matroid $M/c/x_i'$ has an $N$-minor, by \cref{contractdistincttriads}, so we may assume that there are elements $\alpha_1,\alpha_2 \in E(M) - \{c,d\}$ such that $\{x_1',\alpha_1,c,d\}$ and $\{x_2',\alpha_2,c,d\}$ are $4$-element circuits of $M$, for otherwise $M$ has an $N$-detachable pair.

    Let $x \in \{x_1,x_2\}$.
    Recall that $M \ba d / c'/ x $ is $3$-connected, by \cref{contractcandidates}.
    Suppose that $M /c'/ x$ is not $3$-connected.
    Then there is a $4$-element circuit of $M$ containing $\{d,x,c'\}$.
    By orthogonality with the cocircuit $C^*$ of \cref{anextrasl}, this circuit intersects $X$ in two elements.
    So there exists $x'' \in X-x$ such that $\{x,x'',c',d\}$ is a circuit of $M$.
    In particular $c' \in \cl(X \cup d)$.

    We work towards showing that $c \in \cl(X \cup d)$.
    Clearly this holds if $\alpha_1 \in X$ or $\alpha_2 \in X$, so we assume that $\alpha_1,\alpha_2 \in W-c$.
    By circuit elimination, there is a circuit of $M \ba d$ contained in $\{x_1',x_2',\alpha_1,\alpha_2,c\}$.
    Suppose that $\{\alpha_1,\alpha_2\} \cap \{w,c'\} = \emptyset$.
    Then, by orthogonality with the triad $\{c,w,c'\}$ of $M \ba d$, and since neither $x_1'$ nor $x_2'$ is in a triangle, we deduce that $\{x_1',x_2',\alpha_1,\alpha_2\}$ is a circuit.
    But this contradicts orthogonality with the cocircuit $\{x_1,x_1',c,d\}$ of $M$.
    So $\{\alpha_1,\alpha_2\} \cap \{w,c'\} \neq \emptyset$.

    Suppose that $\alpha_1 = c'$.
    Then, by circuit elimination on $\{x_1',c',c,d\}$ and $\{x,x'',c',d\}$, there is a circuit contained in $\{x,x_1',x'',c,c'\}$. 
    If this circuit contains $c'$, then $c' \in \cl(X \cup \{c,w\})$; a contradiction.
    Since no element of $X$ is in a triangle, $c$ is in a $4$-element circuit contained in $X \cup c$.  But then $c \in \cl(X)$; a contradiction.
    By symmetry, we deduce that $c' \notin \{\alpha_1,\alpha_2\}$.

    Without loss of generality we may now assume that $\alpha_1 = w$, so $\{x_1',w,c,d\}$ and $\{x,x'',c',d\}$ are circuits of $M$.
    By strong circuit elimination, there is a circuit contained in $\{x,x_1',x'',c,w,c'\}$ that contains $c'$.
    But then $c' \in \cl(X \cup \{c,w\})$; a contradiction.
    We deduce that either $\alpha_1$ or $\alpha_2$ is in $X$, so $c \in \cl(X \cup d)$.

    Thus $r(X \cup \{c,w,c',d\}) = r(X \cup d)$.
    Due to the cocircuits $\{c,w,c',d\}$ and $\{x_1,x_1',c,d\}$, we have $r(W-\{c,w,c'\}) = r(W-w) -2$.
    Since $w \in \cl(W-w)$, 
    \begin{align*}
      \lambda(X \cup \{c,w,c',d\}) &= r(X \cup d) + (r(W-w)-2) - r(M) \\
      &= \lambda(X \cup d) - 2 \le 1;
    \end{align*}
    a contradiction.

    We conclude that $M /c'/ x $ is $3$-connected for each $x \in \{x_1,x_2\}$.
    Since each of these matroids has an $N$-minor, $M$ has an $N$-detachable pair.
  \end{slproof}

  It remains to consider \cref{slcases}\cref{slcase2}; that is, the case where $\{x_1',x_2',c,w\}$ is a circuit.
  Note that
    %Suppose $\{x_1',x_2',c,w\}$ is also a circuit of $M \ba d$, where $\{x_1',x_2'\} \subseteq X$,
  the elements $x_1'$ and $x_2'$ are in distinct triads of $M \ba d$ contained in $X \cup c$.
  Repeating the earlier argument,
  since $M \ba d /x_1'/x_2'$ has an $N$-minor, it follows that
  $M \ba d\ba c/x_1'/x_2'$ has an $N$-minor, and hence
    %we deduce that either $M$ has an $N$-detachable pair, or
    %each of $M \ba d/ c' /x_1'$ and $M \ba d/ c' /x_2'$ has an $N$-minor.
  each of $M / c' /x_1'$ and $M / c' /x_2'$ has an $N$-minor.

  Let $x \in \{x_1,x_1',x_2,x_2'\}$.
  Now, it follows from \cref{contractcandidates} that either $M$ has an $N$-detachable pair, or there is a $4$-element circuit of $M$ containing $\{x,c',d\}$.
  We now consider the fourth element of this circuit.
  By orthogonality with the cocircuit $C^*$ of \cref{anextrasl}, for each $x \in \{x_1,x_1',x_2,x_2'\}$ there exists some $x' \in \{x_1,x_2,x_1',x_2'\}-x$ such that $\{x,x',c',d\}$ is a circuit of $M$.

  Recall that $M \ba d$ has circuits $\{x_1,x_2,c,w\}$ and $\{x_1',x_2',c,w\}$. 
  By orthogonality, any triad contained in $X \cup c$ contains an element in $\{x_1,x_2\}$ and an element in $\{x_1',x_2'\}$.  Thus $|X|=4$.
  By circuit elimination on the circuits $\{x_1,x_2,c,w\}$ and $\{x_1',x_2',c,w\}$, the sets $\{x_1,x_2,x_1',x_2',c\}$ and $\{x_1,x_2,x_1',x_2',w\}$ contain circuits.
  But $c,w \notin \cl(X)$, and no element in $\{x_1,x_2,x_1',x_2'\}$ is in a triangle, so $\{x_1,x_2,x_1',x_2'\}$ is a $4$-element circuit.
  Hence $X$ is a quad in $M \ba d$.

  Let $Q= \{c,w,c',d\}$.
  Observe that since $\{c,w,c'\}$ is a triad of $M$ and $w$ is $N$-deletable in $M \ba d$, this triad is blocked by $d$.  So $Q$ is a cocircuit of $M$.
  We will show that $X \cup Q$ is a \spider\ of $M$ with associated partition $\{X,Q\}$, as illustrated in \cref{spiderlabelling}.

  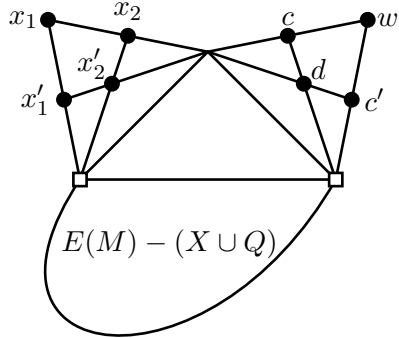
\begin{figure}[tbhp]
    \begin{tikzpicture}[rotate=90,scale=0.85,line width=1pt]
      \tikzset{VertexStyle/.append style = {minimum height=5,minimum width=5}}
      \clip (-2.5,-6) rectangle (3.0,2);
      \node at (-1,-1.4) {$E(M)-(X \cup Q)$};
      \draw (0,0) .. controls (-3,2) and (-3.5,-2) .. (0,-4);

      \draw (0,0) -- (2,-2) -- (0,-4);

      \draw (0,0) -- (2.5,0.5) -- (2,-2);
      \draw (0,0) -- (2.25,-0.75);
      \draw (2,-2) -- (1.25,0.25);

      \draw (0,-4) -- (2.5,-4.5) -- (2,-2);
      \draw (0,-4) -- (2.25,-3.25);
      \draw (2,-2) -- (1.25,-4.25);

      \Vertex[x=1.25,y=0.25,LabelOut=true,L=$x_1'$,Lpos=180]{c1}
      \Vertex[x=2.25,y=-0.75,LabelOut=true,L=$x_2$,Lpos=90]{c2}
      \Vertex[x=2.5,y=0.5,LabelOut=true,L=$x_1$,Lpos=180]{c3}
      \Vertex[x=1.5,y=-0.5,LabelOut=true,L=$x_2'$,Lpos=135]{c4}

      \Vertex[x=1.25,y=-4.25,LabelOut=true,L=$c'$]{c1}
      \Vertex[x=2.25,y=-3.25,LabelOut=true,L=$c$,Lpos=90]{c2}
      \Vertex[x=2.5,y=-4.5,LabelOut=true,L=$w$]{c3}
      \Vertex[x=1.5,y=-3.5,LabelOut=true,L=$d$,Lpos=45]{c4}

      \draw (0,0) -- (0,-4);

      %\SetVertexNoLabel
      %\tikzset{VertexStyle/.append style = {shape=rectangle,fill=white}}
      %\Vertex[x=0,y=0]{a1}
      %\Vertex[x=0,y=-4]{a2}
    \end{tikzpicture}
    \caption{The \spider~$X \cup Q$, when \cref{altar}\ref{altarii} holds.}
    \label{spiderlabelling}
  \end{figure}

  Since $d \in \cocl(Q-d)$, it follows that $d \notin \cocl(X)$, as otherwise $(X,\{d\},W)$ is a cyclic $3$-separation of $M$, contradicting that $M \ba d$ is $3$-connected.
  So $X$ a cocircuit of $M$.
  As $\{x_1,x_2,c,w\}$ and $\{x_1,x_2,c',d\}$ are circuits of $M$, the circuit elimination axiom implies that $\{x_2,c,w,c',d\}$ contains a circuit.  But since $X$ is a cocircuit, $x_2 \notin \cl(Q)$, and it follows that $Q$ is a circuit, since $c$ and $c'$ are $N$-contractible.
  Now $X$ and $Q$ are quads of $M$, and $r(X \cup Q) \le 5$ and $r^*(X \cup Q) \le 5$; it follows that $\lambda(X \cup Q) =2$ and $r(X \cup Q) = r^*(X \cup Q) =5$.

  Suppose $\{x_1,x_1',c',d\}$ is a circuit.
  Then $\cl((Q-w) \cup x_1') = X \cup Q$, so $r(X \cup Q) \le 4$; a contradiction.
  Similarly, $\{x_1,x_2',c',d\}$ is not a circuit.
  We deduce that $\{x_1,x_2,c',d\}$ and $\{x_1',x_2',c',d\}$ are circuits of $M$.

    It remains to show that $\{x_1,x_1',c',w\}$ and $\{x_2,x_2',c',w\}$ are cocircuits.
    Since $X$ and $Q$ are disjoint quads in $M$, and hence no element in $X$ is in the coclosure of $Q$, it follows from \cref{r3cocirc2} that $\co(M\ba w\ba x)$ is $3$-connected.
    Thus $\{w,x\}$ is contained in a $4$-element cocircuit $C_x^*$ for each $x\in X$.
    These cocircuits intersect $X$ and $Q$ in two elements each, by orthogonality.
    Suppose that for some $x\in X$, we have $c'\not\in C_x^*$.
    Now \[E(M)-(C_x^*\cup\{x_1,x_1',c,d\}\cup X\cup Q)\] is a flat of rank at most $r(M)-4$.
    But then $\lambda(X \cup Q) \le 1$; a contradiction.
    So $c'\in C_x^*$ for each $x\in X$.
    A similar argument shows that $x_1'\in C_{x_1}$ and $x_2'\in C_{x_2}$.
    Now $X\cup Q$ is now a \spider\ with associated partition $\{X, Q\}$, as required.
\end{proof}

  We now turn to the case where \cref{temple}\ref{templequadcase} holds.
  %In particular, in this case $|X|=4$. 
  In the analysis of this case, a $3$-separator similar to a \tvamoslike\ arises.
  Although the appearance of this $3$-separator does not prevent us from guaranteeing the existence of an $N$-detachable pair, unlike when we encounter a \tvamoslike, it does require special attention.

  \begin{figure}[bhp]
  %\begin{subfigure}{0.49\textwidth}
    \centering
    \begin{tikzpicture}[rotate=90,xscale=1.1,yscale=0.55,line width=1pt]
      \tikzset{VertexStyle/.append style = {minimum height=5,minimum width=5}}
      \clip (-1.5,2) rectangle (2.7,-6);
      \node at (-0.6,-1.4) {$E(M)-P$};
      \draw (0,0) .. controls (-1.6,2) and (-2,-2) .. (0,-4);

      \draw (0.8,-1) -- (0,0) -- (1.2,1);
      \draw (2,-4) -- (1.2,-3);
      \draw (1.2,1) -- (1.2,-3);
      \draw (0,-4) -- (1.2,-3);

      \draw[white,line width=5pt] (0.8,-1) -- (2,0);
      \draw[white,line width=5pt] (0.8,-1) -- (0.8,-5);
      \draw (2,0) -- (0.8,-1) -- (0.8,-5);
      \draw (1.2,-3) -- (0.8,-5);
      \draw (1.2,1) -- (0.8,-1);
      \draw (1.2,1) -- (2,0) -- (2,-4);
      \draw (0,-4) -- (0.8,-5) -- (2,-4);

      \Vertex[x=1.2,y=1,LabelOut=true,L=$q_1$,Lpos=180]{c1}
      \Vertex[x=2,y=0,LabelOut=true,L=$s_1$,Lpos=90]{c2}
      \Vertex[x=2,y=-4,LabelOut=true,L=$s_2$,Lpos=90]{c3}
      \Vertex[x=1.2,y=-3,LabelOut=true,L=$q_2$,Lpos=135]{c4}
      \Vertex[x=0.8,y=-1,LabelOut=true,L=$p_1$,Lpos=-45]{c5}
      \Vertex[x=0.8,y=-5,LabelOut=true,L=$p_2$,Lpos=0]{c6}

      \draw (0,0) -- (0,-4);

      %\SetVertexNoLabel
      %\tikzset{VertexStyle/.append style = {shape=rectangle,fill=white}}
      %\Vertex[x=0,y=0]{a1}
      %\Vertex[x=0,y=-4]{a2}
    \end{tikzpicture}
    \caption{A \vamoslike\ in $M$.}
    \label{vamosfig}
  %\end{subfigure}
\end{figure}

  Let $M$ be a matroid with a $6$-element, rank-$4$, corank-$4$, exactly $3$-separating set $P=\{p_1,p_2,q_1,q_2,s_1,s_2\}$ such that
  \begin{enumerate}[label=\rm(\alph*)]
    \item $\{p_1,p_2,s_1,s_2\}$, $\{q_1,q_2,s_1,s_2\}$, and $\{p_1,p_2,q_1,q_2\}$ are the circuits of $M$ contained in $P$; and
    \item $\{p_1,p_2,s_1,s_2\}$, $\{q_1,q_2,s_1,s_2\}$, 
      $\{p_1,p_2,q_1,q_2,s_1\}$ and $\{p_1,p_2,q_1,q_2,s_2\}$ 
      are the cocircuits of $M$ contained in $P$. %; and
    %\item $\lc(\{p_1,p_2\}, Y)=1$ and
      %$\lc(\{q_1,q_2\}, Y)=1$, but
      %$\lc(\{s_1,s_2\}, Y)=0$.
  \end{enumerate}
  Then we say $P$ is a \emph{\vamoslike} of $M$. %with {\em associated partition} $\{P,\{q_1,q_2\}\}$
  See \cref{vamosfig} for an illustration.

\begin{lemma}
  \label{crucifix}
  Suppose $X=\{x_1,x_1',x_2,x_2'\}$ is a quad in $M \ba d$,
  there exists an element $c \in W$ such that $\{x_1,x_1',c\}$ and $\{x_2,x_2',c\}$ are triads of $M \ba d$, and
  for each $x \in X$ there is a $4$-element circuit of $M$ containing $\{x,c,d\}$.
  Then either
  \begin{enumerate}
    \item $M$ has an $N$-detachable pair,\label{crucifixi}
    \item $X\cup\{c,d\}$ is an \pspider\ of $M$,%\ with associated partition $(X,\{c,d\})$,
      \label{crucifixii}
    \item $X\cup\{c,d\}$ is a \spikelike\ of $M$, or\label{crucifixiii}
    \item $X\cup\{c,d\}$ is a \tvamoslike\ of $M$.%
      %, where $M \ba w_1 \ba w_2$ does not have an $N$-minor for $w_1 \in \{x_1,x_1'\}$ and $w_2 \in \{x_2,x_2'\}$.
      \label{crucifixiv}
  \end{enumerate}
\end{lemma}
\begin{proof}
  Note that $(X,\{c\},W-c)$ is a cyclic $3$-separation of $M \ba d$.
  In what follows, we assume that \cref{crucifixi} does not hold, and show that one of the other cases holds.

  \begin{sublemma}
    \label{crucsubby1}
    There is a circuit $\{x,x',c,d\}$ for some distinct $x, x' \in X$.
  \end{sublemma}

  \begin{slproof}
    For each $x\in X$, the set $\{x,c,d\}$ is contained in a $4$-element circuit of $M$. 
    Suppose that \cref{crucsubby1} does not hold.  Then 
    $M$ has circuits
    $\{x_1,c,d,y_1\}$, $\{x_1',c,d,y_1'\}$, $\{x_2,c,d,y_2\}$, and $\{x_2',c,d,y_2'\}$, where $\{y_1,y_1',y_2,y_2'\} \subseteq W-c$.
    If $y_1 = y_2$, say, then $\{x_1,x_2,c,d\}$ is a circuit, by circuit elimination and since neither $x_1$ nor $x_2$ is in a triangle. But then \cref{crucsubby1} holds. 
    So we may assume that the elements $y_1$, $y_1'$, $y_2$, and $y_2'$ are distinct.
    Moreover, by orthogonality, $X \cup d$ is a cocircuit of $M$.

    Suppose that $\{y_1,y_1',y_2,y_2'\}$ is an independent set; then, as $c \in \cocl(X \cup d)$ and $X \cup d$ is a cocircuit of $M$, the set $\{y_1,y_1',y_2,y_2',c,d\}$ is independent, so $r(X \cup \{c,d\}) =r(\cl(X \cup \{c,d\})) \ge 6$; a contradiction.
    So we may assume, without loss of generality, that there is a circuit~$C$ of $M$ with $\{y_1,y_1'\} \subseteq C \subseteq \{y_1,y_1',y_2,y_2'\}$.
    Also, since $\{x_1,c,d,y_1\}$ and $\{x_1',c,d,y_1'\}$ are circuits of $M$, $\{x_1,x_1',c,y_1,y_1'\}$ contains a circuit, by circuit elimination.
    Due to the cocircuit $\{x_2,x_2',c,d\}$, %this circuit cannot contain $c$, and so
    it follows that $\{x_1,x_1',y_1,y_1'\}$ is a circuit.

    Next, we show that $M \ba x_2 \ba y_1$ has an $N$-minor.
    %By \cref{contractdistincttriads}, $M \ba d/x_1' / x_2$ has an $N$-minor, and since $\{x_1,x_2'\}$ is a parallel pair in this matroid, $M \ba d \ba x_2' /x_1'$ has an $N$-minor too.  Now $\{x_2,c\}$ is a series pair in this matroid, so $M \ba x_2' /x_1' /c$ has an $N$-minor.  As $\{d,y_1'\}$ is a parallel pair in this matroid, $M \ba x_2' \ba y_1'$ has an $N$-minor.
    By \cref{contractdistincttriads}, $M \ba d/x_1 / x_2'$ has an $N$-minor, and since $\{x_1',x_2\}$ is a parallel pair in this matroid, $M \ba d \ba x_2 /x_1$ has an $N$-minor too.  Now $\{x_2',c\}$ is a series pair in this matroid, so $M \ba x_2 /x_1 /c$ has an $N$-minor.  As $\{d,y_1\}$ is a parallel pair in this matroid, $M \ba x_2 \ba y_1$ has an $N$-minor, as claimed.
    By a similar argument, $M \ba x_2' \ba y_1'$ also has an $N$-minor.

    Now, if $M \ba x_2 \ba y_1$ is $3$-connected, then \cref{crucifixi} holds; so assume otherwise.
    Suppose $\{x_2,y_1\}$ is not contained in a $4$-element cocircuit of $M$.
    Since $y_1$ is $N$-deletable, it is not in an \unfortunate\ triad, so $M \ba x_2 \ba y_1$ has no series pairs or parallel pairs.
    Thus, if $M \ba x_2 \ba y_1$ is not $3$-connected, then it has a $2$-separation $(P,Q)$ for which $P$ is fully closed, by \cref{aquickaside1}.
    Without loss of generality, we may assume that $\{x_2',c,d\} \subseteq P$. %, and thus $y_2' \in P$.
    If $x_1 \in P$, then $x_1' \in P$, and $(P \cup \{x_2,y_1\},Q)$ is a $2$-separation of $M$; a contradiction.
    So $x_1 \in Q$, and, similarly, $x_1' \in Q$.
    If $y_1' \in P$, then $x_1' \in P$; %, contradicting that $P$ is fully closed.
    a contradiction.
    So $y_1' \in Q$.
    But now, since $\{x_1,x_1',y_1,y_1'\}$ is a circuit, $(P,Q \cup y_1)$ is a $2$-separation of $M \ba x_2$; a contradiction.
    We deduce that $\{x_2,y_1\}$ is contained in a $4$-element cocircuit of $M$.
    By a similar argument, $\{x_2',y_1'\}$ is contained in a $4$-element cocircuit of $M$.

    Recall that $\{x_1,x_1',x_2,x_2'\}$, $\{x_2,c,d,y_2\}$, $\{x_1,c,d,y_1\}$, and $C$ are circuits of $M$.
    Thus, by orthogonality, the $4$-element cocircuit containing $\{x_2,y_1\}$ is %$C_1^*=
    $\{x_1,x_2,y_1,y_2\}$.
    Similarly, the cocircuit containing $\{x_2',y_1'\}$ is %$C_2^*=
    $\{x_1',x_2',y_1',y_2'\}$.
    Since $\{c,d,x_1,x_1'\}$ and $\{c,d,x_2,x_2'\}$ are also cocircuits, we have $r^*(X \cup \{c,d,y_1,y_2,y_1',y_2'\}) = 6$, so $\lambda(X \cup \{c,d,y_1,y_2,y_1',y_2'\}) \le 5 + 6 - 10 = 1$.

    Now $|E(M)| \le 11$.
    If $r(X \cup \{c,d\}) = 4$, then \cref{crucsubby1} holds unless $(X-x) \cup \{c,d\}$ is a circuit for some $x \in X$.
    But no such circuit exists, by orthogonality with either $\{x_1,x_2,y_1,y_2\}$ or $\{x_1',x_2',y_1',y_2'\}$.
    So we may assume $r(X \cup \{c,d\}) = 5$ and $|E(M)| = 11$.
    Let $\{w\} = E(M)-(X \cup \{c,d,y_1,y_2,y_1',y_2'\})$.
    Since $\lambda(X \cup \{c,d\}) = 3$ and $\{y_1,y_2,y_1',y_2',w\}$ is coindependent, $r(\{y_1,y_2,y_1',y_2',w\})=3$.
    By orthogonality with the cocircuits $\{x_1,x_2,y_1,y_2\}$ and $\{x_1',x_2',y_1',y_2'\}$, it follows that $\{y_1,y_2,w\}$ and $\{y'_1,y'_2,w\}$ are triangles.
    But since $M \ba x_2 \ba y_1$ has an $N$-minor, and $\{x_1,y_2\}$ is a series pair in this matroid, $y_2$ is $N$-contractible.
    Similarly, $y_2'$ is $N$-contractible.
    So $\{y_1,y_2,w\}$ and $\{y'_1,y'_2,w\}$ are not \unfortunate\ triangles; a contradiction.
  \end{slproof}

  Observe now that $r(X \cup \{c,d\}) \le 4$.
  If $r(X \cup \{c,d\}) = 3$, then $\lambda(X \cup \{c,d\}) < 2$; a contradiction.
  So $r(X \cup \{c,d\}) = 4$ and $\lambda(X \cup \{c,d\}) = 2$. In particular, $d \notin \cl(X)$.

  \begin{sublemma}
    \label{crucsubby2}
    Either
    \begin{enumerate}[label=\rm(\Roman*)]
      \item $\{x_1,x_2,c,d\}$ and $\{x_1',x_2',c,d\}$ are circuits of $M$, up to swapping the labels on $x_2$ and $x_2'$; or
      \item $\{x_1,x_1',c,d\}$ and $\{x_2,x_2',c,d\}$ are circuits of $M$.
    \end{enumerate}
  \end{sublemma}
  \begin{slproof}
    By \cref{crucsubby1}, we may assume, up to labels, that either $\{x_1,x_2,c,d\}$ or $\{x_1,x_1',c,d\}$ is a circuit.

    First, suppose that $\{x_1,x_2,c,d\}$ is a circuit.
    The elements $x_1'$ and $x_2'$ are also in $4$-element circuits $\{x_1',c,d,y_1\}$ and $\{x_2',c,d,y_2\}$, respectively.
    If $y_1 \in X$ or $y_2 \in X$, then, by circuit elimination with $\{x_1,x_2,c,d\}$, and since $c \notin \cl(X)$ and $X$ does not contain a triangle, \cref{crucsubby2}(I) holds.
    %If $\{x_1',x_2',c,d\}$ is also a circuit, \cref{crucsubby2} holds immediately.
    So let %$\{x_1',c,d,y_1\}$ and $\{x_2',c,d,y_2\}$ be circuits, where
    $y_1,y_2 \in W-c$.
    %Note that, by orthogonality, $X$ is not a cocircuit of $M$, so $X \cup d$ is a cocircuit. Also,
    Note that $y_1 \neq y_2$, otherwise $\{x_1',x_2',c,d\}$ is a circuit, as required, by circuit elimination.

    We claim that $M \ba x_1' \ba y_2$ has an $N$-minor.
    By \cref{contractdistincttriads}, $M \ba d/x_1/x_2'$ has an $N$-minor, and since $\{x_2,x_1'\}$ is a parallel pair in this matroid, $M \ba d\ba x_1'/x_2'$ has an $N$-minor too.
    Now $\{x_1,c\}$ is a series pair in $M \ba d \ba x_1'/x_2'$, so the matroid $M \ba x_1'/x_2'/c$ has an $N$-minor.
    As $\{d,y_2\}$ is a parallel pair in $M \ba x_1'/x_2'/c$, the matroid $M \ba x_1'\ba y_2$ has an $N$-minor as required.

    Suppose $\{x_1',y_2\}$ is contained in a $4$-element cocircuit~$C^*$ of $M$.
    By orthogonality, $C^*$ meets $\{c,d,y_1\}$ and $\{x_2',c,d\}$, so
    if neither $c$ nor $d$ is in $C^*$, then $C^* = \{x_1',x_2',y_1,y_2\}$.
    But then $r^*(X \cup \{c,d,y_1,y_2\}) = 5$, and, as $r(X \cup \{c,d,y_1,y_2\})= 4$, we have $\lambda(X \cup \{c,d,y_1,y_2\})=1$.
    So $|E(M)| = 9$, in which case $E(M)-(X \cup \{c,d\})$ is a coindependent triangle containing $\{y_1,y_2\}$.
    But since $M \ba x_1'\ba y_2$ has an $N$-minor and $\{x_2',y_1\}$ is a series pair in this matroid, $y_1$ is $N$-contractible, so it is not in an \unfortunate\ triangle; a contradiction.
    So $C^*$ contains either $c$ or $d$.

    By orthogonality with $\{x_1,x_2,c,d\}$, the cocircuit $C^*$ either contains $\{c,d\}$, or meets $\{x_1,x_2\}$.
    Now $y_2$ is in the closure and the coclosure of the $3$-separating set $X\cup \{c,d\}$, so $|E(M)| = 8$.  But then $r^*_{M \ba d}(W) = 2$; a contradiction.
    We deduce that $\{x_1',y_2\}$ is not contained in a $4$-element cocircuit of $M$.
    Since $y_2$ is $N$-deletable, it is not in an \unfortunate\ triad, so $M \ba x_1' \ba y_2$ does not have any series pairs.

    Now, if $M \ba x_1' \ba y_2$ is not $3$-connected, then it has a $2$-separation $(P,Q)$ where we may assume $\{x_1,c,d\} \subseteq P$, and $P$ is fully closed by \cref{aquickaside1}.
    Then $x_2 \in P$, due to the circuit $\{x_1,x_2,c,d\}$, and $x_2' \in P$, due to the cocircuit $\{x_2,x_2',c,d\}$.
    But then $x_1',y_2 \in \cl(P)$, so $(P \cup \{x_1',y_2\},Q)$ is a $2$-separation of $M$; a contradiction.
    Thus $M \ba x_1' \ba y_2$ is $3$-connected, so \cref{crucifixi} holds.

    \medskip
    Now we may assume that $\{x_1,x_1',c,d\}$ is a circuit,
    and $\{x_2,c,d,y\}$ and $\{x_2',c,d,y'\}$ are circuits for some distinct $y,y' \in W-c$.
    By circuit elimination, $\{x_2,x_2',c,y,y'\}$ contains a circuit, and, by orthogonality with the triad $\{x_1,x_1',c\}$ of $M \ba d$, this circuit is $\{x_2,x_2',y,y'\}$.

    We will show that $M \ba x_1' \ba y'$ is $3$-connected and has an $N$-minor, using a similar approach as in the case where $\{x_1,x_2,c,d\}$ is a circuit.
    Firstly, observe that $M \ba x_1' \ba y'$ has an $N$-minor, using a similar argument as in this other case.

    Suppose that $\{x_1',y'\}$ is contained in a $4$-element cocircuit~$C^*$ of $M$.
    We claim that $C^* \subseteq X \cup \{c,d,y'\}$.
    By orthogonality, $C^*$ meets $X-x_1'$ and $\{x_1,c,d\}$.
    Thus if $C^* \nsubseteq X \cup \{c,d,y'\}$, then $\{x_1,x_1',y\} \subseteq C^*$.  But $C^*$ also meets $\{x_2',c,d\}$, so $C^* \subseteq X \cup \{c,d,y'\}$ as claimed.
    In particular, $y' \in \cocl(X \cup \{c,d\})$.
    Observe that since $r(X \cup \{c,d,y,y'\})=4$ and $\lambda(X \cup \{c,d,y,y'\}) \ge 2$, we have $r^*(X \cup \{c,d,y,y'\}) \ge 6$.  But $r^*(X \cup \{c,d\})=4$, so $y' \notin \cocl(X \cup \{c,d\})$; a contradiction.
    So $\{x_1',y'\}$ is not contained in a $4$-element cocircuit~$C^*$ of $M$.

    Now, if $M\ba x_1' \ba y'$ is not $3$-connected, then it has a $2$-separation $(P,Q)$ where we may assume $\{x_1,c,d\} \subseteq P$, and $P$ is fully closed.
    If $x_2' \in P$, then $y' \in \cl_{M \ba x_1'}(P)$, so $(P \cup y', Q)$ is a $2$-separation of $M \ba x_1'$; a contradiction.
    So $x_2' \in Q$, and it follows that $x_2 \in Q$, due to the cocircuit $\{x_2,x_2',c,d\}$, and $y \in Q$, due to the circuit $\{x_2,c,d,y\}$.
    Now $\{x_2,x_2',y\} \subseteq Q$, so $y' \in \cl(Q)$ and $(P, Q \cup y')$ is a $2$-separation of $M\ba x_1'$; a contradiction.
    Hence $M \ba x_1' \ba y'$ is $3$-connected, so \cref{crucifixi} holds.
  \end{slproof}

  We consider two cases, depending on whether or not $d$ fully blocks $(X,W)$.
  Since $d \notin \cl(X)$, these correspond to either $d \in \cocl(X)$, or $d \notin \cocl(X)$, respectively.

  \begin{sublemma}
    \label{subby7}
    If $d \notin \cocl(X)$, then \cref{crucifixii} or \cref{crucifixiii} holds.
  \end{sublemma}

  \begin{slproof}
    Suppose $d \notin \cocl(X)$.
    As $X$ is a quad of $M \ba d$, it follows that $X$ is also a quad of $M$.
    If \cref{crucsubby2}(I) holds, then $\{x_1,x_2,c,d\}$ and $\{x_1',x_2',c,d\}$ are circuits, up to swapping the labels on $x_2$ and $x_2'$, in which case $X\cup\{c,d\}$ is an \pspider, so \cref{crucifixii} holds. 
    On the other hand,
    if \cref{crucsubby2}(II) holds, then $X\cup\{c,d\}$ is a \spikelike\ of $M$, so \cref{crucifixiii} holds.
  \end{slproof}

  We may now assume that $d \in \cocl(X)$, so $X \cup d$ is a cocircuit of $M$.
  Moreover, as $\{x_1,x_1',c,d\}$ and $\{x_2,x_2',c,d\}$ are cocircuits of $M$, cocircuit elimination, and the fact that $X$ is not a cocircuit, implies that $X \cup c$ is a cocircuit.
  %Since $d \notin \cl(W-c)$, it follows that $\lc(W-c,\{c,d\})=0$;
  %whereas, for $\{i,j\} = \{1,2\}$, the cocircuit $\{x_j,x_j',c,d\}$ implies that $\lc(W-c,\{x_i,x_i'\})=1$.
%
  Hence, as illustrated in \cref{tptws}, $X \cup \{c,d\}$ is a \tvamoslike\ of $M$ when \cref{crucsubby2}(I) holds, or a \vamoslike\ of $M$ when \cref{crucsubby2}(II) holds.
  %In order to show that $X \cup \{c,d\}$ is a \vamoslike\ of $M$, it remains to prove that $\{x_i,x_i',c,d\}$ is a circuit for each $i \in \{1,2\}$.
%
  In the former case, \cref{crucifixiv} holds.

  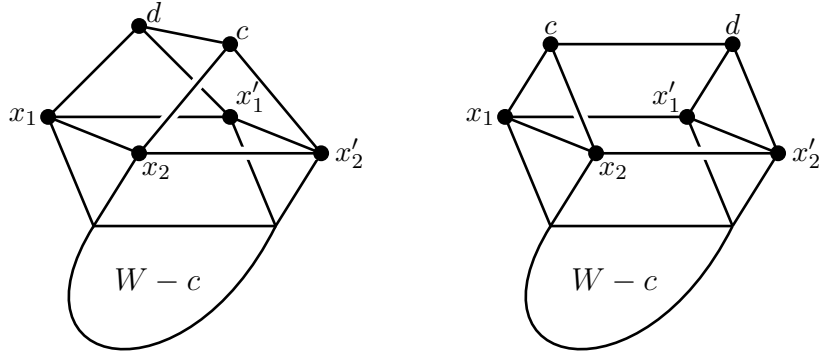
\begin{figure}[htb]
    \begin{subfigure}{0.47\textwidth}
      \centering
      \begin{tikzpicture}[rotate=90,xscale=1.21,yscale=0.605,line width=1pt]
        \tikzset{VertexStyle/.append style = {minimum height=5,minimum width=5}}
        \clip (-1.5,2) rectangle (2.7,-6);
        \node at (-0.6,-1.4) {$W-c$};
        \draw (0,0) .. controls (-1.6,2) and (-2,-2) .. (0,-4);

        \draw (0.8,-1) -- (0,0) -- (1.2,1) -- (0.8,-1);
        \draw (0.8,-5) -- (0,-4) -- (1.2,-3);

        \draw (1.2,1) -- (2.2,-1) -- (1.2,-3);
        \draw (2.2,-1) -- (2.0,-3);

        \draw (1.2,1) -- (1.2,-3);

        \draw (0,0) -- (0,-4);

        \draw[white,line width=5pt] (0.8,-1) -- (2.0,-3) -- (0.8,-5);
        \draw[white,line width=5pt] (0.8,-1) -- (0.8,-5);
        \draw (0.8,-1) -- (2.0,-3) -- (0.8,-5);
        \draw (0.8,-1) -- (0.8,-5);
        \draw (1.2,-3) -- (0.8,-5);

        \Vertex[x=2.0,y=-3,LabelOut=true,L=$c$,Lpos=45]{c5}
        \Vertex[x=2.2,y=-1,LabelOut=true,L=$d$,Lpos=45]{c6}
        \Vertex[x=0.8,y=-1,LabelOut=true,L=$x_2$,Lpos=-45]{c5}
        \Vertex[x=1.2,y=1,LabelOut=true,L=$x_1$,Lpos=180]{c1}
        \Vertex[x=0.8,y=-5,LabelOut=true,L=$x_2'$,Lpos=0]{c6}
        \Vertex[x=1.2,y=-3,LabelOut=true,L=$x_1'$,Lpos=45]{c4}
      \end{tikzpicture}
      \caption{\Tvamoslike\ of $M$.}
      \label{tptw1}
    \end{subfigure}
    \begin{subfigure}{0.47\textwidth}
      \centering
      \begin{tikzpicture}[rotate=90,xscale=1.21,yscale=0.605,line width=1pt]
        \tikzset{VertexStyle/.append style = {minimum height=5,minimum width=5}}
        \clip (-1.5,2) rectangle (2.7,-6);
        \node at (-0.6,-1.4) {$W-c$};
        \draw (0,0) .. controls (-1.6,2) and (-2,-2) .. (0,-4);

        \draw (0.8,-1) -- (0,0) -- (1.2,1);
        \draw (2,-4) -- (1.2,-3);
        \draw (1.2,1) -- (1.2,-3);
        \draw (0,-4) -- (1.2,-3);

        \draw[white,line width=5pt] (0.8,-1) -- (2,0);
        \draw[white,line width=5pt] (0.8,-1) -- (0.8,-5);
        \draw (2,0) -- (0.8,-1) -- (0.8,-5);
        \draw (1.2,-3) -- (0.8,-5);
        \draw (1.2,1) -- (0.8,-1);
        \draw (1.2,1) -- (2,0) -- (2,-4);
        \draw (0,-4) -- (0.8,-5) -- (2,-4);

        \Vertex[x=1.2,y=1,LabelOut=true,L=$x_1$,Lpos=180]{c1}
        \Vertex[x=2,y=0,LabelOut=true,L=$c$,Lpos=90]{c2}
        \Vertex[x=2,y=-4,LabelOut=true,L=$d$,Lpos=90]{c3}
        \Vertex[x=1.2,y=-3,LabelOut=true,L=$x_1'$,Lpos=135]{c4}
        \Vertex[x=0.8,y=-1,LabelOut=true,L=$x_2$,Lpos=-45]{c5}
        \Vertex[x=0.8,y=-5,LabelOut=true,L=$x_2'$,Lpos=0]{c6}

        \draw (0,0) -- (0,-4);

        %\SetVertexNoLabel
        %\tikzset{VertexStyle/.append style = {shape=rectangle,fill=white}}
        %\Vertex[x=0,y=0]{a1}
        %\Vertex[x=0,y=-4]{a2}
      \end{tikzpicture}
      \caption{\vamoslike\ of $M$.}
      \label{tptw2}
    \end{subfigure}
    \caption{The labellings of the \tvamoslike\ or \vamoslike\ when \cref{crucifix}\cref{crucifixiv} holds.}
    \label{tptws}
  \end{figure}

  So we may now assume that \cref{crucsubby2}(II) holds, and $X \cup \{c,d\}$ is a \vamoslike\ of $M$.
  We will show that $M$ has an $N$-detachable pair.

\begin{sublemma}
  \label{vamos1}
  There exists some $p \in \{x_1,x_1'\}$ and $q \in \{x_2,x_2'\}$ such that $M \ba p \ba q$ has an $N$-minor.
\end{sublemma}
\begin{slproof}
  Since $M \ba d$ has an $N$-minor with $|X \cap E(N)| \le 1$, we may assume, without loss of generality, that $X \cap E(N) = \{x_1\}$.
  Suppose that $x_1'$ is $N$-deletable in $M \ba d$.
  If either $x_2$ or $x_2'$ is $N$-deletable in $M \ba d \ba x_1'$, then \cref{vamos1} holds, so we may assume that $x_2$ and $x_2'$ are $N$-contractible in $M \ba d \ba x_1'$.
  Since $\{x_1,c\}$ is a series pair in $M \ba d \ba x_1'$, the matroid $M \ba d \ba x_1'/ c / x_2'$ has an $N$-minor, as does $M \ba x_1'/c/x_2'$.  As $\{d,x_2\}$ is a parallel pair in the latter matroid, $M \ba x_1' \ba x_2$ has an $N$-minor, as required.

  So we may assume that $x_1'$ is $N$-contractible in $M \ba d$.
  If $x_2$ is $N$-contractible in $M \ba d / x_1'$, then, as $\{x_1,x_2'\}$ is a parallel pair in $M \ba d / x_1' /x_2$, the matroid $M \ba d /x_1' \ba x_2'$ has an $N$-minor.
  Similarly, if $x_2'$ is $N$-contractible in $M \ba d / x_1'$, then $M \ba d /x_1' \ba x_2$ has an $N$-minor.
  So there is some $q \in \{x_2,x_2'\}$ such that $q$ is $N$-deletable in $M \ba d / x_1'$.
  Let $\{x_2,x_2'\}-q = \{q'\}$.
  Since $M \ba d \ba q/x_1'$ has an $N$-minor and $\{c,q'\}$ is a series pair in this matroid, $M \ba q/x_1'/c$ has an $N$-minor.  But $\{d,x_1\}$ is a parallel pair in this matroid, so $M \ba q \ba x_1$ has an $N$-minor, as required.
\end{slproof}

\begin{sublemma}
  \label{vamos2}
  If $p \in \{x_1,x_1'\}$ and $q \in \{x_2,x_2'\}$, then $M \ba p \ba q$ is $3$-connected.
\end{sublemma}
\begin{slproof}
  Pick $p'$ and $q'$ such that $\{p,p'\} = \{x_1,x_1'\}$ and $\{q,q'\} = \{x_2,x_2'\}$.
  First, observe that since $\{c,d,q,q'\}$ is a quad, and $q$ is not in a triad, $M \ba q$ is $3$-connected by \cref{r3cocirc2}.

  Suppose $\co(M \ba p \ba q)$ is not $3$-connected.
  Then $M \ba p \ba q$ has a $2$-separation $(P,Q)$ where we may assume either $P$ or $Q$ is fully closed, by \cref{aquickaside1}.  Thus, without loss of generality, we may assume that the triad $\{p',c,d\}$ is contained in $P$.
  Since $\{p,p',c,d\}$ is a circuit, $(P \cup p,Q)$ is a $2$-separation of $M \ba q$; a contradiction.
  So $\co(M \ba p \ba q)$ is $3$-connected.

  Suppose that $M \ba p \ba q$ is not $3$-connected.
  Then $M$ has a $4$-element cocircuit~$C^*$ containing $\{p,q\}$.
  Since $X \cup d$ and $X \cup c$ are cocircuits, $C^*$ is not contained in $X \cup d$ or $X \cup c$.
  So $C^*$
  meets $W'=E(M)-(X \cup \{c,d\})$.
  Suppose $C^* \cap (X \cup \{c,d\}) = \{p,q\}$.
  Then $p \in \cocl(W' \cup q)$, % \subseteq \cocl(E(M)-\{p',q',c,d\})$,
  so $p \notin \cl(\{p',q',c,d\})$.
  Since $r(X \cup \{c,d\}) = 4$, it follows that
  $\{p',q',c,d\}$ is a circuit of $M$; a contradiction.

  Now we may assume that $C^* \cap W'=\{w\}$, in which case $w \in \cocl(X \cup \{c,d\}) = \cocl(X)$. 
  So $X \cup w$ contains a cocircuit.  Since each $x \in X$ is not contained in an \unfortunate\ triad, this cocircuit contains at least three elements of $X$.  Then, by orthogonality, $X \cup w$ is a cocircuit.
  Now $(X,\{w\},\{c\},\{d\},E(M)-(X \cup \{w,c,d\}))$ is a path of $3$-separations where $w,c,d \in \cocl(X)$.
  But then $\{w,c,d\}$ is a triad in $M$; a contradiction.
\end{slproof}
Now $\{p,q\}$ is an $N$-detachable pair by \cref{vamos1,vamos2}, thus completing the proof.
\end{proof}

Putting the results of this section together we have:

\begin{theorem}
  \label{mosque}
  Let $M$ be a $3$-connected matroid with an element $d$ such that $M\ba d$ is $3$-connected.
  Let $N$ be a $3$-connected minor of $M$, % and $M \ba d$,
  where every triangle or triad of $M$ is \unfortunate, and $|E(N)| \ge 4$.
  Suppose that $M\ba d$ has a cyclic $3$-separation $(Y, \{d'\}, Z)$ with $|Y| \ge 4$, where $M\ba d \ba d'$ has an $N$-minor with $|Y \cap E(N)| \le 1$.
  Let $X$ be a minimal $3$-separating subset of $Y$ such that $|X|\geq 4$ and, for each $x \in X$,
  \begin{enumerate}[label=\rm(\alph*)]
    \item $\co(M\ba d\ba x)$ is $3$-connected, 
    \item $M\ba d/x$ is $3$-connected, and
    \item $x$ is doubly $N$-labelled in $M \ba d$.
  \end{enumerate}
  Suppose $X$ does not contain a triad of $M\ba d$.
  Then, either $M$ has an $N$-detachable pair, or there exists some $c \in \cocl_{M \ba d}(X)-X$ such that
  one of the following holds:
  \begin{enumerate}
    %\item $M$ has an $N$-detachable pair; or
    \item $X\cup\{c,d\}$ is a \spikelike;\label{out1}
    \item $X\cup\{c,d\}$ is an \pspider;\label{out2}
    \item $X\cup\{c,d\}$ is a \tvamoslike\ of $M$; or\label{out3}
    \item there exists a set $Q\subseteq E(M)-X$ with $\{c,d\}\subseteq Q$ such that $X \cup Q$ is a \spider\ with associated partition $\{X,Q\}$.\label{out4}
  \end{enumerate}
\end{theorem}
\begin{proof}
  If there is some element $y \in \cocl_{M \ba d}(Y \cup d') \cap Z$, then $(Y \cup y, \{d'\}, Z-y)$ is a cyclic $3$-separation with $|(Y \cup y) \cap E(N)| \le 1$, since $|E(N)| \ge 4$, so,
  without loss of generality, we may assume that $Y \cup d'$ is coclosed in $M \ba d$.

  First, we remark that $r(Z \cup d') \ge 3$.
  Indeed, if not, then since $d' \notin \cl(Z)$, we have $r(Z) \le 1$; a contradiction.
  Note also that
  each $x \in X$ is not contained in a triangle or triad of $M$,
  since each $x \in X$ is $N$-contractible and $N$-deletable in $M \ba d$.
  Let $W = E(M\ba d)-X$.
  Now $(X, W)$ is a $3$-separation of $M \ba d$ that satisfies the criteria of \cref{temple}.
  So one of \cref{temple}\cref{templefirstcase,templequadcase,templecontractpair,templehardcase} holds.

  %Note that if $c =d'$, then $X=Y$, and $W = Z \cup d'$.
  %On the other hand, when $c \neq d'$, either $c \in Y - X$, or $X = Y$ and $\{c,d'\} \subseteq \cocl(X) - X$.

  It is clear that if \cref{temple}\cref{templefirstcase} holds, then $M$ has an $N$-detachable pair by (c).
  If \cref{temple}\cref{templecontractpair} holds, then %again $M$ has an $N$-detachable pair by \cref{contractdistincttriads}.
  there exist elements $s$ and $t$ such that $M/s/t$ is $3$-connected, and $M/s/t$ has an $N$-minor by \cref{contractdistincttriads}; in particular, if $s = d'$, then (ii) of the lemma applies since $d'$ is in a triad meeting $X$ that does not contain $t$.
  So, again, $M$ has an $N$-detachable pair in this case.

  Suppose \cref{temple}\cref{templehardcase} holds.
  %In order to apply \cref{altar}, we require that $r(W - \{c,w\}) \ge 3$ and $r^*(W-\{c,w\}) \ge 2$.
%
  Since $c \in \cocl_{M \ba d}(X)$ and $Y \cup d'$ is coclosed, $c \in Y \cup d'$.
  Thus, if $w \in Z$, then $w \notin \cocl(Y \cup d')$, so $w \in \cl(Z-w)$.
  Hence $r(W - \{c,w\}) \ge r(Z-w) = r(Z) \ge 2$.
  Also, $r^*_{M \ba d}(W-\{c,w\}) \ge r^*_{M \ba d}(Z-w) \ge 2$.
  Now, by \cref{altar}, either $M$ has an $N$-detachable pair, or \cref{out4} holds.

  Finally, if \cref{temple}\cref{templequadcase} holds, then, by \cref{crucifix}, either $M$ has an $N$-detachable pair, or \cref{out1}, \cref{out2}, or \cref{out3} holds.
  %Finally, if \cref{temple}\cref{templespikecase} holds, then \cref{out1} holds immediately.
\end{proof}

\section{Conclusion}
Combining the two main results of this paper with the main result of \cite{paper1}, we have the following:

\begin{theorem}
  \label{usefulonep2}
  Let $M$ be a $3$-connected matroid and let $N$ be a $3$-connected minor of $M$ where $|E(N)| \ge 4$, and every triangle or triad of $M$ is \unfortunate.
  Suppose, for some $d \in E(M)$, that $M \ba d$ is $3$-connected and has a cyclic $3$-separation $(Y, \{d'\}, Z)$ with $|Y| \ge 4$, where $M \ba d \ba d'$ has an $N$-minor with $|Y \cap E(N)| \le 1$.
  Then either
  \begin{enumerate}
    \item $M$ has an $N$-detachable pair; or\label{ppc1}
    \item there is a subset $X$ of $Y$ such that
      %each $x \in X$ is doubly $N$-labelled, and,
      for some $c \in \cocl_{M \ba d}(X)-X$, one of the following holds:
      \begin{enumerate}[label=\rm(\alph*)]
        \item $X \cup \{c,d\}$ is a \twisted\ of $M$,
        \item $X \cup \{c,d\}$ is a \spikelike\ of $M$,
        %\item $X \cup \{c,d\}$ is a \vamoslike\ of $M$,
        \item $X \cup \{c,d\}$ is a \tvamoslike\ of $M$ or $M^*$,
        \item $X \cup \{c,d\}$ is an \pspider\ of $M$, or % with associated partition $(X, \{c,d\})$, or
        \item $X \cup \{a,b,c,d\}$ is a \spider\ of $M$ with associated partition $\{X,\{a,b,c,d\}\}$ for some distinct $a,b \in E(M) - (X \cup \{c,d\})$.
      \end{enumerate}\label{ppc2}
  \end{enumerate}
\end{theorem}
\begin{proof}
  By \cite[Theorem~7.4]{paper1}, if neither \cref{ppc1} nor \cref{ppc2} holds, then $Y$ contains a $3$-separating subset $X$ such that $|X| \ge 4$ and for every $x \in X$, the matroids $\co(M \ba d \ba x)$ and $M \ba d / x$ are $3$-connected, and $x$ is doubly $N$-labelled in $M \ba d$.
  Let $X$ be minimal subject to these properties.
  If $X$ contains a triad, then \cref{ppc1} holds by \cref{cathedral}.
  On the other hand, if $X$ does not contain a triad, then, by \cref{mosque}, either \cref{ppc1} or \cref{ppc2} holds.
\end{proof}

\section*{Acknowledgements}
    We thank the referees for their very careful and thorough reading of the paper, and for their corrections and helpful comments.

\bibliographystyle{abbrv}
\bibliography{lib}

\end{document}